\documentclass[aop]{imsart}   

\usepackage{hyperref} 
\usepackage{mathrsfs,bbold}
\usepackage{amsmath, amsfonts, amsthm}
\usepackage{enumerate}
\usepackage{dsfont}
\usepackage{color}
\usepackage[utf8]{inputenc}
\usepackage{stmaryrd}
\usepackage{mathabx}
\usepackage{appendix}
\usepackage{todonotes, fancyhdr}
\usepackage{upgreek}

\allowdisplaybreaks

\startlocaldefs

\renewcommand{\epsilon}{\varepsilon}

\newtheorem{thm}{Theorem}
\newtheorem{lem}[thm]{Lemma}
\newtheorem{rem}[thm]{Remark}
\newtheorem{proposition}[thm]{Proposition}
\newtheorem{definition}[thm]{Definition}
\newtheorem{Example}[thm]{{Example} }

\def\XXint#1#2#3{{\setbox0=\hbox{$#1{#2#3}{\int}$}
     \vcenter{\hbox{$#2#3$}}\kern-.5\wd0}}

\numberwithin{subsection}{section}
\numberwithin{subsubsection}{subsection}

\newcommand\restr[2]{{
  \left.\kern-\nulldelimiterspace 
  #1 
  \vphantom{\big|} 
  \right|_{#2} 
  }}

\def\WW{{\mathbb W}}
\def\EE{{\mathbb E}}
\def\RR{{\mathbb R}}

\def\LL{{\mathbb L}}

\def\PP{{\mathbb P}}
\def\P{{\mathbb P}}

\def\cL{{\mathcal L}}

\def\cC{{\mathcal C}}
\def\cP{{\mathcal P}}

\newcommand\indep{\protect\mathpalette{\protect\independenT}{\perp}}
\def\independenT#1#2{\mathrel{\rlap{$#1#2$}\mkern2mu{#1#2}}}

\def\independenT#1#2{\mathrel{\rlap{$#1#2$}\mkern2mu{#1#2}}}

\makeatletter
\DeclareFontFamily{OMX}{MnSymbolE}{}
\DeclareSymbolFont{MnLargeSymbols}{OMX}{MnSymbolE}{m}{n}
\SetSymbolFont{MnLargeSymbols}{bold}{OMX}{MnSymbolE}{b}{n}
\DeclareFontShape{OMX}{MnSymbolE}{m}{n}{
    <-6>  MnSymbolE5
   <6-7>  MnSymbolE6
   <7-8>  MnSymbolE7
   <8-9>  MnSymbolE8
   <9-10> MnSymbolE9
  <10-12> MnSymbolE10
  <12->   MnSymbolE12
}{}
\DeclareFontShape{OMX}{MnSymbolE}{b}{n}{
    <-6>  MnSymbolE-Bold5
   <6-7>  MnSymbolE-Bold6
   <7-8>  MnSymbolE-Bold7
   <8-9>  MnSymbolE-Bold8
   <9-10> MnSymbolE-Bold9
  <10-12> MnSymbolE-Bold10
  <12->   MnSymbolE-Bold12
}{}

\let\llangle\@undefined
\let\rrangle\@undefined
\DeclareMathDelimiter{\llangle}{\mathopen}%
                     {MnLargeSymbols}{'164}{MnLargeSymbols}{'164}
\DeclareMathDelimiter{\rrangle}{\mathclose}%
                     {MnLargeSymbols}{'171}{MnLargeSymbols}{'171}
\makeatother

\endlocaldefs

\numberwithin{thm}{section}
\numberwithin{equation}{section}

\begin{document}
\begin{frontmatter}

\title{Propagation of chaos for mean field rough differential equations}
\runtitle{Propagation of chaos for mean field rough equations}

 \author{\fnms{Isma\"el} \snm{Bailleul} \ead[label=e1]{ismael.bailleul@univ-rennes1.fr}\thanksref{t1}}
\thankstext{t1}{I.Bailleul thanks the Centre Henri Lebesgue ANR-11-LABX-0020-01 for its stimulating mathematical research programs, and the U.B.O. for their hospitality, part of this work was written there. Partial support from the ANR-16-CE40-0020-01.} 
 \address{Universit\'e Rennes  1
 \\
 IRMAR - UMR 6625, \\ 35000 Rennes, France \\ \printead{e1}}
 \affiliation{Universit\'e Rennes 1, CNRS, IRMAR}

 \author{\fnms{R\'emi} \snm{Catellier} \ead[label=e2]{remi.catellier@unice.fr}}
 \address{Universit\'e C\^ote d'Azur, 
 \\
 Laboratoire J.A.Dieudonn\'e, \\ 06108 Nice, France
  \\ \printead{e2}}
 \affiliation{Universit\'e C\^ote d'Azur, CNRS, LJAD}

 \author{\fnms{Fran\c{c}ois} \snm{Delarue}\corref{}\ead[label=e3]{francois.delarue@unice.fr} \thanksref{t3}}
\thankstext{t3}{F. Delarue thanks the Institut Universitaire de France.}
 \address{Universit\'e C\^ote d'Azur, \\ Laboratoire J.A.Dieudonn\'e, \\ 06108 Nice, France
  \\ \printead{e3}}
 \affiliation{Universit\'e C\^ote d'Azur, CNRS, LJAD}

\runauthor{I. Bailleul, R. Catellier, F. Delarue}

\begin{abstract}
We address propagation of chaos for large systems of rough differential equations associated with random rough differential equations of mean field type
$$
dX_t = \textrm{V}\big( X_t,\cL(X_t)\big)dt + \textrm{F}\bigl( X_t,\cL(X_t)\bigr) dW_t,
$$
where $W$ is a random rough path and $\cL(X_t)$ is the law of $X_t$. We prove propagation of chaos, and provide also an explicit optimal convergence rate. {The analysis is based upon the tools we developed in our companion paper \cite{BCD1} for solving mean field rough differential equations and in particular upon a corresponding version of the It\^o-Lyons continuity theorem. The rate of convergence is obtained by a  coupling argument developed first by Sznitman for particle systems with Brownian inputs.   }
\end{abstract}

\begin{keyword}[class=MSC]
\kwd[Primary ]{60H10}
\kwd{60G99}
\end{keyword}

\begin{keyword}
\kwd{random rough differential equations}
\kwd{particle system}
\kwd{mean field interaction}
\kwd{propagation of chaos}
\kwd{convergence rate}
\end{keyword}

\end{frontmatter}

\section{Introduction}	

The study of mean field stochastic dynamics and interacting diffusions / Markov processes finds its roots in Kac's simplified approach to kinetic theory \cite{Kac} and McKean's work \cite{McKean} on nonlinear parabolic equations. It provides the description of evolutions $(\mu_t)_{t\geq 0}$ in the space of probability measures  under the form of a pathspace random dynamics 
\begin{equation}
\label{eq:1:1}
\begin{split}
dX_t(\omega) =& \textrm{V}\bigl(X_t(\omega),\mu_t\bigr)dt + \textrm{F}\bigl(X_t(\omega), \mu_t\bigr) dW_t(\omega),   
\\
\mu_t :=& \cL(X_t),
\end{split}
\end{equation}
where $\cL(A)$ stands for the law of a random variable $A$ over a probability space $(\Omega,{\mathcal F},\P)$ containing $\omega$ and relates it to the empirical behaviour of large systems of interacting dynamics. The main emphasis of subsequent works has been on proving propagation of chaos and other limit theorems, and giving stochastic representations of solutions to nonlinear parabolic equations under more and more general settings; see for instance \cite{Sznitman,Tanaka,Gartner,DawsonGartner,DawsonVaillancourt,meleard,jourdain-meleard,BudhirajaDupuisFischer,BudhirajaWu}. Classical stochastic calculus makes sense of equation \eqref{eq:1:1} only when the process $W$ is a semi-martingale under $\PP$, for some filtration, and the integrand is predictable. However, this setting happens to be too restrictive in a number of situations, especially when the diffusivity is random. 
This prompted several authors to address equation \eqref{eq:1:1} by means of rough paths theory. Indeed, one may understand rough paths theory as a natural framework for providing probabilistic models of interacting populations, beyond the realm of It\^o calculus. Cass and Lyons \cite{CassLyons} did the first study of mean field random rough differential equations and proved the well-posed character of equation \eqref{eq:1:1}, and propagation of chaos for an associated system of interacting particles, under the crucial assumption that there is no mean field interaction in the diffusivity, \textit{i.e.} $\textrm{F}(x,\mu) = \textrm{F}(x)$, and that the drift depends linearly on the mean field interaction.  Bailleul extended partly these results in \cite{BailleulRMI} by proving well-posedness of the mean field rough differential equation \eqref{eq:1:1} in the case where the drift depends nonlinearly on the interaction term and the diffusivity is still independent of the interaction, and by proving an existence result when the diffusivity depends on the interaction. 
Another breakthrough 
came with 
our 
earlier 
arXiv deposit   
\cite{BCDArxiv}, 
in which we explained how to handle the case when 
$\textrm{F}$ truly depends on the interaction term by making a systematic use of  
Lions' approach to differential calculus on Wasserstein space. 
To make the content more accessible, we eventually decided to split \cite{BCDArxiv} into two parts: 
While 
the current work is mainly inspired from 
the second half of \cite{BCDArxiv}, 
our companion article
\cite{BCD1} corresponds to the first half of \cite{BCDArxiv}; Therein, we address the well-posedness of the mean field rough equation  
\eqref{eq:1:1} for a genuinely nonlinear F.

In fact, as explained in \cite{BCD1}, the general case may be easily reduced to the study of the simpler equation 
\begin{equation}
\label{EqRDE}
dX_t(\omega) = \textrm{F} \bigl( X_t(\omega),\cL(X_t)\bigr) dW_t(\omega), 
\end{equation}
which is precisely the version we address in this paper. To make it clear, the purpose of the present article is to prove that, under suitable assumptions, the solution of \eqref{EqRDE} coincides with the limit (in a convenient sense), as $n$ tends to $\infty$, of the $n$-particle system 
\begin{equation}
\label{eq:particle:system}
\begin{split}
X_{t}^i(\omega) = X_{0}^i(\omega) + \int_{0}^t \textrm{\rm F}\biggl(X_{s}^{{i}}(\omega),\frac1{n} \sum_{j=1}^n \delta_{X_{s}^{{j}}(\omega)}\biggr) 
dW_{s}^i(\omega), \quad t \geq 0,
\end{split}
\end{equation}
for $1 \leq i \leq n$, where $\bigl(X_{0}^i(\cdot),W^i(\cdot)\bigr)_{1 \leq i \leq n}$ is a collection of independent and identically distributed variables with the same distribution as $(X_{0}(\cdot),W(\cdot))$, the first component being regarded as a random variable with values in ${\mathbb R}^d$ and the second one as a random variable with values in the space of continuous functions. Of course, equation \eqref{eq:particle:system} must be understood as a rough differential equation driven by the signal $(W^1(\omega),\cdots,W^n(\omega))$ with $(X^1(\omega),\cdots,X^n(\omega))$ as output. As it is well-known, this requires to lift $(W^1(\omega),\cdots,W^n(\omega))$ into an enhanced rough path ${\boldsymbol W}^{(n)}(\omega)$ and henceforth to define the various iterated integrals. Asking the paths $W(\omega)$, $\omega \in \Omega$, to have a finite $p$-variation for $2 \leq p < 3$, this prompts us to assume that, instead of $\bigl((X_{0}^1(\cdot),W^1(\cdot)),\cdots,(X_{0}^n,W^n(\cdot))\bigr)$, we have in fact $n$ independent copies $(X_{0}^i(\cdot),W^i(\cdot),{\mathbb W}^i(\cdot))_{1 \leq i \leq n}$ of the triple $(X_{0}(\cdot),W(\cdot),{\mathbb W}(\cdot))$, where ${\mathbb W}(\omega)$ is the iterated integral of $W(\omega)$ and ${\mathbb W}^i(\omega)$ is the iterated integral of $W^i(\omega)$. Of course,  it is also needed to define the iterated integrals of $W^j(\omega)$ with respect to $W^i(\omega)$, for $j \not =i$. Not only we assume below that such iterated integrals do indeed exist, but we make the additional assumption that there is a measurable map ${\mathcal I}$ giving ${\mathbb W}^{i,j}(\omega)$ from $W^i(\omega)$ and $W^j(\omega)$, that is
\begin{equation}
\label{eq:intro:W:i:j}
{\mathbb W}^{i,j}(\omega) = {\mathcal I}\bigl(W^i(\omega),W^j(\omega)\bigr), \quad i \not = j.
\end{equation}
In words, \eqref{eq:intro:W:i:j} says that there exists a measurable way to construct the iterated integral of two independent copies of 
the signal in the limiting equation \eqref{EqRDE}.  Hence, \eqref{eq:intro:W:i:j} should be really regarded as an intrinsic 
property of \eqref{EqRDE} and not as a specific feature of the particle system \eqref{eq:particle:system}. 

\smallskip

More generally, it is in fact a key point in the subsequent analysis to draw a parallel between the underlying rough path used to give a meaning to 
\eqref{eq:particle:system} and the notion of \textit{extended}\footnote{In fact, the term \textit{extended} does not appear 
in \cite{BCD1}, but it is here of a convenient use to distinguish from the standard rough set-up used to solve the particle system \eqref{eq:particle:system}.} rough set-up used in \cite{BCD1} to address \eqref{EqRDE}. 
We provide a reminder of the latter notion in Section \ref{SectionRoughStructure}.
Basically, it says that, in order to solve \eqref{EqRDE}, we must not only lift, for a given $\omega \in \Omega$, the trajectory $W(\omega)$ into an enhanced rough path $(W(\omega),{\mathbb W}(\omega))$, but we must in fact lift the whole trajectory $(W(\omega),W(\cdot))$, the second component being seen as a path with values in some $\LL^q(\Omega,{\mathcal F},\PP;\RR^m)$ space, where $m$ is the dimension of the signal. 
{Then, we call \textit{extended} rough path set-up the enhancement of $(W(\omega),W(\cdot))$.}

\smallskip

The striking fact of our analysis  is then based upon an observation noticed {first by {Tanaka} in his seminal work \cite{TanakaTrick} on limit theorems for mean field type diffusions, and used crucially} by Cass and Lyons in their seminal work \cite{CassLyons}. {We refer to it as Tanaka's trick}. {It says that, for a given $\omega \in \Omega$, the particle system  
\eqref{eq:particle:system} itself may be interpreted as a mean field equation, but with respect to the empirical measure of the driving noise.
Adapted to the rough paths theory, it says that,  
for any fixed $\omega \in \Omega$, the path} 
\begin{equation*}\begin{split}
{\boldsymbol W}^{(n)}(\omega) &= \Bigl( \big(W^i(\omega)\big)_{1 \leq i \leq n}, \big({\mathbb W}^{i,j}(\omega)\big)_{1 \leq i,j \leq n} \Bigr)   
=: \Big(W^{(n)}(\omega), {\mathbb W}^{(n)}(\omega)\Big),
\end{split}\end{equation*}
which underpins the rough structure used to solve \eqref{eq:particle:system}, may be seen as an extended rough set-up on its own -- below, we just say a rough set-up) but defined on the finite probability space
\begin{equation*}
\biggl( \bigl\{1,\cdots,n\bigr\},{\mathcal P}\bigl(\bigl\{1,\cdots,n\bigr\}\bigr), \frac1n \sum_{i=1}^n \delta_{i} \biggr),
\end{equation*}
where ${\mathcal P}(\{1,\cdots,n\})$ denotes the collection of subsets of $\{1,\cdots,n\}$, instead of the former probability space $(\Omega,{\mathcal F},\PP)$. We call this set-up the \textsf{\textbf{empirical rough set-up}}, and we make its construction entirely clear in the sequel of the paper.  For sure, given the iterated integrals of the signal $(W^1(\omega),\cdots,W^n(\omega))$, the rough integral \eqref{eq:particle:system} should be interpreted in the usual sense, as given by standard Lyons' rough paths theory. In short, this requires to expand locally the integrand in \eqref{eq:particle:system}, which in turns requires to have a convenient notion of derivative with respect to the measure argument. In this regard, a crucial fact in \cite{BCD1} is to use Lions' approach \cite{Lions,LionsCardialiaguet,CarmonaDelarue_book_I} to differential calculus on the space $\cP_{2}(\RR^d)$ of probability measures on $\RR^d$ with a finite square moment, the so-called $d$-dimensional Wasserstein space, $d$ denoting here and throughout the dimension of the output in \eqref{EqRDE} ({we refer to Subsection \ref{SubsectionStability} for a longer discussion on 
the connection with other notions of derivatives on the Wasserstein space}). The core of our analysis in  Section  \ref{SectionParticleSystem} is that, whenever Wasserstein derivatives on $\cP_{2}(\RR^d)$ are projected, through empirical measures, into classical derivatives on $(\RR^d)^n$, as it is needed to differentiate the integrand in \eqref{eq:particle:system}, the resulting solution for \eqref{eq:particle:system}, as given by standard rough paths theory, coincides with the solution obtained by interpreting \eqref{eq:particle:system} as a  mean field rough equation driven by the aforementioned empirical rough set-up -- see Section \ref{SectionSolving} for reminders on solvability results for mean field rough  equations. In this way, the convergence of solutions of \eqref{eq:particle:system} to solutions of \eqref{EqRDE} as $n$ tends to $\infty$ is reduced to a form of continuity of the solutions to mean field rough differential equations with respect to the underlying rough set-up. We called the latter \textit{continuity of the It\^o-Lyons solution map}, see Theorem 5.4 of our companion work \cite{BCD1}. Our first main result, Theorem \ref{theorem:prop:of:chaos}, shows that, for a sufficiently large class of input signals, propagation of chaos is in fact a consequence of the \textit{continuity of the It\^o-Lyons solution map} for mean field rough differential equations. At this stage, it is worth mentioning that  it is precisely in the requirements of the \textit{continuity of the It\^o-Lyons map} that the structure condition \eqref{eq:intro:W:i:j} about the cross-iterated integrals comes in. In \cite{BCD1}, a \textit{rough set-up}  that satisfies \eqref{eq:intro:W:i:j} is said to be \textit{strong}.

\smallskip

While the proofs of both our first main result and the underlying \textit{continuity property of the It\^o-Lyons solution map} are mostly based on compactness arguments, our second main result is to elucidate, under slightly stronger assumptions the convergence rate in the propagation of chaos; see Theorem \ref{theorem:rate:cv} in Section \ref{SubsectionConvergenceRate}. The strategy is directly inspired from original Sznitman's coupling argument for mean field systems driven by Brownian signals, see \cite{Sznitman}. Although the proof is much more involved than in the Brownian setting, we recover the same rate of convergence: It coincides with the rate of convergence (in Wasserstein metric) of the empirical measure of an $n$-sample of (sufficiently integrable) i.i.d. variables to their common distribution. In particular, the speed decays with the dimension. 

As in \cite{BCD1}, our analysis holds for continuous rough paths whose $p$-variation, for  some $p \in [2,3)$, 
is finite and has sub-exponential tails and for which the so-called \textit{local accumulated variation} --that counts the increments of the signal of a given size over a bounded interval-- has super-exponential tails, see
 \cite[Theorem 1.1]{BCD1}. Among others, our results apply to  continuous centred Gaussian signals defined over some time interval $[0,T]$ that have independent components and whose covariance function has finite $\rho$-two dimensional variation, for some $\rho\in [1,3/2)$.

\medskip

{For the sake of completeness, it is worth adding that 
there has been a number of works recently using probabilistic tools to investigate mean field problems, or connecting probabilistic and analytical approaches. We emphasize Barbu and R\"ockner's study \cite{BarbuRockner} of the solvability of McKean-Vlasov equations with local interactions by 
means of the superposition principle, and Coghi and Gess' work \cite{CoghiGess} on stochastic nonlinear non-local Fokker-Planck equations that arise in the mean field limit of weakly interacting diffusions with a common noise. Even though it is purely analytic, Jabin and Wang's work \cite{JabinWang} on quantitative estimates of propagation of chaos for particle systems with a singular interaction kernel should be read at the light of a probabilistic intuition. The literature at the intersection of rough paths theory and mean field dynamics is sparser. Cass and Lyons launched the subject in \cite{CassLyons} by proving a well-posedness result for a mean field rough differential equation with no interaction in the diffusivity and a linear interaction in the drift, and a propagation of chaos result. The well-posedness result was generalised by Bailleul under a different set of assumptions in \cite{BailleulRMI} for mean field rough differential equations with no interaction in the diffusivity and a nonlinear interaction in the drift. The setting happens to be much simpler for equations with an additive noise, as no rough paths are needed and propagation of chaos, central limit theorem and large deviation results can be proved, as shown by Coghi, Deuschel, Friz and Maurelli in \cite{CoghiDeuschelFrizMaurelli}. Coghi and Nilssen developed in \cite{CoghiNilssen} a variant of Bailleul and Riedel's approach of rough flows \cite{BailleulRiedel} to study rough nonlocal mean field type Fokker-Planck equations {subjected to a rough common noise (meaning that, in \cite{CoghiNilssen}, 
the process $(\mu_{t})_{0 \le t \le T}$
in 
\eqref{eq:1:1} becomes random under the 
action of an additional noise that is precisely assumed to be a rough path; obviously, there is no common noise in our model)
 but to an idiosyncratic Brownian noise (to make it clear, the idiosyncratic noise should be regarded as $W$ in 
 \eqref{eq:1:1}; of course, a  substantial difference is that, in our framework, $W$ may not be a Brownian motion)}. Last, Cass, dos Reis and Salkeld \cite{CassdosReisSalkeld} used rough paths theory to investigate in depth support theorems for solutions of McKean Vlasov equations driven by a Brownian motion. }

\smallskip

The present work leaves wide open the question of refining the strong law of large numbers given by the propagation of chaos result stated in Theorem \ref{theorem:prop:of:chaos}. A central limit theorem for the fluctuations of the empirical measure of the particle system is expected to hold under reasonable conditions on the common law of the rough drivers. Large and moderate deviation results would also be most welcome. In a different direction, it would be interesting to investigate the propagation of chaos phenomenon for systems of interacting rough dynamics subject to a common noise. Very interesting things happen in the It\^o setting in relation with mean field games \cite{CarmonaDelarue,KolokoltsovTroeva}. Also, one would get a more realistic model of natural phenomena by working with systems of particles driven by non-independent noises. Individuals with close initial conditions could have drivers strongly correlated while individuals started far apart could have (almost-)independent drivers. Limit mean field dynamics are likely to be different from the results obtained here -- see \cite{CoghiFlandoli} for a result in this direction in the It\^o setting. We invite the reader to make her/his own mind about these problems.

\medskip

The paper is organized as follows. We recall in Section \ref{SectionRoughStructure} the construction of a rough set-up, as introduced in \cite{BCD1}. 
We provide in Section \ref{SectionSolving} a sketchy presentation of related solvability results for equation \eqref{EqRDE}, including a review of the main assumptions that we need on the diffusivity F. Convergence of the particle system \eqref{eq:particle:system} is established in Section \ref{SectionParticleSystem}. The convergence rate is addressed in Section \ref{SubsectionConvergenceRate}, under additional regularity assumptions on F and integrability assumptions on the signal. Proofs of some technical results are given in Appendix \ref{AppendixIntegrability} and \ref{AppendixAuxiliary}.

\medskip

\noindent \textbf{\textsf{Notations.}} We gather here a number of notations that will be used throughout the text.

\textcolor{gray}{$\bullet$} We set ${\mathcal S}_{2} := \big\{ (s,t) \in [0,\infty)^2 : s \leq t\big\}$, 
and 
$
{\mathcal S}_{2}^T := \big\{(s,t) \in [0,T]^2 : s \leq t \big\}.
$

\textcolor{gray}{$\bullet$} We denote by $(\Omega,{\mathcal F},\PP)$ an atomless Polish probability space, \textcolor{black}{${\mathcal F}$ standing for the completion of the Borel $\sigma$-field under $\PP$}, and denote by $\langle \cdot \rangle$ the expectation operator, by $\langle \cdot \rangle_{r}$, for $r \in [1,+\infty]$, the $\LL^r$-norm on $(\Omega,{\mathcal F},\PP)$ and by $\llangle \cdot \rrangle$ and $\llangle \cdot \rrangle_{r}$ the expectation operator and the $\LL^r$-norm on 
$
\big(\Omega^2,{\mathcal F}^{\otimes 2},\PP^{\otimes 2}\big).$
When $r$ is finite, $\LL^r(\Omega,{\mathcal F},\PP;\RR)$ is separable as $\Omega$ is Polish.

\textcolor{gray}{$\bullet$}
As for processes $X_{\bullet}=(X_{t})_{t \in I}$, defined on a time interval $I$, we often write $X$ for $X_{\bullet}$.

\section{From Probabilistic Rough Structures to Rough Integrals}
\label{SectionRoughStructure}

\subsection{Overview on Probabilistic Rough Structures}
\label{SubsectionRoughStructure}

We here provide a brief reminder of the content of Section 2 in \cite{BCD1}. We refer the reader to the paper for a complete
review.  Throughout the section, we work on a finite time horizon $[0,T]$, for a given $T>0$.

\smallskip

The first level of the rough path structure used to give a meaning to \eqref{EqRDE} is defined as an $\omega$-indexed pair of paths
\begin{equation}
\label{EqLevelOne}
\bigl( W_t(\omega),W_t(\cdot) \bigr)_{0\leq t\leq T},
\end{equation}
where $\big(W_t(\cdot)\big)_{0\leq t\leq T}$ 
\textcolor{black}{is a collection of $q$-integrable $\RR^m$-valued random variables on
$(\Omega,{\mathcal F},\PP)$, 
which we regard as a deterministic $\LL^q(\Omega,{\mathcal F},\PP;\RR^m)$-valued path}, for some exponent $q\geq 8$, and
$\bigl(W_{t}(\omega)\bigr)_{0 \le t \le T}$ stands for the realizations of these random variables along the outcome $\omega \in \Omega$; so the pair \eqref{EqLevelOne} takes values in $\RR^m\times \LL^q(\Omega,{\mathcal F},\PP;\RR^m)$. The second level has the form of an $\omega$-dependent two-index path with values in $\big(\RR^m \times \LL^q(\Omega,{\mathcal F},\PP;\RR^m)\big)^{\otimes 2}$ \textcolor{black}{and is} encoded in matrix form as
\begin{equation}
\label{eq:W:2}
\left(
\begin{array}{cc}
\WW_{s,t}(\omega) &\WW_{s,t}^{\indep}(\omega,\cdot)
\\
\WW_{s,t}^{\indep}(\cdot,\omega) &\WW_{s,t}^{\indep}(\cdot,\cdot)
\end{array}
\right)_{0 \leq s \leq t\leq T},
\end{equation}
where 
\begin{itemize}
   \item $\WW_{s,t}(\omega)$ is in $(\RR^{m})^{\otimes 2} \simeq \RR^{m \times m}$,  
   
   \item $\WW_{s,t}^{\indep}(\omega,\cdot)$ is in $\RR^m \otimes \LL^q\big(\Omega,{\mathcal F},\PP;\RR^m\big)\simeq \LL^q\big(\Omega,\textcolor{black}{\mathcal F},\PP;\RR^{m \times m}\big)$,   
   
   \item $\WW_{s,t}^{\indep}(\cdot, \omega)$ is in $\LL^q\big(\Omega,{\mathcal F},\PP;\RR^m\big) \otimes \RR^m \simeq \LL^q\big(\Omega,\textcolor{black}{\mathcal F},\PP;\RR^{m \times m}\big)$,
   
 \item $\WW_{s,t}^{\indep}(\cdot,\cdot)$ is in $\LL^q\bigl(\Omega^{\otimes 2},{\mathcal F}^{\otimes 2},\PP^{\otimes 2};
 \RR^{m \times m} \bigr)$, the realizations of which read in the form $\Omega^2 \ni (\omega,\omega') \mapsto \WW_{s,t}^{\indep}(\omega,\omega') \in \RR^{m \times m}$ and the two sections of which are precisely given by $\WW_{s,t}^{\indep}(\omega,\cdot) : \Omega \ni \omega' \mapsto \WW_{s,t}^{\indep}(\omega,\omega')$, and $\WW_{s,t}^{\indep}(\cdot,\omega) \ni \omega' \mapsto \WW_{s,t}^{\indep}(\omega',\omega)$, for $\omega \in \Omega$.  
\end{itemize}
A convenient form of Chen's relations is required, {for any $\omega \in \Omega$,}
\begin{equation}
\label{eq:chen}
\begin{split}
&\WW_{r,t}(\omega) = \WW_{r,s}(\omega) + \WW_{s,t}(\omega) + W_{r,s}(\omega) \otimes W_{s,t}(\omega),   \\
&\WW_{r,t}^{\indep}(\cdot,\omega) = \WW_{r,s}^{\indep}(\cdot,\omega) + \WW_{s,t}^{\indep}(\cdot,\omega) + W_{r,s}(\cdot) \otimes W_{s,t}(\omega),   \\
&\WW_{r,t}^{\indep}(\omega,\cdot) = \WW_{r,s}^{\indep}(\omega,\cdot) + \WW_{s,t}^{\indep}(\omega,\cdot) + W_{r,s}(\omega) \otimes W_{s,t}(\cdot),   \\
&\WW_{r,t}^{\indep}(\cdot,\cdot) = \WW_{r,s}^{\indep}(\cdot,\cdot) + \WW_{s,t}^{\indep}(\cdot,\cdot) + W_{r,s}(\cdot) \otimes W_{s,t}(\cdot),
\end{split}
\end{equation}
for any $0\leq r \le s \le t\leq T$, with notation 
$
f_{r,s} := f_{s} - f_{r}$, 
for a function $f$ from $[0,\infty)$ into a vector space. In \eqref{eq:chen}, we denoted  by $X(\cdot) \otimes Y(\cdot)$, 
for any two $X$ and $Y$ in $\LL^q(\Omega,{\mathcal F},\PP;\RR^m)$, the random variable
$
\bigl(\omega,\omega') \mapsto \big(X_{i}(\omega) Y_{j}(\omega')\bigr)_{1 \leq i,j \leq m}$
defined on $\Omega^2$. It is in $\LL^q\big(\Omega^2,\textcolor{black}{\mathcal F}^{\otimes 2},\PP^{\otimes 2};\RR^{m \times m}\big)$. The notation $\indep$ in $\WW^{\indep}$ is used to indicate that $\WW_{s,t}^{\indep}(\cdot,\cdot)$ should be thought of as the random variable
\begin{equation*}
(\omega,\omega') \mapsto \int_{s}^t \Bigl( W_{r}(\omega) - W_{s}(\omega) \Bigr) \otimes dW_{r}(\omega'). 
\end{equation*}
Since $\Omega^2 \ni (\omega,\omega') \mapsto (W_{t}(\omega))_{t \geq 0}$ and $\Omega^2 \ni (\omega,\omega') \mapsto (W_{t}(\omega'))_{t \geq 0}$ are independent under $\PP^{\otimes 2}$, we then understand $\WW_{s,t}^{\indep}$ as an iterated integral for two independent copies of the noise. We refer to Examples 2.1 and 2.2 in 
\cite{BCD1}. In the end, 
we denote by ${\boldsymbol W}(\omega)$ the so-called {rough set-up} specified by the $\omega$-dependent collection of maps given by \eqref{EqLevelOne} and \eqref{eq:W:2}.

\subsection{Regularity of the Rough Set-Up}
\label{SubsectionRegularityRoughSetup}

Following \cite{BCD1}, we use the notion of $p$-variation to handle the regularity of the various trajectories in hand. 
\textcolor{black}{Throughout, the exponent $p$ is taken} in the interval $[2,3)$. For a continuous function ${\mathbb G}$ from the simplex ${\mathcal S}_{2}^T$ into some $\RR^\ell$, we set, for any $p'\geq 1$, 
\begin{equation*}
\begin{split}
&\| {\mathbb G} \|_{[0,T],p'-\textrm{\rm v}}^{p'} := \sup_{0 = t_{0}<t_{1} \cdots < t_{n}=T } \, \sum_{i=1}^n \vert {\mathbb G}_{t_{i-1},t_{i}}\vert^{p'},
\end{split}
\end{equation*}
and define for any function $g$ from $[0,T]$ into $\RR^\ell$, 
$
\| g \|_{[0,T],p-\textrm{\rm v}}^{p} := \| {\mathbb G} \|_{[0,T],p-\textrm{\rm v}}^{p}$
as the $p$-variation \textcolor{black}{semi-}norm of its associated two index function ${\mathbb G}_{s,t} := g_t - g_s$. Similarly, for a random variable ${\mathbb G}(\cdot)$ on $\Omega$ with values in ${\mathcal C}({\mathcal S}_{2}^T;\RR^{\ell})$, and $p'\geq 1$, we define its $p'$-variation in $\LL^q$ as 
\begin{equation}
\label{eq:q:p-var}
\begin{split}
&\langle {\mathbb G}(\cdot) \rangle_{q; [0,T],p'-\textrm{\rm v}}^{p'} := \sup_{0 = t_{0}<t_{1} \cdots < t_{n}=T } \, \sum_{i=1}^n \big\langle  {\mathbb G}_{t_{i-1},t_{i}}(\cdot) \big\rangle_{q}^{p'},
\end{split}
\end{equation}
and define for a random variable $G(\cdot)$ on $\Omega$, with values in ${\mathcal C}([0,T];\RR^{\ell})$,
$$
\big\langle G(\cdot) \big\rangle_{q ; [0,T],p'-\textrm{\rm v}}^{p'} := \big\langle {\mathbb G}(\cdot) \big\rangle_{q; [0,T],p'-\textrm{\rm v}}^{p'},
$$
as the $p'$-variation semi-norm in $\LL^q$ of its associated two-index function ${\mathcal S}_{2}^T \ni (s,t) \mapsto {\mathbb G}_{s,t}(\cdot) = G_{t}(\cdot) - G_{s}(\cdot)$. Lastly, for a random variable ${\mathbb G}(\cdot,\cdot)$ from $(\Omega^2,{\mathcal F}^{\otimes 2})$ into ${\mathcal C}({\mathcal S}_{2}^T;\RR^{\ell})$, we set 
\begin{equation}
\label{eq:q:q:p-var}
\begin{split}
&\llangle {\mathbb G}(\cdot,\cdot) \rrangle_{q ; [0,T],{p'}-\textrm{\rm v}}^{{p'}} := \sup_{0 = t_{0}<t_{1} \cdots < t_{n}=T } \sum_{i=1}^n \left\llangle {\mathbb G}_{t_{i-1},t_{i}}(\cdot,\cdot) \right\rrangle_{q}^{{p'}}.
\end{split}
\end{equation}
Given these definitions, we require from the rough set-up ${\boldsymbol W}$ that \vspace{0.1cm}

\begin{itemize}
   \item For any $\omega \in \Omega$, the path $W(\omega)$ is in the space ${\mathcal C}([0,T];\RR^m)$, and the map $W : \Omega \ni \omega \mapsto W(\omega) \in {\mathcal C}([0,T];\RR^m)$ is Borel-measurable and $q$-integrable. 
  \vspace{0.1cm}
   
   \item For any $\omega \in \Omega$, the two-index path $\WW(\omega)$ is in ${\mathcal C}({\mathcal S}_{2}^T;\RR^{m \times m})$, and the map $\WW : \Omega \ni \omega \mapsto \WW(\omega) \in {\mathcal C}({\mathcal S}_{2}^T;\RR^{m \times m})$ is Borel-measurable and $q$-integrable.   
   \vspace{0.1cm}

   \item For any $(\omega,\omega') \in \Omega^2$, the two-index path $\WW^{\indep}(\omega,\omega')$ is an element of ${\mathcal C}({\mathcal S}_{2}^T;\RR^{m \times m})$, and the map $\WW^{\indep} : \Omega^2 \ni (\omega,\omega') \mapsto \WW^{\indep}(\omega,\omega') \in {\mathcal C}({\mathcal S}_{2}^T;\RR^{m \times m})$ is Borel-measurable and $q$-integrable. 
\end{itemize}
Moreover, we may set, for some fixed $p \in [2,3)$ and for all $0\leq s\leq t \leq T$ and $\omega\in\Omega$,
\begin{equation}
\label{eq:v}
\begin{split}
v(s,t,\omega) &:= \big\| W(\omega) \big\|_{[s,t],p-\textrm{\rm v}}^p + \big\langle W(\cdot) \big\rangle_{q ; [s,t],p-\textrm{\rm v}}^p   \\
&\hspace{15pt} + \big\| \WW(\omega) \big\|_{[s,t],p/2-\textrm{\rm v}}^{p/2} + \big\langle \WW^{\indep}(\omega,\cdot) \big\rangle_{q ; [s,t],p/2-\textrm{\rm v}}^{p/2}   \\
&\hspace{15pt} + \big\langle \WW^{\indep}(\cdot,\omega) \big\rangle_{q ; [s,t],p/2-\textrm{\rm v}}^{p/2} + \big\llangle \WW^{\indep}(\cdot,\cdot) \big\rrangle_{q ; [s,t],p/2-\textrm{\rm v}}^{p/2},
\end{split}
\end{equation}
and we assume that, for any positive finite time $T$ and any $\omega \in \Omega$, the quantity $v(0,T,\omega)$ is finite. Importantly,
$\omega \mapsto (v(s,t,\omega))_{(s,t)\in {\mathcal S}_{2}^T}$ is a random variable with values in ${\mathcal C}({\mathcal S}_{2}^T;\RR_{+})$ and is super-additive, namely, for any $0 \leq r \leq s \leq t \leq T$, and $\omega\in\Omega$,
\begin{equation*}
v(r,t,\omega) \geq v(r,s,\omega) + v(s,t,\omega). 
\end{equation*}
 We then assume 
$\langle v(0,T, \cdot) \rangle_{q} < \infty$, which implies, by Lebesgue's dominated convergence theorem, that the function 
$
{\mathcal S}_{2}^T \ni (s,t) \mapsto \langle v(s,t,\cdot) \rangle_q
$ 
is continuous. We assume that it is of bounded variation on $[0,T]$, \textit{i.e.} 
\begin{equation}
\label{eq:v:1-var}
\langle v(\cdot) \rangle_{q;[s,t],1-\textrm{\rm v}} := \sup_{0 \leq t_{1} < \cdots < t_{K} \leq T} \sum_{i=1}^{K} \langle v(t_{i-1},t_{i},\cdot) \rangle_{q} < \infty. 
\end{equation}
We then call a control any family
of random variables $(\omega \mapsto w(s,t,\omega))_{(s,t) \in {\mathcal S}_{2}^T}$ 
 that 
is jointly continuous in $(s,t)$ and that satisfies, 
\begin{equation}
\label{eq:w:s:t:omega:ineq}
w (s,t,\omega) \geq v(s,t,\omega) + \langle v(\cdot) \rangle_{q;[s,t],1-\textrm{\rm v}},
\end{equation}
together with
\begin{equation}
\begin{split}
\label{eq:useful:inequality:wT}
&\langle w(s,t,\cdot) \rangle_{q} \leq 2 \, w(s,t,\omega),   \\
&w(r,t,\omega) \geq w(r,s,\omega) + w(s,t,\omega), \quad r \leq s \leq t.
\end{split}
\end{equation}
A typical choice to get 
\eqref{eq:w:s:t:omega:ineq}
and 
\eqref{eq:useful:inequality:wT} is to choose
\begin{equation}
\label{eq:w:s:t:omega}
w (s,t,\omega) := v(s,t,\omega) + \langle v(\cdot) \rangle_{q;[s,t],1-\textrm{\rm v}}.
\end{equation}

\subsection{Controlled Trajectories}
\label{SubsectionControlledTraj}

With a rough set-up at hands on a given finite time interval $[0,T]$, we define an associated notion of controlled path and rough integral in the spirit of Gubinelli \cite{Gubinelli}. Again, we refer to \cite{BCD1} for details, see Definition 3.1 therein.

\begin{definition}
\label{definition:omega:controlled:trajectory}
An $\omega$-dependent continuous $\RR^d$-valued path $(X_{t}(\omega))_{0 \le t \le T}$ is called an {{$\omega$-controlled path}} on $[0,T]$ if its increments can be decomposed as 
\begin{equation}
\label{eq:omega:controlled:eq}
X_{s,t}(\omega) = \delta_{x} X_{s}(\omega) W_{s,t}(\omega) + \EE \bigl[ \delta_{\mu} X_{s}(\omega,\cdot) W_{s,t}(\cdot) \bigr] + R^X_{s,t}(\omega),
\end{equation}
where 
$\big(\delta_{x} X_{t}(\omega)\big)_{0 \leq t \leq T}$ belongs to $\cC\big([0,T];\RR^{d \times m}\big)$ and 
$\big(\delta_{\mu} X_{t}(\omega,\cdot)\big)_{0 \leq t \leq T}$ to  ${\mathcal C}\big([0,T];\LL^{4/3}(\Omega,{\mathcal F},\PP;\RR^{d \times m})\big)$,
$\big(R_{s,t}^X(\omega)\big)_{s,t \in {\mathcal S}_{2}^T}$ is in the space $\cC({\mathcal S}_{2}^T;\RR^d)$,
and 
\begin{equation*}
\begin{split}
\vvvert X(\omega) \vvvert_{\star,[0,T],w,p} 
&:= \vert X_{0}(\omega) \vert + \big\vert \delta_{x} X_{0}(\omega) \big\vert + \big\langle  \delta_{\mu} X_{0}(\omega,\cdot) \big\rangle_{4/3} 
\\
&\hspace{15pt}+ \vvvert X(\omega) \vvvert_{[0,T],w,p} < \infty,
\end{split}
\end{equation*}
where 
\begin{equation*}
\begin{split}
\vvvert X(\omega) \vvvert_{[0,T],w,p} &:= \| X(\omega) \|_{[0,T],w,p} + \| \delta_{x} X(\omega) \|_{[0,T],w,p} 
\\
&+ \big\langle  \delta_{\mu} X(\omega,\cdot)  \big\rangle_{[0,T],w,p,4/3}   
+ \| R^X(\omega) \|_{[0,T],w,p/2},
\end{split}
\end{equation*}
with
\begin{equation*}
\begin{split}
&\| X(\omega) \|_{[0,T],w,p} := \sup_{{\emptyset \not = (s,t)} \subset [0,T]} \frac{\big\vert   X_{{s,t}}(\omega) \big\vert}{w(s,t,\omega)^{1/p}},   \quad {\textrm{\rm and similarly for} \ {\delta}_{x} X}
\\
&\big\langle \delta_{\mu} X(\omega,\cdot) \big\rangle_{[0,T],w,p,4/3} := \sup_{\emptyset \not = (s,t) \subset [0,T]} \frac{\big\langle  \delta_{\mu}X_{{s,t}}(\omega,\cdot)  \big\rangle_{4/3}}{w(s,t,\omega)^{1/p}},
\\
&\| R^X(\omega) \|_{[0,T],w,p/2} := \sup_{\emptyset \not = (s,t) \subset [0,T]} \frac{\big\vert R_{{s,t}}^X(\omega) \big\vert}{w(s,t,\omega)^{2/p}}.
\end{split}
\end{equation*}
We call \textcolor{black}{$\delta_x X(\omega)$} and \textcolor{black}{$\delta_\mu X(\omega,\cdot)$} in  \eqref{eq:omega:controlled:eq} the  {{derivatives of}} \textcolor{black}{$X(\omega)$}.
\end{definition}

We then define the notion of random controlled trajectory, which consists of a collection of $\omega$-controlled trajectories indexed by the elements of $\Omega$.

\begin{definition}
\label{definition:random:controlled:trajectory}
{A family of $\omega$-controlled paths $(X(\omega))_{\omega \in \Omega}$ such that 
$X$, $\delta_{x} X$,
$\delta_\mu X$
and 
$R^X$
are measurable from $\Omega$ into  
$\cC\big([0,T];\RR^d\big)$,
$\cC\big([0,T];\RR^{d \times m}\big)$,
$\cC\big([0,T];\LL^{4/3}(\Omega,{\mathcal F},\PP;\RR^{d \times m})\big)$
and 
$\cC\big({\mathcal S}_{2}^T;\RR^d\big)$}, and \textcolor{black}{satisfy} 
\begin{equation}
\label{eq:Lq:vvvert}
\big\langle X_{0}(\cdot) \big\rangle_{2} + \bigl\langle \vvvert X(\cdot) \vvvert_{[0,T],w,p}\bigr\rangle_{8} < \infty
\end{equation}
is called a {random controlled path} on $[0,T]$.
\end{definition}
It is proven in \cite[Lemma 3.3]{BCD1} that a random controlled trajectory induces a continuous path $t \mapsto X_{t}(\cdot)$ from $[0,T]$ to $\LL^2(\Omega,{\mathcal F},\PP;\RR^d)$.
 {On another matter, the reader may observe that we do not require any integrability property on the initial conditions $\delta_{x} X_{0}$ and $\delta_{\mu} X_{0}$ of the two derivative processes. In fact, it must be understood that, when dealing with solutions to 
\eqref{EqRDE}, both $\delta_{x} X_{0}$ and $\delta_{\mu} X_{0}$ are automatically prescribed: $\delta_{x} X_{0}(\omega)=
\text{F}(X_{0}(\omega),{\mathcal L}(X_{0}))$ and $\delta_{\mu} X_{0}(\omega,\cdot)=0$, 
see the forthcoming 
Theorem 
\ref{theorem:integral}
for more details. Simple bounds for both 
$\delta_{x} X_{0}$ and $\delta_{\mu} X_{0}$ then easily follow.
Lastly, as it is explained in our companion paper \cite{BCD1}, see Definition 3.1 therein, 
the values of $4/3$ and $8$ in Definition \ref{definition:random:controlled:trajectory} are somewhat arbitrary. In particular, the analysis could be managed with other exponents provided that a certain trade-off between the two of them (here $4/3$ and $8$) still holds. To make it clear, 
$4/3$ should be here regarded as the conjugate exponent of $4$; and the reader may easily guess that $8$ shows up in the computations when squaring some fourth moments. }

\subsection{Rough Integral}
\label{SubsectionIntegral}

As for the construction, of the rough integral, we recall the following statement from \cite[Theorem 3.4]{BCD1}.

\begin{thm}
\label{theorem:integral}
\textcolor{black}{There exists a universal constant $c_0$ and}, for any $\omega \in \Omega$, there exists a continuous linear map 
\begin{equation*}
\bigl(X_{t}(\omega)\bigr)_{0 \leq t \leq T} \mapsto \biggl( \int_{s}^t X_{s,u}(\omega) \otimes d {\boldsymbol W}_{u}(\omega) \biggr)_{\textcolor{black}{(s,t)} \in {\mathcal S}_{2}^T}
\end{equation*}
from the space of $\omega$-controlled trajectories equipped with the norm $\vvvert \cdot \vvvert_{\star,[0,T],p}$, onto the space of continuous functions from ${\mathcal S}_{2}^T$ into $\RR^d \otimes \RR^m$ with finite norm $\| \cdot \|_{[0,T],w,p/2}$, with $w$ in the latter norm being evaluated along the realization $\omega$, that satisfies for any $0\leq r\leq s\leq t\leq T$ the identity
\begin{equation*}
\begin{split}
&\int_{r}^t X_{r,u}(\omega) \otimes d {\boldsymbol W}_{u}(\omega) 
\\
&\hspace{15pt}= \int_{r}^s X_{r,u}(\omega) \otimes d {\boldsymbol W}_{u}(\omega) + \int_{s}^t X_{s,u}(\omega) \otimes d {\boldsymbol W}_{u}(\omega) + X_{r,s}(\omega)
\otimes W_{s,t}(\omega),
\end{split}
\end{equation*}
together with the estimate
\begin{align}
&\biggl\vert \int_{s}^t X_{s,u}(\omega) \otimes d {\boldsymbol W}_{u}(\omega) 
 -  \delta_{x} X_{s}(\omega) {\mathbb W}_{s,t}(\omega) - {\mathbb E} \bigl[ \delta_{\mu} X_{s}(\omega,\cdot) {\mathbb W}_{s,t}^{\indep}(\cdot,\omega) \bigr]  \biggr\vert \nonumber 
  \\
&\hspace{5pt}\leq c_0\, \vvvert X(\omega) \vvvert_{\textcolor{black}{[0,T],w,p}}\, w(s,t,\omega)^{3/p}. 
\label{eq:remainder:integral}
\end{align}
\end{thm}
Above, $\delta_{x} X_{s}(\omega) \, {\mathbb W}_{s,t}(\omega)$ is the product of a $d \times m$ matrix and an $m \times m$ matrix, so it gives back a $d \times m$ matrix, with components 
$
\bigl( \delta_{x} X_{s}(\omega) {\mathbb W}_{s,t}(\omega) \bigr)_{i,j} = \sum_{k=1}^m \bigl( \delta_{x} X_{s}^{i}(\omega) \bigr)_{k} \bigl( {\mathbb W}_{s,t}(\omega) \bigr)_{k,j}$, 
for $i \in \{1,\cdots,d\}$ and $j \in \{1,\cdots,m\}$, and similarly 
for
$
{\mathbb E}[\delta_{\mu} X_{s}(\omega,\cdot) {\mathbb W}_{s,t}^{\indep}(\cdot,\omega)]
$.
 As usual, the above construction allows us to define an additive process setting
\begin{equation*}
\int_{s}^t X_{u}(\omega) \otimes d {\boldsymbol W}_{u}(\omega) := \int_{s}^t X_{s,u}(\omega) \otimes d {\boldsymbol W}_{u}(\omega) + X_{s}(\omega) \otimes \textcolor{black}{W_{s,t}}(\omega), 
\end{equation*}
for $0\leq s \leq t\leq T$. 
We can thus consider the integral process $\big(\int_{0}^t X_{s}(\omega) \otimes d{\boldsymbol W}_{s}(\omega)\big)_{0 \leq t \leq T}$ as an $\omega$-controlled trajectory with values in $\RR^{d \times m}$, with 
\begin{equation*}
\biggl( \delta_{x} \biggl[ \int_{0}^{\cdot} X_{s}(\omega) \otimes d{\boldsymbol W}_{s}(\omega) \biggr]_{t} \biggr)_{(i,j),k}= \textcolor{black}{\bigl( X_{t}(\omega)
\bigr)_{i} \delta_{j,k}},
\end{equation*}
for $i \in \{1,\cdots,d\}$ and $j,k \in \{1,\cdots,m\}$, where $\delta_{j,k}$ stands for the usual Kronecker symbol, and
with null $\mu$-derivative.

When the trajectory $X(\omega)$ takes in values in $\RR^{d} \otimes \RR^m$ rather than $\RR^d$, the integral 
$\int_{0}^t X_{s}(\omega) \otimes d {\boldsymbol W}_{s}(\omega)$ belongs to 
$\RR^{d} \otimes \RR^m \otimes \RR^m$. 
We then set for $i \in \{1,\cdots,d\}$
\begin{equation*}
\biggl( \int_{0}^t X_{s}(\omega) d {\boldsymbol W}_{s}(\omega) \biggr)_{i} := \sum_{j=1}^m \biggl( \int_{0}^t X_{s}(\omega)\otimes d {\boldsymbol W}_{s}(\omega) \biggr)_{i,j,j},
\end{equation*}
and consider $\int_{0}^t X_{s}(\omega) d {\boldsymbol W}_{s}(\omega)$ as an element of $\RR^d$.

\bigskip

\subsection{Stability of Controlled Paths under Nonlinear Maps}
\label{SubsectionStability}

A key fact in \cite{BCD1} is to use regularity properties of functions defined on Wasserstein space through a lifting procedure to an $L^2$ space standing above the probability space. We refer the reader to Lions' lectures \cite{Lions}, to the lecture notes \cite{LionsCardialiaguet} of Cardaliaguet or to Carmona and Delarue's monograph \cite[Chapter 5]{CarmonaDelarue_book_I} for basics on the subject. {Among others, the latter addresses the connection with other forms of derivatives on the Wasserstein space, such as flat or intrinsic derivatives.
Flat derivatives are
 inherited from the linear structure of the space of (hence non-normalized)
signed measures and are also called, depending on the definition, extrinsic or convexity extrinsic derivatives, see for instance  \cite{ORS,RenWang}. 
Intrinsic derivatives
were 
described by \cite{AmbrosioGigliSavare,Lott,Villani} after the pioneering works 
\cite{JKO,Otto} on gradient flows on the space of probability measures; 
a fine comparison with Lions' approach is carried out in  the recent work 
\cite{GangboTudorascu}. The typical example for a lift in $L^2$ is to consider an ordinary differential equation 
driven by some vector field $b$ starting from 
some random variable; it provides a lift of the solution to a Fokker-Planck equation 
together with a systematic way to identify derivatives of functionals of a measure argument in the direction $b$.
To the best of our knowledge, this latter idea goes back to the earlier works 
\cite{AKR1,AKR2}.}
\smallskip

\textcolor{gray}{$\bullet$} Recall $(\Omega,\mathcal{F},{\mathbb P})$ stands for an atomless probability space, with $\Omega$ a Polish space and $\mathcal{F}$ its Borel $\sigma$-algebra. Fix a finite dimensional space $E=\RR^k$ and denote,  {for $r \geq 1$}, by $\LL^{ {r}}:\textcolor{black}{=\LL^{ {r}}(\Omega,{\mathcal F},\PP;E)}$ the space of $E$-valued random variables on $\Omega$ with finite  {$r$} moment. We equip the space $\mathcal{P}_{{r}}(E) := \big\{\mathcal{L}(Z)\,;\,Z\in \LL^{{r}}\big\}$ with the ${r}$-Wasserstein distance 
\begin{equation}
\label{eq:wasserstein:distance}
{\mathbf d}_{{r}}(\mu_1,\mu_2) := \inf \Big\{\|Z_1-Z_2\|_{{r}}\,;\, \mathcal{L}(Z_1) = \mu_1,\, \mathcal{L}(Z_2)=\mu_2\Big\}.
\end{equation}
{When $r=2$}, an $\RR^k$-valued function $u$ defined on $\mathcal{P}_2(E)$ is canonically extended to $\LL^2$ by setting, for any $Z\in \LL^2$,
$
U(Z) := u\big(\mathcal{L}(Z)\big).$
\smallskip  

\textcolor{gray}{$\bullet$} The function $u$ is then said to be differentiable at $\mu\in\mathcal{P}_2(E)$ if its canonical lift is Fr\'echet differentiable at some point $Z$ such that $\mathcal{L}(Z)=\mu$; \textcolor{black}{we} denote by $\nabla_ZU\in (\LL^2)^k$ the gradient of $U$ at $Z$. The function $U$ is then differentiable at any other point $Z'\in \LL^2$ such that $\mathcal{L}(Z')=\mu$, and the laws of $\nabla_ZU$ and $\nabla_{Z'}U$ are equal, for any such $Z'$. 
\smallskip  
   
 \textcolor{gray}{$\bullet$} The function $u$ is said to be of class $C^1$ if its canonical lift is of class $C^1$. If $u$ is of class $C^1$ on ${\mathcal P}_{2}(E)$, then $\nabla_ZU$ is $\sigma(Z)$-measurable and given by an $\mathcal{L}(Z)$-dependent function $Du$ from $E$ to 	\textcolor{black}{$E^k$} such that 
\begin{equation}
\label{EqdefinitionDu}
\nabla_ZU = (Du)(Z).
\end{equation}
In order to emphasize the fact that $Du$ depends upon 
${\mathcal L}(Z)$, we shall write $D_{\mu} u\big({\mathcal L}(Z)\big)(\cdot)$ instead of $Du(\cdot)$. Importantly, this representation is independent of the choice of the probability space $(\Omega,{\mathcal F},{\mathbb P})$ and can be easily transported from one probability space to another. \vskip 1pt

Throughout the paper, we regard the function F in \eqref{EqRDE} as a map from $\RR^d\times \LL^2(\Omega,{\mathcal F},\PP;\RR^d)$ into the space $L(\RR^m,\RR^d) \cong \RR^d \otimes \RR^m$ of linear mappings from $\RR^m$ to $\RR^d$. Intuitively, we identify the coefficient driving equation \eqref{EqRDE} with its lift $\widehat {\textrm{\rm F}}$. Following \cite[Subsection 3.3]{BCD1}, we require $\text{F}$ to satisfy the following regularity assumptions.

\medskip

\noindent \textbf{\textsf{Regularity assumptions  1 -- }} \textit{Assume that \emph{F} is continuously differentiable in the joint variable $(x,Z)$, that 
$\partial_{x} F$ is also continuously differentiable in $(x,Z)$ and that there is some positive finite constant $\Lambda$ such that $\vert \textrm{\emph{F}}(x,\mu) \vert $, $\vert \partial_{x} \textrm{\emph{F}}(x,\mu) \vert$, $\vert \partial_{x}^2 \textrm{\emph{F}}(x,\mu) \vert$, $\| \nabla_{Z}\textrm{\emph{F}}(x,Z) \|_{2}$ and $\|\partial_x \nabla_{Z}\textrm{\emph{F}}(x,Z)\|_2$ are bounded by $\Lambda$, for any $x \in \RR^d$, $\mu \in {\mathcal P}_{2}(\RR^d)$ and $Z \in \LL^2(\Omega,{\mathcal F},\PP;\RR^d)$. Assume moreover that, for any $x \in \RR^d$, the mapping $Z \mapsto \nabla_{\textcolor{black}{Z}}\textrm{\emph{F}}(x,Z)$ is \textcolor{black}{a $\Lambda$}-Lipschitz function of $\textcolor{black}{Z} \in \LL^2(\Omega,{\mathcal F},\PP;\RR^d)$.} 

\medskip

We recall below that, for an $\omega$-controlled path $X(\omega)$  and for an $\RR^d$-valued random controlled path $Y(\cdot)$, $\textrm{F}(X(\omega),Y(\cdot)):=\bigl( \textrm{F}(X_t(\omega),Y_t(\cdot))\big)_{0\leq t\leq T}$ may be also expanded in the form of an $\omega$-controlled trajectory. As explained in \cite[(3.8)]{BCD1}, it suffices for our purpose to provide the form of the expansion when $\delta_\mu X(\omega) \equiv 0$ and $\delta_\mu Y(\cdot) \equiv 0$.

\begin{proposition}
\label{prop:chaining}
Let $X(\omega)$ be an $\omega$-controlled path and $Y(\cdot)$ be an $\RR^d$-valued random controlled path. Assume that
$\delta_\mu X(\omega) \equiv 0$ and $\delta_\mu Y(\cdot) \equiv 0$ and that 
 {$\sup_{0 \le t \le T} \big( \vert \delta_{x} X_{t}(\omega) \vert \vee \langle \delta_{x} Y_{t}(\cdot)\rangle_{\infty}\big) < \infty$}. Then, $\textrm{\emph{F}}\big(X(\omega),Y(\cdot)\big)$ is an $\omega$-controlled path with 
\begin{equation*}
\delta_{x} \Bigl( \textrm{\emph{F}}\bigl(X(\omega),Y(\cdot)\bigr) \Bigr)_{t} = \partial_x\textrm{\emph{F}}\bigl(X_{t}(\omega),Y_{t}(\cdot)\bigr) \, \delta_{x} X_{t}(\omega),
\end{equation*}
which is understood as $\bigl( \sum_{\ell=1}^d \partial_{x_{\ell}}
\textrm{\emph{F}}^{i,j}\bigl(X_{t}(\omega),Y_{t}(\cdot)\bigr)
\bigl(\delta_{x} X_{t}^{\ell}(\omega) \bigr)_{k} \bigr)_{i,j,k}$, with $i,k \in \{1,\cdots,d\}$
and $j \in \{1,\cdots,m\}$, and (with a similar interpretation for the product)
\begin{equation*}
\begin{split}
\delta_{\mu} \Bigl(\textrm{\emph{F}}\bigl(X(\omega),Y(\cdot)\bigr) \Bigr)_{t} &=  D_{\mu} \textrm{\rm F}\bigl(X_{t}(\omega),\cL(Y_{t})\bigr)\bigl(X_{t}(\cdot)\bigr) \delta_{x} Y_{t}(\cdot).
\end{split}
\end{equation*}
\end{proposition}

\subsection{Local Accumulation}

In order to proceed with the analysis of \eqref{eq:particle:system}, we make use of the notion of local accumulation. Following
\cite{BCD1}, we define it as follows. Given a  {nondecreasing}\footnote{In the sense that 
$\varpi(a,b) \geq \varpi(a',b')$ if $(a',b') \subset (a,b)$.} continuous positive valued function $\varpi$ on ${\mathcal S}_{2}$, a non-negative parameter $s$ and a positive threshold $\alpha$, we define inductively a sequence of times setting $\tau_{0}(s,\alpha) := s$, and 
\begin{equation}
\label{eq:stopping:times}
\tau_{n+1}^{\varpi}(s,\alpha) := \inf\Bigl\{ u \geq \tau_{n}^{\varpi}(s,\alpha) \, : \, \varpi\bigl(\tau_{n}^{\varpi}(s,\alpha),u\bigr) \geq \alpha \Bigr\}, 
\end{equation}
with the understanding that $\inf \emptyset := + \infty$. For $t \geq s$, set
\begin{equation}
\label{eq:N:s:t:omega}
N_{\varpi}\bigl([s,t],\alpha\bigr) := \sup \Bigl\{ n \in {\mathbb N} \ : \ \tau_{n}^{\varpi}(s,\alpha) \leq t \Bigr\}. 
\end{equation}
We call $N_{\varpi}$ the local accumulation of $\varpi$ (of size $\alpha$ if we specify the value of the threshold): 
{$N_{\varpi}([s,t],\alpha)$ is the largest number of disjoint open sub-intervals $(a,b)$ of  
$[s,t]$ on which $\varpi(a,b)$ is greater than or equal to $\alpha$.}
When $\varpi(s,t) = w(s,t,\omega)^{1/p}$ with $w$  a control satisfying \eqref{eq:w:s:t:omega:ineq} and \eqref{eq:useful:inequality:wT} and when the framework makes it clear, we just write $N([s,t],\omega,\alpha)$ for $N_{\varpi}([s,t],\alpha)$. Similarly, we also write $\tau_{n}(s,\omega,\alpha)$ for $\tau_{n}^{\varpi}(s,\alpha)$ when $\varpi(s,t) = w(s,t,\omega)$. We will also use the convenient notation
\begin{equation*}
\tau_{n}^{\varpi}(s,t,\alpha) := \tau_{n}^{\varpi}(s,\alpha) \wedge t. 
\end{equation*}

\section{Analysis of the Mean Field Rough Differential Equation}
\label{SectionSolving}

\subsection{Solving the Equation}

The following notion of solution to \eqref{EqRDE} is taken from \cite[Definition 4.1]{BCD1}.

\begin{definition}
\label{definitionGamma}
Let $W$ together with its enhancement ${\boldsymbol W}$ satisfy the assumption of Section \emph{\ref{SubsectionRegularityRoughSetup}} on a finite interval $[0,T]$.
%
%
{A solution to \eqref{EqRDE} on the time interval $[0,T]$, with initial condition}
$
X_{0}(\cdot) \in \LL^2(\Omega,{\mathcal F},{\mathbb P};\RR^d)
$, 
is a random controlled path $X(\cdot)$ such that for ${\mathbb P}$-a.e. $\omega$ {the paths $X(\omega)$ and $X_{0}(\omega) + \int_{0}^{\cdot} \textrm{\emph{F}}\bigl(X_{s}(\omega),{X}_{s}(\cdot)\bigr) d{\boldsymbol W}_{s}(\omega)$ coincide}.
\end{definition}

We formulate here the regularity assumptions on $\textrm{F}(x,\mu)$ used in \cite{BCD1},  {in addition to \textsf{\textbf{Regularity assumptions 1}}}, to show the well-posed character of Equation \eqref{EqRDE}. 
{Below}, we denote by $\big(\widetilde{\Omega},\widetilde{\mathcal F},\widetilde{\PP}\big)$ a copy of $(\Omega,{\mathcal F},\PP)$, and given a random variable $Z$ on $(\Omega,{\mathcal F},\PP)$, write $\widetilde{Z}$ for its copy on $\big(\widetilde\Omega,\widetilde{\mathcal F},\widetilde\PP\big)$. 

\bigskip

\noindent \textbf{\textsf{Regularity assumptions  2.}}   \vspace{0.15cm} 

\textcolor{gray}{$\bullet$} \textit{The function $\partial_{x} \textrm{\rm F}$ is differentiable in $(x,\mu)$.}  \vspace{0.15cm}

\textcolor{gray}{$\bullet$} \textit{For each $(x,\mu) \in \RR^d \times \cP_{2}(\RR^d)$, there exists a version of $D_{\mu}\textrm{\emph{F}}(x,\mu)(\cdot) \in L^2_{\mu}(\RR^d;\RR^d \otimes \RR^m)$ such that the map
$
(x,\mu,z) \mapsto D_{\mu}\textrm{\emph{F}}(x,\mu)(z)
$
from $\RR^d\times\mathcal{P}_2(\RR^d)\times\RR^d$ to $\textcolor{black}{\RR^{d} \otimes \RR^m \otimes \RR^d}$ is of class $C^1$, the derivative in the direction $\mu$ being understood as before.}      \vspace{0.15cm}

\textcolor{gray}{$\bullet$}
\textit{The function 
$\big(x,Z\big) \mapsto \partial_{x}^2 \textrm{\emph{F}}\bigl(x,\cL(Z)\bigr)$ from $\RR^d \times \LL^2(\Omega,{\mathcal F},\PP;\RR^d)$ to 
$
\RR^{d} \otimes \RR^{m} \otimes \RR^d \otimes \RR^d$ 
is bounded by $\Lambda$ and $\Lambda$-Lipschitz continuous. }   \vspace{0.15cm}

\textcolor{gray}{$\bullet$} \textit{The two derivative functions
$(x,Z) \mapsto \partial_x D_{\mu}\textrm{\emph{F}}\bigl(x,\cL(Z)\bigr)(Z(\cdot))$
(which is the same as 
$(x,Z) \mapsto D_{\mu} \partial_x \textrm{\emph{F}}\bigl(x,\cL(Z)\bigr)(Z(\cdot))$
by Schwarz' theorem)
and
$(x,Z) \mapsto \partial_z D_{\mu}\textrm{\emph{F}}\bigl(x,\cL(Z)\bigr)(Z(\cdot))$
are bounded by $\Lambda$ and $\Lambda$-Lipschitz continuous
from $\RR^d \times \LL^2\big(\Omega,{\mathcal F},\PP;\RR^d\big)$ to 
$
\LL^2\bigl(\Omega,{\mathcal F},\PP;\RR^{d} \otimes \RR^{m} \otimes \RR^d \otimes \RR^d\bigr)$.}
 
     \vspace{0.15cm}

\textcolor{gray}{$\bullet$} \textit{For each $\mu \in {\mathcal P}_{2}(\RR^d)$, we denote by $D^2_{\mu}\textrm{\emph{F}}(x,\mu)(z,\cdot)$,
the derivative of $D_{\mu}\textrm{\emph{F}}(x,\mu)(z)$ with respect to $\mu$ -- which is indeed given by a function. For $z' \in \RR^d$, 
$D_{\mu}^2\textrm{\emph{F}}(x,\mu)(z,z')$ is an element of $\RR^d \otimes \RR^m \otimes \RR^d \otimes \RR^d$. We assume that 
$$
(x,Z) \mapsto D^2_{\mu} \textrm{\emph{F}}\bigl(x,\cL(Z)\bigr)\big(Z(\cdot),\widetilde{Z}(\cdot)\big),
$$ 
from $\RR^d \times \LL^2(\Omega,{\mathcal F},\PP;\RR^d)$ to 
$
\LL^2\Bigl(\Omega \times \widetilde{\Omega},{\mathcal F} \otimes \widetilde{\mathcal F},\PP \otimes \widetilde{\mathbb P};\RR^{d} \otimes \RR^{m} \otimes \RR^d \otimes \RR^d\Bigr),$ 
is bounded by $\Lambda$ and $\Lambda$-Lipschitz continuous.}   

\medskip

{The two functions
$
\textrm{F}(x,\mu) = \int f(x,y)\mu(dy)$ and $\textrm{F}(x,\mu) = g\big(x,\int y\mu(dy)\big)$, 
for functions $f,g\in C^3_b$ (meaning that $f$ and $g$ are bounded and have bounded derivatives of order 1, 2 and 3), satisfy the \textsf{\textbf{Regularity assumptions 1 and 2}}.}
The following {property} is taken from 
\cite[Proposition 4.3 and (4.21)]{BCD1}. 

\begin{proposition}
\label{theorem:fixed:1}
Let \emph{F} satisfy \textbf{\textsf{Regularity assumptions 1}}  and \textbf{\textsf{2}} and $w$ be a control 
satisfying \eqref{eq:w:s:t:omega:ineq}
and \eqref{eq:useful:inequality:wT}. Consider two $\omega$-controlled paths $X(\omega)$ and $X'(\omega)$ {with possibly different initial conditions $(X_{0}(\omega),\delta_{x} X_{0}(\omega))$
and $(X_{0}'(\omega),\delta_{x} X_{0}'(\omega))$}, defined on a time interval $[0,T]$, together with two random controlled paths $Y(\cdot)$ and $Y'(\cdot)$,
{with possibly different initial conditions $(Y_{0}(\omega),\delta_{x} Y_{0}(\omega))$
and $(Y_{0}'(\omega),\delta_{x} Y_{0}'(\omega))$, 
all of them 
satisfying
$\delta_{\mu} X(\omega) \equiv \delta_{\mu} X'(\omega) \equiv 0$
and 
$\delta_{\mu} Y(\cdot) \equiv \delta_{\mu} Y'(\cdot) \equiv 0$} together with
\begin{equation}
\label{eq:fixed:00}
\bigl|\delta_{x} X(\omega) \bigr| \vee \bigl| \delta_{x} X'(\omega)\bigr| \vee \big\langle \delta_{x} Y(\cdot)\big\rangle_{\infty} \vee  \big\langle \delta_{x} Y'(\cdot) \big\rangle_{\infty} \leq \textcolor{black}{\Lambda},
\end{equation}
and the size estimates
\begin{align}
\label{eq:fixed:1}
&\bigl\langle \vvvert Y(\cdot) \vvvert_{[0,T],w,p} \bigr\rangle_{8}^2 \leq L_{0},   \   \bigl\langle \vvvert Y'(\cdot) \vvvert_{[0,T],w,p} \bigr\rangle_{8}^2 \leq L_{0}, 
\\
\label{eq:fixed:2}
&\big\vvvert X(\omega) \big\vvvert_{[t_{i}^0,t_{i+1}^0],w,p}^2 \leq {L_{0}}, \qquad \big\vvvert X'(\omega) \big\vvvert_{[t_{i}^0,t_{i+1}^0],w,p}^2 \leq {L_{0}},  
\end{align}
for $i \in \{0,\cdots,N^{0}\}$, for some $L_0 \geq 1$, and $N^0=N\bigl(\textcolor{black}{[0,T]},\omega,1/(4L_{0})\bigr)$ given by \eqref{eq:N:s:t:omega}, and for the sequence $\bigl(t_{i}^0=\tau_{i}(0,T,\omega,1/(4L_{0}))\bigr)_{i=0,\cdots,N^0+1}$ given by \eqref{eq:stopping:times}. 

Then, we can find a constant $\gamma$ depending on $L_{0}$ and $\Lambda$ such that, for any partition $(t_{i})_{\textcolor{black}{i=0,\cdots,N}}$ included in $(t_i^0)_{\textcolor{black}{i=0,\cdots,N^0}}$ and satisfying 
$
w(t_{i},t_{i+1},\omega)^{1/p} \leq 1/(4L)$ 
for some \textcolor{black}{$L\geq L_0$}, we have
\begin{equation}
\label{eq:3.16:BCD1}
\begin{split}
&\biggl\vvvert \int_{t_{i}}^{\cdot} \Bigl( \textrm{F}\bigl(X_{r}(\omega),Y_{r}(\cdot)\bigr) - 
\textrm{F}\bigl(X_{r}'(\omega),Y_{r}'(\cdot)\bigr) \Bigr) d {\boldsymbol W}_{r}(\omega) \biggr\vvvert_{[t_{i},t_{i+1}],w,p}   
\\
&\leq \gamma \, \Bigl(  \big\vert \Delta X_{0}(\omega)\big\vert +
\big\vert \delta_{x} \Delta  X_{0}(\omega)\big\vert \Bigr)
+ 
 \bigl\langle \Delta Y_{0}(\cdot)  \bigr\rangle_{4}
 +
 \bigl\langle  \delta_{x} \Delta Y_{0}(\cdot)  \bigr\rangle_{4}
 \\ 
& \hspace{5pt} + \gamma\, w(0,t_{i},\omega)^{1/p}\,\Big(\big\vvvert \Delta X(\omega) \big\vvvert_{[0,t_{i}],w,p} + \big\langle \vvvert \Delta Y(\cdot) \vvvert_{[0,T],w,p} \big\rangle_{8} \Big)   \\
&\hspace{5pt} + \frac{\gamma}{4L}  \Big(\big\vvvert \Delta X(\omega)\big\vvvert_{[t_{i},t_{i+1}],w,p} + \big\langle \vvvert \Delta Y(\cdot) \vvvert_{[0,T],w,p}\big\rangle_{8} \Big),
\end{split}
\end{equation}\textcolor{black}{where
$
\Delta X_{t}(\omega) := X_{t}(\omega) - X_{t}'(\omega)$, $\Delta Y_{t}(\cdot) := Y_{t}(\cdot) - Y_{t}'(\cdot), \quad t \in [0,T].$
}
\end{proposition}

In \cite{BCD1}, Proposition \ref{theorem:fixed:1} is used to prove the following existence and uniqueness result, see Theorems 1.1 and 4.4 therein, 
to which we add the final estimate in the statement.  

\begin{thm}
\label{main:theorem:existence:small:time}
Let \emph{F} satisfy \textbf{\textsf{Regularity assumptions 1}}  and \textbf{\textsf{2}} and $w$ be a control 
satisfying \eqref{eq:w:s:t:omega:ineq} and \eqref{eq:useful:inequality:wT}. Assume there exists a positive time horizon $T$ such that the random variables $w(0,T,\textcolor{black}{\cdot})$ and $\bigl(N\big([0,T],\textcolor{black}{\cdot},\alpha\big)\bigr)_{\alpha >0}$ have sub and super exponential tails respectively, {in the sense that}
\begin{equation}
\label{EqTailAssumptions}
\begin{split}
&\PP \bigl( w(0,T,\cdot) \geq t \bigr) \leq c_{1} \exp \bigl( - t^{\varepsilon_1} \bigr),   
\\
&\PP \bigl( N([0,T],\cdot,\alpha) \geq t \bigr) \leq c_{2}(\alpha) \exp \bigl( - t^{1+ \varepsilon_2(\alpha)} \bigr),
\end{split}
\end{equation}
for some positive \textcolor{black}{constants $c_{1}$} and $\epsilon_1$, and possibly $\alpha$-dependent positive constants $c_{2}(\alpha)$ and $\epsilon_2(\alpha)$. Then, for any $d$-dimensional square-integrable random variable $X_{0}$, the mean field rough differential equation 
\eqref{EqRDE} has a unique solution defined on the whole interval $[0,T]$. Moreover, there exist four positive real numbers $\gamma_{0}$, $L_{0}$, $L$ and $\eta_{0}$ (with $\gamma_{0},\eta_{0} >1$), only depending on  $\Lambda$ and $T$, such that, for any 
subinterval 
$[S_{1},S_{2}] \subset [0,T]$ for which 
\begin{equation*}
{\Bigl\langle N\bigl([S_{1},S_{2}],\cdot,1/(4L_{0})\bigr)   \Bigr\rangle_{8}}  \leq {1},
\end{equation*}
and 
\begin{equation*}
\left\langle \Bigl[ \gamma \Bigl( 1 + w(0,T,\cdot)^{1/p} \Bigr)\Bigr]^{N([S_{1},S_{2}],\cdot,1/(4L))} \right\rangle_{32} \textcolor{black}{\leq \eta_{0}},
\end{equation*}
it holds, for any $\omega \in \Omega$
\begin{equation*}
\begin{split}
&\vvvert X(\omega) \vvvert_{[S_{1},S_{2}],w,p}  
  \leq \Bigl[ C 
 \Bigl( 1 + w(0,T,\omega)^{1/p} \Bigr)
 \Bigr]^{2 N([0,T],\omega,1/(4L))}, 
\end{split}
\end{equation*} 
 for a constant $C$ depending only on $\Lambda$ and $T$.  
\end{thm}

\begin{proof}
{We just address the derivation of the last inequality since the latter is not given in 
\cite[Theorem 4.4]{BCD1}. The key point is to sum over $n \geq 1$ in \cite[(4.30)]{BCD1}, replacing 
$[0,S]$ therein by $[S_{1},S_{2}]$, which is indeed licit provided that 
$\bigl\langle N\bigl([S_{1},S_{2}],\cdot,1/(4L_{0})\bigr)   \bigr\rangle_{8}  \leq 1$, 
see for instance \cite[(4.23)]{BCD1}, 
and
$\left\langle \bigl[ \gamma_{0} \bigl( 1 + w(0,T,\cdot)^{1/p} \bigr)\bigr]^{N([S_{1},S_{2}],\cdot,1/(4L))} \right\rangle_{32} \textcolor{black}{\leq \eta_{0}}$ for $\eta_{0}$ small enough, see \cite[(4.29)]{BCD1}. 
}
\end{proof}

\subsection{Strong Rough Set-Ups and Continuity of the It\^o-Lyons solution Map}
\label{subse:geom}

Uniqueness in law of the solutions to \eqref{EqRDE} is proven in \cite[Theorem 5.3]{BCD1} under the additional assumption that the set-up satisfies the following definition.

\begin{definition}
\label{def:strong}
A  {{rough set-up}} is called  {{strong}} if there exists a measurable mapping ${\mathcal I}$ from ${\mathcal C}\big([0,T];\RR^m\big)^{2}$ into ${\mathcal C}\big({\mathcal S}_{2}^T;\RR^m \otimes \RR^m\big)$ such that
\begin{equation}
\label{eq:representation}
\PP^{\otimes 2} \Bigl( \bigl\{ (\omega,\omega') \in \Omega^2 : {\mathbb W}^{\indep}(\omega,\omega')={\mathcal I}\bigl(W(\omega),W(\omega')\bigr) \bigr\} \Bigr) = 1.
\end{equation}
\end{definition}

For our prospect, the following \textit{continuity theorem} is of crucial interest; see\footnote{\label{foo:concatenation} Here, we feel useful to say a word about the proof of Theorem \cite[Theorem 5.4]{BCD1}. The proof of Step 2b therein is a bit short. The reader may indeed wonder why $K$ therein may be chosen independently of $n$. In fact, it suffices to observe, with the same notations as therein, that we can render 
${\mathbb P}(N^{n}([S_{j},S_{j+1}],\cdot,1/(4L_{0})) \geq 1)$ as small as needed, uniformly in $n$. 
This follows from the fact that 
${\mathbb P}(N^{n}([S_{j},S_{j+1}],\cdot,1/(4L_{0})) \geq 1) = {\mathbb P}(w^{n}(S_{j},S_{j+1},\cdot) \geq 1/(4L_{0})) \leq 4L_{0} \langle w^{n}(S_{j},S_{j+1},\cdot)  \rangle$. By the second item in the assumption of 
Theorem \ref{theoremContinuity}, 
the last term is less than $C (S_{j}-S_{j+1})$, for $C$ independent of $n$.} \cite[Theorem 5.4]{BCD1}.

\begin{thm}
\label{theoremContinuity}
Let \emph{F} satisfy the same assumptions as in \emph{Theorem \ref{main:theorem:existence:small:time}}. 
Given a time interval $[0,T]$ and a sequence of probability spaces $(\Omega_{n},{\mathcal F}_{n},\PP_{n})$, indexed by $n \in {\mathbb N}$, 
let, for any $n$, 
$X^n_{0}(\cdot) := (X^n_{0}(\omega_{n}))_{\omega_{n} \in \Omega_{n}}$
be an $\RR^d$-valued square-integrable initial condition
and
$$
{\boldsymbol W}^n(\cdot) := \Big(W^n(\omega_{n}),\WW^n(\omega_{n}),\WW^{n,\indep}(\omega_{n},\omega_{n}')\Big)_{\omega_{n},\omega_{n}' \in \Omega_{n}}
$$
be an $m$-dimensional rough set-up with corresponding control $w^n$, as given by 
\eqref{eq:w:s:t:omega}, and local accumulated variation $N^n$, for fixed values of $p \in [2,3)$ and $q > 8$. Assume that
\begin{itemize}
\item[\textcolor{gray}{$\bullet$}] the collection 
$\bigl(\PP_{n} \circ (\vert X^n_{0}(\cdot) \vert^2)^{-1}\bigr)_{n \geq 0}$ is uniformly integrable;
\item[\textcolor{gray}{$\bullet$}] for positive constants $\varepsilon_{1},c_{1}$ and $(\varepsilon_{2}(\alpha),c_{2}(\alpha))_{\alpha >0}$,  the tail assumption \eqref{EqTailAssumptions} hold for $w^n$ and $N^n$, for all $n\geq 0$;
\item[\textcolor{gray}{$\bullet$}] associating $v^n$ with each ${\boldsymbol W}^n(\cdot)$ as in \eqref{eq:v}, the functions 
   $
   \bigl({\mathcal S}_{2}^T \ni (s,t) \mapsto \langle v^n(s,t,\cdot) \rangle_{2q}\bigr)_{n \geq 0}$ 
   are uniformly Lipschitz continuous.  
 \end{itemize}
 Assume also that 
\begin{itemize}
\item[\textcolor{gray}{$\bullet$}] 
 there exist, on another probability space
$(\Omega,{\mathcal F},\PP)$, 
a square integrable initial condition $X_{0}(\cdot)$
with values in $\RR^d$ and 
 a strong rough set-up
$$
{\boldsymbol W}(\cdot) := \Big(W(\omega),\WW(\omega),\WW^{\indep}(\omega,\omega')\Big)_{\omega,\omega' \in \Omega}
$$
with values in $\RR^m$, such that
the law under the probability measure $\PP_{n}^{\otimes 2}$ 
of the random variable
$$
\Omega_{n}^2 \ni (\omega_{n},\omega_{n}') \mapsto \bigl(X_{0}^n(\omega_{n}), W^n(\omega_{n}),{\mathbb W}_{n}(\omega_{n}),{\mathbb W}^{\indep}_{n}(\omega_{n},\omega_{n}')\bigr),
$$
seen as a random variable with values in the space ${\RR^d \times} {\mathcal C}([0,T];\RR^m) \times \bigl\{ {\mathcal C}({\mathcal S}_{2}^T;\RR^m \otimes \RR^m) \bigr\}^2$, converges in the weak sense to the law of 
$$
\Omega^2 \ni (\omega,\omega') \mapsto \bigl(X_{0}(\omega), W(\omega),{\mathbb W}(\omega_{n}),{\mathbb W}^{\indep}(\omega,\omega')\bigr).
$$
\end{itemize}
Then, ${\boldsymbol W}(\cdot)$ satisfies the requirements of \emph{Theorem \ref{main:theorem:existence:small:time}}
for some $p' \in (p,3)$ and $q' \in [8,q)$, with control $w$ therein given by \eqref{eq:w:s:t:omega}. Moreover, if $X^n(\cdot)$, resp. $X(\cdot)$, is the solution of the mean field rough differential equation driven by ${\boldsymbol W}^n(\cdot)$, resp. ${\boldsymbol W}(\cdot)$, then $X^n(\cdot)$ converges in law to $X(\cdot)$ on ${\mathcal C}([0,T];\RR^d)$.
\end{thm}

\section{Particle System and Propagation of Chaos}
\label{SectionParticleSystem}

We now have all the ingredients to write down the limiting mean field rough differential equation \eqref{EqRDE} as the limit of a system of particles driven by rough signals \eqref{eq:particle:system}.

\subsection{Empirical Rough Set-Up}
\label{subse:empirical}

We recall the framework used to address \eqref{eq:particle:system}. 
The initial conditions 
 $(X_{0}^i(\cdot))_{1 \leq i \leq n}$ are $\RR^d$-valued variables with the same distribution as $X_{0}$ (in the statement of Theorem \ref{main:theorem:existence:small:time}) and the enhanced signals 
 $\big(W^i(\cdot),{\mathbb W}^i(\cdot)\big)_{1 \leq i \leq n}$
 are $\RR^m \oplus \RR^{m} \otimes \RR^m$-valued with the same distribution 
 as $(W(\cdot),{\mathbb W}(\cdot))$
 on the space of continuous functions. 
Moreover, the variables  $(X_{0}^i(\cdot),W^i(\cdot),{\mathbb W}^i(\cdot)\big)_{1 \leq i \leq n}$
are independent and identically distributed. 
 All of them are constructed on a single probability space, {still denoted by} $(\Omega,{\mathcal F},{\mathbb P})$. 
Assuming 
the rough set-up in 
 Theorem \ref{main:theorem:existence:small:time}
 to be strong, see Definition 
\ref{def:strong}, we let
\begin{equation*}
{\mathbb W}^{i,j}(\omega) = {\mathcal I}\bigl(W^i(\omega),W^j(\omega)\bigr), \quad i \not = j, \quad 1 \leq i,j \leq n.
\end{equation*}
 
 Obviously, equation \eqref{eq:particle:system} must be understood as a rough differential equation driven by an $(n\times m)$-dimensional signal $\big(W^1(\omega),\cdots,W^n(\omega)\big)$, and with $\big(X^1(\omega),\cdots,X^n(\omega)\big)$ as $(n \times d)$-dimensional output. 
 Our first task is to prove that \eqref{eq:particle:system} may be also understood as 
 a mean field rough differential equation on a suitable rough set-up
and that the two interpretations coincide.  
 If
 we require 
$\PP^{\otimes 2} \bigl( \{ (\omega,\omega') : 
\|
{\mathbb W}^{\indep}(\omega,\omega') \|_{[0,T],p/2-\textrm{\rm v}}
< \infty \} \bigr)=1$
in 
Definition 
\ref{def:strong}, then it is pretty clear that, for almost every $\omega \in \Omega$, 
\begin{equation*}
\begin{split}
{\boldsymbol W}^{(n)}(\omega) &= \Bigl( \big(W^i(\omega)\big)_{1 \leq i \leq n}, \big({\mathbb W}^{i,j}(\omega)\big)_{1 \leq i,j \leq n} \Bigr)   
=: \Big(W^{(n)}(\omega), {\mathbb W}^{(n)}(\omega)\Big),
\end{split}
\end{equation*}
is a rough path {of finite $p$-variation}, with the convention that ${\mathbb W}^{i,i}(\omega) = {\mathbb W}^i(\omega)$, for $i \in \{1,\cdots,n\}$. As explained in \cite[Proposition 2.3]{BCD1}, we may change the definition of $\big((W^i(\omega))_{1 \leq i \leq n},({\mathbb W}^{i,j}(\omega))_{1 \leq i,j \leq n}\big)$ on a $\PP$-null set so that ${\boldsymbol W}^{(n)}(\omega)$ is in fact a rough path for any $\omega \in \Omega$. 
\medskip

 {As mentioned in Introduction,} the striking fact of the analysis {was first introduced by Tanaka in \cite{TanakaTrick} and used} by Cass and Lyons in their seminal work \cite{CassLyons}. The quantity ${\mathbb W}^{(n)}(\omega)$ may be seen as a rough set-up defined on a finite probability space for any fixed $\omega \in \Omega$; we call it the \textsf{\textbf{empirical rough set-up}}. To make it clear, observe that, throughout Section \ref{SectionRoughStructure}, the rough structure is supported by the probability space $(\Omega,{\mathcal F},{\mathbb P})$ itself. Here, $\omega$ is fixed, and we see the probability space as
\begin{equation}
\label{eq:empirical:space:probability}
\biggl( \bigl\{1,\cdots,n\bigr\},{\mathcal P}\bigl(\bigl\{1,\cdots,n\bigr\}\bigr), \frac1n \sum_{i=1}^n \delta_{i} \biggr),
\end{equation}
where ${\mathcal P}(\{1,\cdots,n\})$ denotes the collection of subsets of $\{1,\cdots,n\}$. The reader may object that such a probability space is not atomless whilst we explicitly assumed $(\Omega,{\mathcal F},{\mathbb P})$ to be atomless in the introduction (see also \cite[Section 2]{BCD1}); actually, the reader must realize that, in \cite{BCD1}, the atomless property is just used to guarantee that, for any probability measure $\mu$ on a given Polish space $S$, the probability space $(\Omega,{\mathcal F},{\mathbb P})$ carries an $S$-valued random variable with $\mu$ as distribution. 
{In other words, we could instead say that, in \cite{BCD1}, the probability space $(\Omega,{\mathcal F},{\mathbb P})$ has to be rich enough, which is indeed guaranteed under the assumptions used therein.
So, it is here not a hindrance that $\{1,\cdots,n\}$ is finite: We must restrict ourselves to random variables taking at most 
$n$-values, but this is exactly what we need for our purposes since all the relevant probability distributions showing up from the particle system are $n$-empirical distributions.}

Hence, in {order to draw a parallel with \eqref{eq:W:2}}, the role played by $\omega \in \Omega$ is {here} played by $i \in \{1,\cdots,n\}$ and the matrix \eqref{eq:W:2} must read  
\begin{equation}
\label{eq:W:2:empirical}
\left(
\begin{array}{cc}
\WW_{s,t}^{i,i}(\omega) &\WW_{s,t}^{i,\bullet}(\omega)
\\
\WW_{s,t}^{\bullet,i}(\omega) &\WW_{s,t}^{\bullet,\bullet}(\omega)
\end{array}
\right)_{0 \leq s \leq t\leq T},
\end{equation}
where $\WW_{s,t}^{i,\bullet}(\omega)$ is seen as $\{1,\cdots,n\} \ni j \mapsto \WW_{s,t}^{i,j}(\omega)$, $\WW_{s,t}^{\bullet,i}(\omega)$ as $\{1,\cdots,n\} \ni j \mapsto \WW_{s,t}^{j,i}(\omega)$ and $\WW_{s,t}^{\bullet,\bullet}(\omega)$ as $\{1,\cdots,n\} \ni (i,j) \mapsto \WW_{s,t}^{i,j}(\omega)$.

\smallskip

In the same spirit, the variation function $v$ in \eqref{eq:v} is ({we put a subscript $p$ in the variation function below to emphasize the dependence upon $p$})
\begin{equation}
\label{eq:v:N}
\begin{split}
v^{i,n}_{{p}}(s,t,\omega) &:= \big\| W^i(\omega) \big\|_{[s,t],p-\textrm{\rm v}}^p + 
{{}^{(n)} \hspace{-1pt} \big\lgroup} {W}^{\bullet}(\omega) \big\rgroup_{q ; [s,t],p-\textrm{\rm v}}^p   
\\
&\hspace{15pt} + \big\| \WW^i(\omega) \big\|_{[s,t],p/2-\textrm{\rm v}}^{p/2} + 
{{}^{(n)} \hspace{-1pt} \big\lgroup} \WW^{i,\bullet}(\omega) \big\rgroup_{q ; [s,t],p/2-\textrm{\rm v}}^{p/2}   \\
&\hspace{15pt} + 
{{}^{(n)} \hspace{-1pt} \big\lgroup}
 \WW^{\bullet,i}(\omega) \big\rgroup_{q ; [s,t],p/2-\textrm{\rm v}}^{p/2} + 
 {}^{(n)} \hspace{-1pt}\big\lgroup \hspace{-2pt}\big\lgroup
  \WW^{\bullet,\bullet}(\omega) \big\rgroup \hspace{-2pt} \big\rgroup_{q ; [s,t],p/2-\textrm{\rm v}}^{p/2},
\end{split}
\end{equation}
where we used the notations
\begin{equation*}
{}^{(n)} \hspace{-.7pt} \lgroup X^{\bullet} \rgroup_{q}= 
\biggl( \frac1n
\sum_{j=1}^n \vert X^j \vert^q \biggr)^{1/q}, \quad
{}^{(n)} \hspace{-.7pt} \lgroup \hspace{-1.3pt} \lgroup X^{\bullet,\bullet} 
\rgroup \hspace{-1.3pt}
\rgroup_{q} = \biggl( \frac1{n^2}\sum_{j,k=1}^n \vert X^{j,k} \vert^q \biggr)^{1/q},
\end{equation*}
the corresponding $p$-variation being defined as in \eqref{eq:q:p-var} and \eqref{eq:q:q:p-var}. 
{Obviously, $v_{{p}}^{i,n}(0,T,\omega)$ is almost surely finite. Hence,}
in order to check that ${\boldsymbol W}^{(n)}(\omega)$ defines a rough set-up, it remains to check that it satisfies \eqref{eq:v:1-var}. 
To do so, we 
strengthen the assumptions on the signal and assume that, for the same parameter $q$ as in Section 
\ref{SectionRoughStructure}, it holds
\begin{equation}
\label{eq:integrability:Holder:norm}
\begin{split}
&{\mathbb E}
\Bigl[
\bigl\|
W(\cdot)
\bigr\|_{[0,T],(1/p)-\textrm{\rm H}}^{pq}
+
\bigl\|{\mathbb W}(\cdot) \bigr\|_{[0,T],(2/p)-\textrm{\rm H}}^{pq/2}
\Bigr]
\\
&\hspace{15pt}
+
{\mathbb E}^{\otimes 2}
\Bigl[
\bigl\|
{\mathbb W}^{\indep}(\cdot,\cdot) \bigr\|_{[0,T],(2/p)-\textrm{\rm H}}^{pq/2}
\Bigr]
 < \infty,
 \end{split}
 \end{equation}
where
\begin{equation*}
\begin{split}
\bigl\|
W(\omega) \bigr\|_{[s,t],(1/p)-\textrm{\rm H}}
&= \sup_{\emptyset \not = (s',t') \subset [s,t]}
\frac{\vert W_{t'}(\omega) - W_{s'}(\omega) \vert}{\vert t'-s' \vert^{1/p}}
\\
\bigl\| {\mathbb W}(\omega) \bigr\|_{[s,t],(2/p)-\textrm{\rm H}}
&=
\sup_{\emptyset \not = (s',t') \subset [s,t]}
\frac{\vert {\mathbb W}_{s',t'}(\omega) \vert}{\vert t'-s' \vert^{2/p}},
\end{split}
\end{equation*}
{and similarly for 
$\bigl\| {\mathbb W}^{\indep}(\omega,{\omega'}) \bigr\|_{[s,t],(2/p)-\textrm{\rm H}}$},
stand for the standard H\"older semi-norms of the rough path. 
Then, back 
to 
\eqref{eq:v:N}, 
we can find a universal positive constant $c$ such that 
\begin{equation}
\label{eq:v:p:i:N}
\begin{split}
&v_{p}^{i,n}(s,t,\omega) \leq  c\,  \Big\{ 
\bigl\|
W^i(\omega) \bigr\|_{[s,t],(1/p)-\textrm{\rm H}}^p
+
\bigl\|
{\mathbb W}^i(\omega) \bigr\|_{[s,t],(2/p)-\textrm{\rm H}}^{p/2}
\\
&\hspace{2pt} 
 + 
{}^{(n)} \hspace{-1pt} \bigl\lgroup 
\bigl\|
W^{\bullet}(\omega) \bigr\|_{[s,t],(1/p)-\textrm{\rm H}}^p
 \bigr\rgroup_{q}  
 + 
{}^{(n)} \hspace{-1pt} \bigl\lgroup 
\bigl\|
 {\mathbb W}^{i,\bullet}(\omega) \bigr\|_{[s,t],(2/p)-\textrm{\rm H}}^{p/2}
 \bigr\rgroup_{q} 
\\
&\hspace{2pt} 
 + 
{}^{(n)} \hspace{-1pt} \bigl\lgroup 
\bigl\|
 {\mathbb W}^{\bullet,i}(\omega) \bigr\|_{[s,t],(2/p)-\textrm{\rm H}}^{p/2}
 \bigr\rgroup_{q} 
+
 {}^{(n)} \hspace{-1pt}\big\lgroup \hspace{-2pt}\big\lgroup
 \bigl\|
 {\mathbb W}^{\bullet,\bullet}(\omega) \bigr\|_{[s,t],(2/p)-\textrm{\rm H}}^{p/2}   
   \big\rgroup \hspace{-2pt} \big\rgroup_{q}
\Big\} (t-s).
\end{split}
\end{equation}
Taking the mean over $i \in \{1,\cdots,n\}$ and invoking the law of large numbers
({see Lemma \ref{lem:LLN:2nd order} in Appendix \ref{subse:lln} for a version of the law of large numbers with 
second order interactions}),
we deduce that, for almost every $\omega \in \Omega$, 
\begin{align}
\nonumber
&\limsup_{n \geq 1}
\sup_{0 \leq s < t \leq T}
\frac{{}^{(n)} \hspace{-1pt} \bigl\lgroup v_{p}^{\bullet,n}(s,t,\omega)
\bigr\rgroup_{q}}{t-s} 
\\
&\hspace{5pt} \leq c \, 
\Bigl(
\Bigl\langle
\bigl\|
W(\cdot)
\bigr\|_{[0,T],(1/p)-\textrm{\rm H}}^{pq}
+
\bigl\|{\mathbb W}(\cdot) \bigr\|_{[0,T],(2/p)-\textrm{\rm H}}^{pq/2}
\Bigr\rangle^{1/q} \label{eq:lln:holder}
\\
&\hspace{30pt}
+
\Bigl\llangle
\bigl\|
{\mathbb W}^{\indep}(\cdot,\cdot) \bigr\|_{[0,T],(2/p)-\textrm{\rm H}}^{pq/2}
\Bigr\rrangle^{1/q} \Bigr),
\nonumber
\end{align}
for a new value of the constant $c$. Observe that, in order to derive \eqref{eq:lln:holder}, the law of large numbers can be directly applied to each of the {first three terms} in the right-hand side of \eqref{eq:v:p:i:N}, since each of them 
{lead to empirical means over terms of the form}
 ${\mathcal J}\bigl(W^{i}(\omega)\bigr)$, for a suitable {functional ${\mathcal J}$ (which has nothing to do with the mapping ${\mathcal I}$ used in Definition \ref{def:strong})}. {Differently, the last three terms in \eqref{eq:v:p:i:N} require a modicum of care as they lead to empirical means of the form}
\begin{equation*}
\begin{split}
&\frac1{n^2} \sum_{j,k=1, j \not = k}^n 
\bigl\| {\mathcal I}\bigl(W^j(\omega),W^k(\omega)\bigr)
\bigr\|_{[0,T],(2/p)-\textrm{\rm H}}^{pq/2}
 + \frac1{n^2} \sum_{j=1}^n \bigl\| {\mathbb W}^j(\omega) \bigr\|_{[0,T],(2/p)-\textrm{\rm H}}^{pq/2},
\end{split}
\end{equation*}
with 
${\mathcal I}$ as in 
\eqref{eq:representation}. 
Still, if the summands in the two sums are integrable, 
the limit is $\bigl\langle
\bigl\| {\mathcal I}\bigl(W^1(\cdot),W^2(\cdot)\bigr)
\bigr\|_{[0,T],(2/p)-\textrm{\rm H}}^{pq/2}
\bigr\rangle$, {see {once again} Lemma 
\ref{lem:LLN:2nd order} in Appendix \ref{subse:lln}}. Hence 
\eqref{eq:lln:holder}. Now, 
the fact that the right-hand side of 
\eqref{eq:lln:holder} is finite guarantees that the 1-variation in the mean in \eqref{eq:v:1-var} is uniformly controlled in $n \geq 1$, the mean in \eqref{eq:v:1-var} being understood 
as the mean on the probability space $\big(\{1,\cdots,n \},{\mathcal P} (\{1,\cdots,n \} ), \frac1n \sum_{i=1}^n \delta_{i}\big)$. Here are two examples under which \eqref{eq:v:p:i:N} holds true. 

\begin{Example}
Assume that the regularity index $q$ used in \eqref{eq:v} satisfies the inequality
$q > 1/(1-p/3)$,
and that, for some constant $K_{T} \geq 0$, $\langle v(s,t,\cdot) \rangle_{q} \leq K_{T} (t-s)$ for $(s,t) \in {\mathcal S}_{2}^T$. Then, we get the bounds
\begin{equation*}
\begin{split}
&\bigl\langle \vert (W_{t} - W_{s})(\cdot) \vert^{pq} \bigr\rangle \leq K_{T}^q \,\vert t-s \vert^{q},
\\
&\bigl\langle  \vert {\mathbb W}_{s,t}(\cdot) \vert^{pq/2} \bigr\rangle \leq K_{T}^q \,\vert t-s \vert^{q},
\ \bigl\llangle \vert {\mathbb W}_{s,t}^{\indep}(\cdot,\cdot) \vert^{pq/2} \bigr\rrangle \leq K_{T}^q \, \vert t-s \vert^{q}.
\end{split}
\end{equation*}
By Kolmogorov's criterion for rough paths, {see} Theorem 3.1 in \cite{FrizHairer}, we deduce that $W$ has paths that are ${1/p'}:=(1-1/q)/p > 1/3$-H\"older continuous. Similarly, ${\mathbb W}$
and ${\mathbb W}^{\indep}$ have paths that are $2p'=2(1-1/q)/p > 2/3$-H\"older continuous
and 
\eqref{eq:integrability:Holder:norm}
holds true {with $p'$ instead of $p$}. 
So, the empirical rough set-up satisfies the required conditions provided we replace $p$ by $p'$.
\end{Example}

\begin{Example}
\label{Gaussian}
{Assume that  $W := (W^1,\cdots,W^m)$ is a tuple of independent and centred continuous Gaussian processes, defined on $[0,T]$, for which there exist 
{an exponent 
$\varrho \in [1,3/2)$
and
a constant $K$} such that, for any subinterval $[s,t]\subset [0,T]$ and any $k=1,\cdots,m$, it holds
\begin{equation}
\label{eq:FV:rho:variation:covariance}
\begin{split}
\sup \, \sum_{i,j} \, \Bigl\vert {\mathbb E} \Bigl[ \bigl(W^k_{t_{i+1}} - W^{k}_{t_{i}} \bigr) \bigl( W^k_{s_{j+1}} - W^k_{s_{j}} \bigr) \Bigr] \Bigr\vert^\rho \leq K \vert t-s\vert,
\end{split}
\end{equation} 
the \textrm{\rm sup} being over divisions $(t_{i})_{i}$ and $(s_{j})_{j}$ of $[s,t]$.}
Then, $\|W(\cdot) \|_{[0,T],(1/p)-\textrm{\rm H}}$ has Gaussian tail and $
\|{\mathbb W}(\cdot) \bigr\|_{[0,T],(2/p)-\textrm{\rm H}}$ and  $\|{\mathbb W}^{\indep}(\cdot,\cdot) \bigr\|_{[0,T],(2/p)-\textrm{\rm H}}$ 
have exponential tails, {for any $p \in (2 \varrho,3)$}; see Theorem 11.9 in \cite{FrizHairer}. 
\end{Example}

Now that we have defined the empirical rough set-up, we must make clear the meaning given to the rough differential equation \eqref{EqRDE} in Definition \ref{definitionGamma} when the rough set-up therein is precisely the empirical rough set-up. We call the corresponding rough differential equation the \emph{empirical rough differential equation}.

\medskip

For a given $\omega \in \Omega$, the probability space that carries the empirical rough-set up is 
{given by \eqref{eq:empirical:space:probability}}.
Despite the fact it is not atomless, whilst $(\Omega,{\mathcal F},{\mathbb P})$ is, Theorem \ref{main:theorem:existence:small:time} applies and guarantees existence and uniqueness of a solution to the empirical rough differential equation. In this regard, observe that the square integrability requirement on the initial condition here writes
$\frac1n \sum_{i=1}^n \big\vert X_{0}^i(\omega)\big\vert^2 < \infty$,
which is indeed satisfied  for $\omega$ in a full event. The solution reads in the form of a $n$-tuple $X^{(n)}(\omega)=(X^i(\omega))_{1 \leq i \leq n}$ in ${\mathcal C}([0,T];\RR^d)^n$. 
The coefficient driving the equation for $X^i(\omega)$ reads
\begin{equation*}
\textrm{\rm F}\left(X^i_{t}(\omega),X_{t}^{\theta_{n}(\cdot)}(\omega) \right), \quad t \in [0,T],
\end{equation*}
where {$\theta_{n}(\cdot) : \{1,\cdots,n\} \ni j \mapsto j$ is the canonical random variable on $\{1,\cdots,n\}$}. Here the dot in the notation $X^{\theta_{n}(\cdot)}_{t}(\omega)$ refers to the current element in $\{1,\cdots,n\}$.
{With this notation}, the law of $X^{\theta_{n}(\cdot)}_{t}(\omega)$
({on $\{1,\cdots,n\}$})
 must be understood as the empirical distribution $\mu^{n}_{t}(\omega)$.
{Moreover,}
each $X^i(\omega)$ is controlled,  in standard Gubinelli's sense, by the enhanced rough path $\big(W^i(\omega),{\mathbb W}^i(\omega)\big)$ (the remainder in the expansion being controlled by $v^{i,n}$).
{In particular, $X^i(\omega)$ may be seen as an $i$-controlled path on the  
empirical rough set-up: If we use $\delta_{x}^{(n)}$ and
 $\delta_{\mu}^{(n)}$ as symbols for the Gubinelli derivatives 
 in 
Definition \ref{definition:omega:controlled:trajectory}
  but on the empirical rough set-up, then 
  $\delta_{x}^{(n)} X^i(\omega)$
  identifies with the standard Gubinelli 
  derivative
 in the expansion of 
 $X^i(\omega)$ along the variations of $\big(W^i(\omega),{\mathbb W}^i(\omega)\big)$
 and 
 $\delta_{\mu}^{(n)} X^\cdot(\omega) \equiv 0$.}
\smallskip

The key fact in our analysis lies in the interpretation of the two {derivatives} 
$$
{\delta_{x}^{(n)}} \Big[ \textrm{\rm F}(X^i(\omega),X^{\theta_{n}(\cdot)}(\omega) )\Big]
\quad \textrm{\rm and}
\quad
{\delta_{\mu}^{(n)} }\Big[\textrm{\rm F}(X^i(\omega),X^{\theta_{n}(\cdot)}(\omega) )\Big]
$$ 
in Proposition \ref{prop:chaining}. First, it is elementary to check that 
\begin{equation}
\label{eq:deltax:empirical}
\begin{split}
{\delta_{x}^{(n)}} \Bigl( \textrm{\rm F}\bigl(X^i(\omega),X^{\theta_{n}(\cdot)}(\omega)\bigr) \Bigr)_{t} &= \partial_x \textrm{\rm F}\bigl(X_{t}^i(\omega),X^{\theta_{n}(\cdot)}_{t}(\omega)\bigr)  
{\delta_{x}^{(n)}} X_{t}^i(\omega)   
\\
&= \partial_x \textrm{\rm F}\bigl(X_{t}^i(\omega),\mu^{n}_{t}(\omega)\bigr)  
{ \delta_{x}^{(n)}} X_{t}^i(\omega).
\end{split}
\end{equation}
More interestingly, we have
\begin{equation}
\label{eq:deltamu:empirical}
\begin{split}
&{\delta_{\mu}^{(n)}} \Bigl(\textrm{\rm F}\bigl(X^i(\omega),X^{\theta_{n}(\cdot)}(\omega)\bigr) \Bigr)_{t} 
\\
&\hspace{15pt} =
D_{\mu} \textrm{\rm F}\bigl(X_{t}^i(\omega),\mu^{n}_{t}(\omega)\bigr)\big(X_{t}^{\theta_{n}(\cdot)}(\omega)\big)\,  
{\delta_{x}^{(n)}} X_{t}^{\theta_{n}(\cdot)}(\omega),
\end{split}
\end{equation}
{both the left- and the right-hand sides being seen as random variables on $\{1,\cdots,n\}$. 
The realizations of the random variable in the right-hand side may be computed by replacing
the symbol
 $\cdot$ by 
$j \in \{1,\cdots,n\}$. }

{So, 
applying
\eqref{eq:remainder:integral}
with $\textrm{F}(X^i(\omega),\mu^n(\omega))$
as integrand}, the third term on the first line of 
\eqref{eq:remainder:integral}
here reads
\begin{equation*}
\begin{split}
\frac1n
\sum_{j=1}^n
D_{\mu} \textrm{\rm F}\bigl(X_{t}^{i}(\omega),\mu^n_{t}\bigr)\big(X_{t}^{j}(\omega)\big) \, {\delta_{x}^{(n)}} X_{t}^{j}(\omega)
{\mathbb W}^{j,i}_{t}(\omega).
\end{split}
\end{equation*}
This shows that the integral
$\int_{0}^t \textrm{\rm F} \Bigl(X^i_{s}(\omega),X_{s}^{\theta_{n}(\cdot)}(\omega) \Bigr) d{\boldsymbol W}_{s}^{(n)}(\omega)$, as defined by Theorem 
\ref{theorem:integral},
is the limit of the compensated Riemann sums
\begin{equation}
\label{eq:riemann:n}
\begin{split}
&\sum_{k=0}^{K-1} 
\biggl(
\textrm{\rm F} \bigl(X^i_{t_{k}}(\omega),X_{t_{k}}^{\theta_{n}(\cdot)}(\omega) \bigr)
W_{t_{k},t_{k+1}}^i(\omega)
\\
&\hspace{1pt}
+
\partial_x
\textrm{\rm F}\bigl(X_{t_{k}}^i(\omega),X^{\theta_{n}(\cdot)}_{t_{k}}(\omega)\bigr)  
\textrm{\rm F}\bigl(X_{t_{k}}^i(\omega),X^{\theta_{n}(\cdot)}_{t_{k}}(\omega)\bigr) 
{\mathbb W}_{t_{k},t_{k+1}}^i(\omega) 
\\
&\hspace{1pt}
+\frac1n
\sum_{j=1}^n
D_{\mu} \textrm{\rm F}\bigl(X_{t_{k}}^{{i}}(\omega),\mu^n_{t}(\omega)\bigr)(X_{t_{k}}^{j}(\omega)) 
\textrm{\rm F}\bigl(X_{t_{k}}^j(\omega),X^{\theta_{n}(\cdot)}_{t_{k}}(\omega)\bigr) 
{\mathbb W}_{t_{k},t_{k+1}}^{j,i}(\omega)
\biggr),
\end{split}
\end{equation}
as the mesh of the dissection $0=t_{0}<\cdots<t_{K}=t$ tends to $0$\footnote{In the second line,   
$\partial_{x} \textrm{\rm F}\bigl(X_{s}^i(\omega),X_{s}^{\theta_{n}(\cdot)}(\omega)\bigr)
\bigl( \textrm{\rm F}\bigl( X^i_{s}(\omega),X^{\theta_{n}(\cdot)}_{s}(\omega)\bigr) 
{\mathbb W}_{s,t}^i(\omega) \bigr)$
is understood 
as $\bigl(\sum_{\ell=1}^d \sum_{j,k=1}^m \partial_{x_{\ell}}
{\rm F}^{\iota,j}\bigl(X^i_{s}(\omega),X^{\theta_{n}(\cdot)}_{s}(\omega) \bigr) 
\bigl( \textrm{\rm F}^{\ell,k}
\bigl(X^i_{s}(\omega),X^{\theta_{n}(\cdot)}_{s}(\cdot)(\omega) \bigr) \bigl( {\mathbb W}_{s,t}^i \bigr)^{k,j}(\omega)
\bigr)\bigr)_{\iota=1,\cdots,d}$ and similarly for the term on the third line.
}. This allows to compare the latter quantity with \eqref{eq:particle:system} if we intepret the integral with respect to $W^i(\omega)$
therein as a rough integral with respect {to the enhanced setting above $(W^1(\omega),\cdots,W^n(\omega))$}, and consider the leading coefficient $\textrm{\rm F}(X_{t}^i(\omega),\mu^n_{t}(\omega))$ as a \textit{standard} Euclidean function of the tuple $X^{(n)}_{t}(\omega) = \big(X^1_{t}(\omega),\cdots,X^n_{t}(\omega)\big)$. Indeed, under the standing \textbf{\textsf{Regularity assumptions 1}} and  \textbf{\textsf{2}}, the function
\begin{equation*}
f^i : (\RR^d)^n \ni \big(x^1,\cdots,x^n\big) \mapsto \textrm{\rm F}\biggl(x^i,\frac1n \sum_{k=1}^n \delta_{x^k} \biggr)
\end{equation*}
is ${\mathcal C}^2$ with Lipschitz derivatives and
\begin{equation*}
\partial_{x^j} f^i\big(x^1,\cdots,x^n\big) = {\delta_{i,j}} \ \partial_{x} \textrm{\rm F}\biggl(x^i,\frac1n \sum_{k=1}^n \delta_{x^k} \biggr) + \frac1n D_{\mu} \textrm{\rm F}\biggl(x^i,\frac1n \sum_{k=1}^n \delta_{x^k} \biggr)({x^j}),
\end{equation*}
{with $\delta_{i,j}=1$ if $i=j$ and $0$ otherwise,}
see Chapter 5 in \cite{CarmonaDelarue_book_I}. Therefore, \eqref{eq:particle:system} is uniquely solvable in the classical sense and the above formulas for the derivatives show that the rough integral therein may be approximated by the same Riemann sum as in \eqref{eq:riemann:n}. 
{Namely, 
\eqref{eq:particle:system} may be rewritten as 
\begin{equation*}
\begin{split}
&\sum_{k=0}^{K-1} 
\biggl(
f^i \bigl(X^1_{t_{k}}(\omega),\cdots,X^n_{t_{k}}(\omega) \bigr) 
W_{t_{k},t_{k+1}}^i(\omega)
\\
&\hspace{15pt}
+
\sum_{j=1}^n
\partial_{x_{j}} f^i \bigl(X^1_{t_{k}}(\omega),\cdots,X^n_{t_{k}}(\omega) \bigr) 
{\mathbb W}_{t_{k},t_{k+1}}^{j,i}(\omega)
\biggr).
\end{split}
\end{equation*}}This proves that the solution to  
\eqref{eq:particle:system},
when the latter is seen as a rough differential equation driven by the enhanced setting above $(W^1(\omega),\cdots,W^n(\omega))$, 
coincides with 
the solution of the empirical version of  \eqref{EqRDE}, when the latter is understood as a mean field rough differential 
equation driven by the empirical rough set up.

\subsection{Propagation of Chaos}
\label{SubsectionPropagChaos}

We now have all the ingredients to prove that the {empiral {measure} of the} solution to {the particle system} \eqref{eq:particle:system} converges, in some sense, to the solution of the rough mean field equation \eqref{EqRDE}, when the rough set-up {therein} is interpreted as originally explained in Section \ref{SectionRoughStructure}. This {is what we call} \textit{propagation of chaos}. The statement takes the following form.

\begin{thm}
\label{theorem:prop:of:chaos}
We make the following assumptions.

\begin{enumerate}
   \item Let \emph{F} satisfy \textbf{\textsf{Regularity assumptions 1}}  and \textbf{\textsf{2}}.   
   
   \item Let $w$ be a control satisfying \eqref{eq:w:s:t:omega:ineq} and \eqref{eq:useful:inequality:wT} {for the same parameters $p \in [2,3)$ and $q \geq 8$ as in Section 
   \ref{SectionRoughStructure}}. Assume that, for a given positive time horizon $T$, the random variables $w(0,T,\textcolor{black}{\cdot})$ and 
$\bigl(N\big([0,T],\textcolor{black}{\cdot},\alpha\big)\bigr)_{\alpha >0}$, see \eqref{eq:N:s:t:omega}, have sub and super exponential tails, 
{see \eqref{EqTailAssumptions}}.

   \item Assume that the rough set-up ${\boldsymbol W}$ is strong. 
   
   \item Assume also that there exists a positive constant $\varepsilon_{1}$ such that 
\end{enumerate}
\begin{equation}
\label{eq:exp:integrability}
\begin{split}
&{\mathbb E}
\Bigl[
\exp\Bigl( 
 \bigl\|W(\cdot) \|^{\varepsilon_{1}}_{[0,T],(1/p)-\textrm{\rm H}}
 \Bigr)
 \Bigr]
 + 
 {\mathbb E}
\Bigl[
\exp\Bigl( 
 \bigl\|{\mathbb W}(\cdot) \|^{\varepsilon_{1}/2}_{[0,T],(2/p)-\textrm{\rm H}}
 \Bigr)
 \Bigr]
 \\
&\hspace{15pt} + 
 {\mathbb E}^{\otimes 2} 
\Bigl[
\exp\Bigl( 
 \bigl\|{\mathbb W}^{\indep}(\cdot,\cdot) \|^{\varepsilon_{1}/2}_{[0,T],(2/p)-\textrm{\rm H}}
 \Bigr)
 \Bigr] < \infty.
\end{split}
\end{equation} 

Then, for almost every $\omega \in \Omega$,
\begin{equation}
\label{eq:convergence:empirical:measure}
\frac1n \sum_{i=1}^n \delta_{X^{i,(n)}(\omega)} \rightarrow {\mathcal L} \bigl( X(\cdot)\bigr),
\end{equation}
where $X^{(n)}(\omega)= {(X^{i,(n)}(\omega))_{i=1,\cdots,n}}$ is the solution to \eqref{eq:particle:system} and $X(\cdot)$ is the solution to \eqref{EqRDE}, the convergence being the convergence in law on ${\mathcal C}\big([0,T];\RR^d\big)$. Moreover, for any fixed $k \geq 1$, the law of 
$
\big(X^{1,(n)}(\cdot),\cdots,X^{k,(n)}(\cdot)\big) 
$ 
converges to ${\mathcal L}\big(X(\cdot)\big)^{\otimes k}$. 
\end{thm}

\begin{rem}
Before we prove the above statement, we feel useful to make the following comments:
\begin{enumerate}
	\item[$\textcolor{gray}{\bullet}$]  In item $(b)$ of the assumption, we can always assume that $w$ is in fact given by the \emph{natural} control \eqref{eq:w:s:t:omega}.

	\item[$\textcolor{gray}{\bullet}$]  The argument used below in Step 3 of the proof would show that item $(d)$ implies the part related to $w$ in item $(b)$ at least when $w$ is chosen as in \eqref{eq:w:s:t:omega}. Despite this form of redundancy, we feel better to keep the current 
formulation
(of the assumptions) as it is consistent with the rest of the text. 

	\item[$\textcolor{gray}{\bullet}$] 
{Following 
\cite[Theorem 2.4]{BCD1}, the above assumptions hold true
 for Gaussian rough paths subject to the classical conditions of Friz-Victoir \cite{FrizVictoirGaussian}, see Example \ref{Gaussian}, and the related Example 2.2 in \cite{BCD1}. 
}
\end{enumerate}
\end{rem}
\begin{proof}
The key tool for passing to the limit is the continuity Theorem \ref{theoremContinuity}, but with
{$p$ therein replaced by some $p' \in (p,3)$}. 
The main difficulty is in controlling the accumulated local variation of the empirical rough set-up.
To make the notations clear, we write 
$X_{0}^{i,(n)}$
for $X^i$, 
$W^{i,(n)}$ for $W^i$, ${\mathbb W}^{i,(n)}$ for ${\mathbb W}^{i}$ and ${\mathbb W}^{i,j,(n)}$ for ${\mathbb W}^{i,j}$. 
\medskip

\textbf{\textsf{Step 1.}} As a starting point, we want to prove that, for almost every $\omega \in \Omega$, for any $\alpha >0$, there exists a constant $\varepsilon_{2}>0$
such that, for all $n\geq 1$, 
\begin{equation}   \label{EqExpConditionLocalVariation}
\sup_{n \geq 1} \frac1n \sum_{i=1}^n \exp \Bigl( N^{i,n}(0,T,\omega,\alpha)^{1+\varepsilon_{2}}
\Bigr) < \infty,
\end{equation}
where $N^{i,n}(0,T,\omega,\alpha)$ is defined as the local accumulation 
\begin{equation}
\label{eq:N:in:N:varpi}
N^{i,n}([0,T],\omega,\alpha) := N_{\varpi}([0,T],\alpha),
\end{equation}
when $\varpi(s,t) =  v^{i,n}_{p'}(s,t,\omega)^{1/p'}$, see \eqref{eq:N:s:t:omega} {and 
\eqref{eq:v:N}}.   
Following 
\eqref{eq:N1:N2} in appendix (see also the longer discussion in the introduction of the appendix in 
\cite{BCD1}), it suffices to prove \eqref{EqExpConditionLocalVariation}
when $\varpi$ in the definition of $N^{i,n}$ is equal to each of the terms in the right-hand side of 
\eqref{eq:v:N}.

When 
$\varpi(s,t) = \big\| W^i(\omega) \big\|_{[s,t],p'-\textrm{\rm v}}$
or  
$\varpi(s,t) = \big\| \WW^i(\omega) \big\|_{[s,t],p'/2-\textrm{\rm v}}^{1/2}$, the resulting 
variables $\bigl(N^{i,n}([0,T],\omega,\alpha)\bigr)_{i=1,\cdots,n}$ in \eqref{eq:N:in:N:varpi} are independent and identically distributed, their common law being independent of $n$. Then, \eqref{EqExpConditionLocalVariation} follows from assumption (b) in the statement and from the law of large numbers ({using the well-known fact that the 
$p'$-variation is less than or equal to the $p$-variation if $p'>p$}). 

If
$\varpi(s,t) = {{}^{(n)} \hspace{-1pt} \big\lgroup} {W}^{\bullet}(\omega) \big\rgroup_{q ; [s,t],p'-\textrm{\rm v}}$
or
$\varpi(s,t) =   {}^{(n)} \hspace{-1pt}\big\lgroup \hspace{-2pt}\big\lgroup
  \WW^{\bullet,\bullet}(\omega) \big\rgroup \hspace{-2pt} \big\rgroup_{q ; [s,t],p'/2-\textrm{\rm v}}^{1/2}$,
  the resulting 
variables $\bigl(N^{i,n}([0,T],\omega,\alpha)\bigr)_{i=1,\cdots,n}$ in \eqref{eq:N:in:N:varpi} 
only depend on $n$. We may denote them by 
$N^{n}([0,T],\omega,\alpha)$. Then, it suffices to prove that, for any $\alpha >0$, 
$\limsup_{n \rightarrow \infty} N^n([0,T],\omega,\alpha)$ is almost surely finite. 
By 
\eqref{eq:v:p:i:N}, we may easily control $N^n([0,T],\omega,\alpha)$ from above by noticing that 
\begin{equation*}
\begin{split}
&\alpha^p N^n([0,T],\omega,\alpha) 
\\
&\hspace{15pt} 
\leq 
c \Bigl( {}^{(n)} \hspace{-1pt} \bigl\lgroup 
\bigl\|
W^{\bullet}(\omega) \bigr\|_{[0,T],(1/p)-\textrm{\rm H}}^{p}
 \bigr\rgroup_{q}  
 +
  {}^{(n)} \hspace{-1pt}\big\lgroup \hspace{-2pt}\big\lgroup
 \bigl\|
 {\mathbb W}^{\bullet,\bullet}(\omega) \bigr\|_{[0,T],(2/p)-\textrm{\rm H}}^{p/2}   
   \big\rgroup \hspace{-2pt} \big\rgroup_{q}
\Bigr),
\end{split}
\end{equation*}
for a constant $c$ that is independent of $n$ and $\omega$. 
Proceeding as in 
\eqref{eq:lln:holder}, the result follows again from the law of large of numbers and from assumption (b). 

In fact, the most difficult cases are 
$\varpi(s,t) =
{{}^{(n)} \hspace{-1pt} \big\lgroup} \WW^{i,\bullet}(\omega) \big\rgroup_{q ; [s,t],p'/2-\textrm{\rm v}}^{1/2}$
or 
$\varpi(s,t) = {{}^{(n)} \hspace{-1pt} \big\lgroup} \WW^{\bullet,i}(\omega) \big\rgroup_{q ; [s,t],p'/2-\textrm{\rm v}}^{1/2}$. By symmetry, it suffices to treat the first one. 
And, by changing in an obvious manner {the parameter $\alpha$}, we may just focus on 
$\varpi(s,t) =
{{}^{(n)} \hspace{-1pt} \big\lgroup} \WW^{i,\bullet}(\omega) \big\rgroup_{q ; [s,t],p'/2-\textrm{\rm v}}^{q}$.
Then,
\begin{equation}
\label{eq:decomposition:WWibullet}
\begin{split}
&{{}^{(n)} \hspace{-1pt} \big\lgroup} \WW^{i,\bullet}_{s,t}(\omega) \big\rgroup_{q}^q
= \frac1n \sum_{j=1}^n \Bigl( 
\bigl\vert \WW^{i,j}_{s,t}(\omega) \bigr\vert^q
- 
 \bigl\langle \WW_{s,t}^{i,\indep}(\omega,\cdot) \bigr\rangle_{q}^q 
 \Bigr) + 
 \bigl\langle \WW_{s,t}^{i,\indep}(\omega,\cdot) \bigr\rangle_{q}^q.
 \end{split}
\end{equation} 
 Now, 
 Rosenthal's inequality (see \cite{rosenthal1972}) together with 
 \eqref{eq:exp:integrability} say that, 
for any $a \geq 2$ and any $i \in \{1,\cdots,n\}$, 
\begin{equation*}
\int_{\Omega}
 \biggl\vert
 \frac1n \sum_{j=1}^n \Bigl( 
\bigl\vert \WW_{s,t}^{i,j}(\omega) \bigr\vert^q
- 
 \bigl\langle \WW_{s,t}^{i,\indep}(\omega,\cdot) \bigr\rangle_{q}^q 
 \Bigr) \biggr\vert^{a}
 d {\mathbb P}(\omega) \leq C_{a} n^{-a/2} \vert t-s \vert^{2a q/p},
 \end{equation*}
 for a constant $C_{a}$ depending on $a$ and on the upper bound for 
 the left-hand side in 
  \eqref{eq:exp:integrability}, but independent of $i$, $n$ and $(s,t)$.  Letting $(t_{k}^{(n)} = k T/n)_{k=0,\cdots,n}$ and allowing the constant 
  $C_{a}$ to vary from line to line, we deduce that 
\begin{equation*}
\sum_{1 \leq k < \ell \leq n}
\sum_{i=1}^n 
\int_{\Omega}  \biggl\vert
 \frac1n \sum_{j=1}^n \Bigl( 
\bigl\vert \WW_{t_{k}^{(n)},t_{\ell}^{(n)}}^{i,j}(\omega) \bigr\vert^q
- 
 \bigl\langle \WW_{t_{k}^{(n)},t_{\ell}^{(n)}}^{i,\indep}(\omega,\cdot) \bigr\rangle_{q}^q 
 \Bigr) \biggr\vert^{a}
 d {\mathbb P}(\omega) \leq C_{a} n^{3-a/2}.
 \end{equation*}  
We deduce from Markov inequality  that, for any $a \geq 2$ and any 
 $n \geq 1$,
  \begin{equation}
  \begin{split}
&{\mathbb P} \biggl(  \max_{1 \leq i \leq n}
\max_{1 \leq k < \ell \leq n} 
   \biggl\vert
 \frac1n \sum_{j=1}^n \Bigl( 
\bigl\vert \WW_{t_{k}^{(n)},t_{\ell}^{(n)}}^{i,j}(\omega) \bigr\vert^q
- \bigl\langle \WW_{t_{k}^{(n)},t_{\ell}^{(n)}}^{i,\indep}(\omega,\cdot) \bigr\rangle_{q}^q 
 \Bigr) \biggr\vert 
\\
&\hspace{15pt} \geq n^{-1/4} \biggr)
  \leq C_{a} n^{3-a/4}.
  \label{eq:borell:cantelli:1}
 \end{split}
  \end{equation}
By item (d) in the statement, we also have, for any $n \geq 1$, $a \geq 2$ and $\delta >0$,  
 \begin{equation}
   \label{eq:borell:cantelli:2:b}
\begin{split}
& {\mathbb P} \biggl( 
 \max_{1 \leq i,j \leq n}
\Bigl(   \bigl\|W^{i} \|^2_{[0,T],(1/p)-\textrm{\rm H}} 
+\bigl\|{\mathbb W}^{i,j} \|_{[0,T],(2/p)-\textrm{\rm H}} 
\Bigr) 
\geq n^{{\delta/2}}  
 \biggr) 
 \\
 &\leq C_{a} n^{2- a \delta/2},
 \end{split}
\end{equation}
which implies, at least for $n$ large enough ({below, we absorb the constant $T$ that appears in the length of the increments
by changing $n^{\delta/2}$ into $n^{\delta}$, which is indeed possible since $n$ is large, it being understood that the lower threshold for $n$ only depends on $\delta$ and $T$})
 \begin{equation}
   \label{eq:borell:cantelli:2}
\begin{split}
&{\mathbb P} \biggl( \max_{1 \leq i,j \leq n}  \sup_{\vert s-t \vert \leq {T}/n} \bigl|{\mathbb W}_{s,t}^{i,j}(\omega)
\bigr\vert 
\geq  n^{\delta-2/p} \biggr) \leq C_{a} n^{2-a \delta/2}, 
\\
&{\mathbb P} \biggl( \max_{1 \leq i,j \leq n}
\sup_{\min( \vert s-t \vert, \vert s'-t' \vert) \leq {T}/n}
\Bigl( \bigl\vert 
W^{i}_{s,t}(\omega) 
\bigr\vert
\bigl\vert 
W^{j}_{s',t'}(\omega) 
\bigr\vert
\Bigr) 
 \geq n^{\delta-1/p}\biggr)
 \\
 &\hspace{15pt} \leq C_{a} n^{2-a \delta/2}.
\end{split}
\end{equation} 
Similarly, we have (for the same ranges of values for $a$, $\delta$ and $n$)  
 \begin{equation}
   \label{eq:borell:cantelli:3:b}
 \begin{split}
   &
   {\mathbb P}
   \biggl( 
   \max_{1 \leq i  \leq n}  \bigl\| \bigl\langle {\mathbb W}^{i,\indep}(\omega,\cdot) \bigr\rangle_{q} \|_{[0,T],(2/p)-\textrm{\rm H}} 
   \geq n^{{\delta/2}} \biggr) \leq C_{a} n^{1-a \delta/2},
\end{split}
\end{equation} 
which implies 
 \begin{equation}
 \label{eq:borell:cantelli:3}
 \begin{split}
&   {\mathbb P}
   \biggl(  \max_{1 \leq i \leq n} \, 
\sup_{\vert s-t \vert \leq {T}/n}
\bigl\langle \WW^{i,\indep}_{s,t}(\omega,\cdot) \bigr\rangle_{q} 
\geq n^{\delta-2/p} \biggr) \leq C_{a} n^{1-a \delta/2},
\\
& {\mathbb P}
   \biggl( \max_{1 \leq i \leq n} \,
\sup_{\min( \vert s-t \vert, \vert s'-t' \vert) \leq {T}/n}
\Bigl( 
\bigl\vert 
W^{i}_{s,t}(\omega) 
\bigr\vert
\, 
\bigl\langle W_{s',t'}(\cdot) 
\bigr\rangle_{q} 
\Bigr) 
 \geq n^{\delta-1/p} \biggr)
 \\
&\hspace{15pt} \leq C_{a} n^{1-a \delta/2}.
\end{split}
\end{equation}
Using Chen's relations 
\eqref{eq:chen}
to write ${\mathbb W}_{s,t}^{i,j} := - {\mathbb W}_{\{s\},s}^{i,j}+
{\mathbb W}_{\{s\},\{t\}}^{i,j}
+
{\mathbb W}_{\{t\},t}^{i,j} + W_{s,\{t\}}^i \otimes W_{\{t\},t}^j
- W_{\{s\},s}^i \otimes W_{s,\{t\}}^j$
(with $\{s\} :=  \lfloor ns/T \rfloor T/n$), 
we can find a constant $c_{q}$ only depending on $q$ ({but the value of which is allowed to change from line to line}) such that, 
for any $(s,t) \in {\mathcal S}_{2}^T$, 
\begin{equation*}
\begin{split}
&\biggl\vert \bigl\vert \WW_{s,t}^{i,j}(\omega) \bigr\vert^q
-
\bigl\vert \WW_{\{s\},\{t\}}^{i,j}(\omega) \bigr\vert^q
\biggr\vert 
\leq c_{q} \Xi^{i,j} 
\sup_{(s',t') \in {\mathcal S}_{2}^T}
\bigl\vert \WW_{s',t'}^{i,j}(\omega) \bigr\vert^{q-1}
\\
&\textrm{\rm with} 
\\
&\hspace{5pt} \Xi^{i,j}:=
\sup_{\vert s' - t' \vert \leq T/n}
\bigl\vert \WW_{s',t'}^{i,j}(\omega) \bigr\vert
+
\sup_{\min( \vert s'-t' \vert, \vert s''-t'' \vert) \leq {T}/n}
\Bigl( \bigl\vert 
W^{i}_{s',t'}(\omega) 
\bigr\vert 
\bigl\vert 
W^{j}_{s'',t''}(\omega) 
\bigr\vert \Bigr),
\end{split}
\end{equation*}
from which, together with
\eqref{eq:borell:cantelli:2:b}
and   \eqref{eq:borell:cantelli:2}, we deduce that 
\begin{equation*}
\begin{split}
&{\mathbb P}
\biggl( \max_{1 \leq i,j \leq n}
\sup_{(s,t) \in {\mathcal S}_{2}^T}
\Bigl\vert \bigl\vert \WW_{s,t}^{i,j}(\omega) \bigr\vert^q
-
\bigl\vert \WW_{\{s\},\{t\}}^{i,j}(\omega) \bigr\vert^q
\Bigr\vert 
\geq c_{q} n^{\delta q - 1/p} \biggr) \leq C_{a} n^{2-a \delta/2}. 
\end{split}
\end{equation*}
Proceeding similarly with 
$\bigl\langle \WW_{s,t}^{i,\indep}(\omega,\cdot) \bigr\rangle_{q}^q$
and inserting \eqref{eq:borell:cantelli:1}, we finally obtain, for $a \geq 2$, 
$\delta >0$ and $n$ large enough (in terms of $T$ and $\delta$ only) 
and for a possibly new value of $c_{q}$,
  \begin{equation}
  \label{eq:borell:cantelli:4}
  \begin{split}
&{\mathbb P}
\biggl( 
 \max_{1 \leq i \leq n}
\sup_{(s,t) \in {\mathcal S}_{2}^T}
  \biggl\vert
 \frac1n \sum_{j=1}^n \Bigl( 
\bigl\vert \WW_{s,t}^{i,j}(\omega) \bigr\vert^q
- 
 \bigl\langle \WW_{s,t}^{i,\indep}(\omega,\cdot) \bigr\rangle_{q}^q 
 \Bigr) \biggr\vert 
\\
&\hspace{15pt} \geq c_{q}  {\Bigl(  n^{- 1/4} + n^{\delta q-1/p} \Bigr)} \biggr) \leq C_{a} \Bigl( n^{3-a/4} + 
n^{2-a\delta /2} \Bigr).
 \end{split}
  \end{equation}
 Meanwhile, we also have, for ${\mathbb P}$-almost every $\omega \in \Omega$,  
  \begin{equation*}
 \begin{split}
&  \biggl\vert
 \frac1n \sum_{j=1}^n \Bigl( 
\bigl\vert \WW_{s,t}^{i,j}(\omega) \bigr\vert^q
- 
 \bigl\langle \WW_{s,t}^{i,\indep}(\omega,\cdot) \bigr\rangle_{q}^q 
 \Bigr) \biggr\vert 
  \\
  &\leq 
  \max_{1 \leq i,j \leq n} 
  \Bigl(  \bigl\|{\mathbb W}^{i,j}(\omega) \|^q_{[0,T],(2/p)-\textrm{\rm H}} 
    +
     \bigl\| \bigl\langle {\mathbb W}^{i,\indep}(\omega,\cdot) \bigr\rangle_{q} \|_{[0,T],(2/p)-\textrm{\rm H}}^q
  \Bigr) 
  (t-s)^{{2q/p}}. 
  \end{split}
  \end{equation*}
By   \eqref{eq:borell:cantelli:2:b}
and 
   \eqref{eq:borell:cantelli:3:b}, we deduce that 
   (again for the same ranges of values for $a$, $\delta$ and $n$)
  \begin{equation}
  \label{eq:borell:cantelli:4:b}
 \begin{split}
&{\mathbb P} \biggl( \max_{1 \leq i \leq n} \sup_{(s,t) \in {\mathcal S}_{2}^T}  \biggl\vert
 \frac1n \sum_{j=1}^n \Bigl( 
\bigl\vert \WW_{s,t}^{i,j}(\omega) \bigr\vert^q
- 
 \bigl\langle \WW_{s,t}^{i,\indep}(\omega,\cdot) \bigr\rangle_{q}^q 
 \Bigr) \biggr\vert 
 \\
&\hspace{15pt}  \geq n^{\delta q} (t-s)^{{2q/p}} \biggr)
   \leq C_{a} n^{2-a \delta/2}. 
  \end{split}
  \end{equation}   
Taking the power $1-p/p'$ in
the expression underpinning the event in 
    \eqref{eq:borell:cantelli:4}
    and, similarly, the power $p/p'$ in the expression underpinning the event in  \eqref{eq:borell:cantelli:4:b}, cross-multiplying both expressions,
     we get,
for $n$ large enough and for $\delta \in (0,1/4)$,
  \begin{equation*}
 \begin{split}
&{\mathbb P} \biggl( \max_{1 \leq i \leq n} \sup_{(s,t) \in {\mathcal S}_2^{T}} \biggl\vert
 \frac1n \sum_{j=1}^n \Bigl( 
\bigl\vert \WW_{s,t}^{i,j}(\omega) \bigr\vert^q
- 
 \bigl\langle \WW_{s,t}^{i,\indep}(\omega,\cdot) \bigr\rangle_{q}^q 
 \Bigr) \biggr\vert 
\\
&\hspace{15pt}  \geq c_{p,p',q} 
   {\bigl(  n^{- (p'-p)/(4p')+\delta qp/p'} + n^{\delta q-(p'-p)/(pp')} \bigr)}
     (t-s)^{2q/p'} \biggr) 
\\
&\leq C_{a}\Bigl( n^{3-a/4}+ n^{2-a \delta/2} \Bigr), 
  \end{split}
  \end{equation*}   
  where $c_{p,p',q}$ only depends on $p$, $p'$ and $q$. 
  Choosing $\delta$ small enough in terms of $p$, $p'$ and $q$, we deduce that 
for $n$ large enough, 
  \begin{equation*}
 \begin{split}
&{\mathbb P} \biggl( 
\max_{1 \leq i \leq n}
\sup_{(s,t) \in {\mathcal S}_2^{T}}  \biggl\vert
 \frac1n \sum_{j=1}^n \Bigl( 
\bigl\vert \WW_{s,t}^{i,j}(\omega) \bigr\vert^q
- 
 \bigl\langle \WW_{s,t}^{i,\indep}(\omega,\cdot) \bigr\rangle_{q}^q 
 \Bigr) \biggr\vert \geq c_{p,p',q} 
     (t-s)^{2q/p'} \biggr)
     \\
&      \leq C_{a} \Bigl( n^{3-a/4}+ n^{2-a \delta/2} \Bigr). 
  \end{split}
  \end{equation*}   
Back to 
\eqref{eq:decomposition:WWibullet}, we deduce that, 
for
$\delta$ as in the previous inequality (assuming also without any loss of generality that 
$\delta \in (0,1/2)$), 
$a$ large enough (in terms of $\delta$, but $\delta$ depending on $p$, $p'$ and $q$)
and then
 $n$ large enough (in terms of $\delta$ and $T$) and also for a new value of $c_{p,p',q}$, 
  \begin{equation}
  \label{eq:conclusion:supplementaire}
 \begin{split}
&{\mathbb P} \biggl( \forall i \in \{1,\cdots,n\}, \ \forall (s,t) \in {\mathcal S}_{2}^T, \quad {{}^{(n)} \hspace{-1pt} \big\lgroup} \WW^{i,\bullet}(\omega) \big\rgroup_{q;[s,t],p'/2-\textrm{\rm v}}
\\ 
&\hspace{15pt} \leq  c_{p,p',q} (t-s)^{2/p'}
 + \big\langle \WW^{i,\indep}(\omega,\cdot) \big\rangle_{q ; [s,t],{p'}/2-\textrm{\rm v}}\biggr) 
\geq 1 - C_{a} n^{-a \delta/4}.
  \end{split}
  \end{equation}   
By Borel-Cantelli lemma, we deduce that, 
for ${\mathbb P}$-almost every $\omega \in \Omega$, 
for $n$ large enough, 
for any $i \in \{1,\cdots,n\}$ and any  $(s,t) \in {\mathcal S}_{2}^T$, 
\begin{equation*}
\begin{split}
&{{}^{(n)} \hspace{-1pt} \big\lgroup} \WW^{i,\bullet}(\omega) \big\rgroup_{q;[s,t],p'/2-\textrm{\rm v}} 
\leq  c_{p,p',q} (t-s)^{2/p'}
 + \big\langle \WW^{i,\indep}(\omega,\cdot) \big\rangle_{q ; [s,t],{p'}/2-\textrm{\rm v}}.
 \end{split}
\end{equation*}
Since the variables 
$\bigl( \WW^{i,\indep}\bigr)_{i \geq 1}$
are independent, 
the local accumulation associated with the second term in the right-hand side may be handled like the local accumulation associated to 
$\varpi(s,t) = \big\| W^i(\omega) \big\|_{[s,t],{p'}-\textrm{\rm v}}$.
The local accumulation associated with the first term is easily handled. 
\medskip

\textbf{\textsf{Step 2.}}  {Now},   from the law of large numbers (see
Lemma 
\ref{lem:LLN:2nd order} for the law of large numbers with second order interaction terms)
and from \cite[Theorem 2.3 and Problem 3.1]{billing}, we deduce that there exists a full subset $E \subset \Omega$ (the definition of which may vary from line to line in the rest of the proof as long as ${\mathbb P}(E)$ remains equal to $1$) such that, for any $\omega \in E$,  
\begin{equation*}
\begin{split}
\pi_{n}(\omega) = 
\biggl( \frac1{n^2}
\sum_{i,j=1}^n 
\delta_{\big(X_{0}^{i,(n)}(\omega),W^{i,(n)}(\omega),{\mathbb W}^{i,(n)}(\omega),{\mathbb W}^{i,j,(n)}(\omega)\big)}
\biggr)_{n \geq 1}
\end{split}
\end{equation*}
converges in the weak sense to $\big(X_{0}(\cdot),W(\cdot),{\mathbb W}(\cdot),{\mathbb W}^{\indep}(\cdot,\cdot)\big)$ on the space
$\RR^d \times {\mathcal C}\big([0,T];\RR^m\big)\times \big\{{\mathcal C}({\mathcal S}_{2}^T;\RR^m \otimes \RR^m)\big\}^2.$

\medskip

\textbf{\textsf{Step 3.}}  Back to the statement of Theorem \ref{theoremContinuity}, the first item in the statement is a  consequence of the law of large numbers. As for the fourth item, it follows directly from the previous step. {In order to check the check the second and third items}, we now have a look at $v^{i,n}_{p'}(s,t,\omega)$ in \eqref{eq:v:N}. Following \eqref{eq:lln:holder}, we already know that 
\begin{equation*}
\begin{split}
&\limsup_{n \geq 1}
\sup_{0 \leq s < t \leq T}
\frac{{}^{(n)} \hspace{-1pt} \bigl\lgroup v_{p'}^{\bullet,n}(s,t,\omega)
\bigr\rgroup_{2q}}{t-s} < \infty,
\end{split}
\end{equation*}
which proves the third item in the statement of Theorem \ref{theoremContinuity}. We end up with the proof of the second item. 
{Following} \eqref{eq:v:p:i:N}, there exists a constant $c'$ such that, for any $\varepsilon >0$, the quantity 
\begin{equation}
\label{eq:empirical:exp:v}
\sup_{n \geq 1}
{}^{(n)} \hspace{-3pt} \Bigl\lgroup
\exp\left( [v^{\bullet,n}_{p'}(0,T,\omega)]^{\varepsilon} \right)
\hspace{-2pt}
\Bigr\rgroup_{\hspace{-3pt} 1}
\end{equation}
is finite if ({notice that we could work below with $1/p'$ instead of $1/p$--H\"older norms, but $1/p$ obviously suffices and is in fact more adapted to the assumption 
\eqref{eq:exp:integrability}})
\begin{equation}
\label{eq:label:demande:referee}
\begin{split}
&\sup_{n \geq 1}
\frac1n 
\sum_{i=1}^n
\exp\left( 
{c'} \bigl\|
W^i(\omega) \bigr\|_{[0,T],(1/p)-\textrm{\rm H}}^{p' \varepsilon}+  {c'} \, \bigl\|{\mathbb W}^i(\omega) \bigr\|_{[0,T],(2/p)-\textrm{\rm H}}^{p' \varepsilon/2} \right) < \infty,   \\
&\sup_{n \geq 1} \frac1n \sum_{i=1}^n \exp\Bigl(  {c'} \, {}^{(n)} \hspace{-3pt} \Bigl\lgroup \bigl\| {\mathbb W}^{\bullet,i}(\omega) \bigr\|_{[0,T],(2/p)-\textrm{\rm H}}^{p'/2} \hspace{-2pt} \Bigr\rgroup_{\hspace{-3pt} q}^{\hspace{-3pt} \varepsilon } \Bigr) < \infty,   
\end{split}
\end{equation}
 {and similarly on the second line {of 
\eqref{eq:label:demande:referee}}
  with ${\mathbb W}^{\bullet,i}(\omega)$ replaced by
${\mathbb W}^{i,\bullet}(\omega)$} {or ${\mathbb W}^{\bullet,\bullet}(\omega)$ (the last two terms appearing on the last line of \eqref{eq:v:p:i:N})}.
By the law of large numbers, the first line holds true on a full event if ${p' \varepsilon} < \varepsilon_{1}$. As for the second one, we use the following trick. Notice that the function 
\begin{equation}
\label{eq:empirical:exp:v:000}
(0,+\infty) \ni x \mapsto \exp \bigl(  x^{\varepsilon/q} \bigr),
\end{equation}
is convex on $[A_{\varepsilon},\infty)$, for some $A_{\varepsilon} >0$. Therefore, 
Jensen's inequality says that,
in order to check the {second line in \eqref{eq:label:demande:referee}}, 
it suffices to prove that 
\begin{equation}
\label{eq:empirical:exp:v:001}
\begin{split}
&\sup_{n \geq 1}
\frac1{n^2} \sum_{i,j=1}^n
\exp\Bigl[ 
\Bigl(A_{\varepsilon}^{\varepsilon/q}
\vee
 \bigl\|
{\mathbb W}^{i,j}(\omega) \bigr\|^{p' \varepsilon/2}_{[0,T],(2/p)-\textrm{\rm H}}
\Bigr)
\Bigr] < \infty,
\end{split}
\end{equation} 
and similarly for the {two terms appearing on the last line of \eqref{eq:v:p:i:N}}. Obviously, under the standing assumption, the latter holds true with probability 1 provided
${p' \varepsilon} < \varepsilon_{1}$. This proves \eqref{eq:empirical:exp:v}. In the statement of Theorem \ref{theoremContinuity}, this proves the condition related to the tails of $w^n$ by a standard application of Markov inequality. 
\smallskip

{The bound on the local accumulation in the second item of Theorem \ref{theoremContinuity} follows from the first step of the proof.}
\medskip

\textbf{\textsf{Step 4.}} 
By Theorem \ref{theoremContinuity}, we get 
\eqref{eq:convergence:empirical:measure} on a set of full measure.
%
By Proposition 2.2 in \cite{Sznitman}, we deduce that, for any fixed $k \geq 1$, the law of $\big(X^{1,(n)},\cdots,X^{k,(n)}\big)$ converges to ${\mathcal L}\big(X(\cdot)\big)^{\otimes k}$.  
\end{proof}

%
%

\begin{rem}
{Recently, the authors in \cite{CoghiDeuschelFrizMaurelli} obtained a quantified propagation of chaos result for mean field stochastic equations with additive noise
\begin{equation}
\label{EqCoghiEtCo}
dx_t = b\big(x_t,\mathcal{L}(x_t)\big)dt + dw_t, \quad x_0=\zeta,
\end{equation}
for a random path $w\in {\mathcal C}\big([0,T],\RR^d\big)$ subject to mild integrability condition, and random initial condition $\zeta$. There is no need of rough paths theory to make sense of this equation and solve it by elementary means, under proper regularity assumptions on the drift $b$. Its distribution is even a Lipschitz function of the distribution of $(\zeta,w)$, in $p$-Wasserstein metric. Using Tanaka's trick, 
this continuity result entails a propagation of chaos result. The global Lipscthiz continuity of the solution map $\mathcal{L}(w,\zeta)\mapsto\mathcal{L}(x)$ ensures in particular a quantitative convergence rate for the particle system no greater than the corresponding convergence rate for the sample empirical mean of the driving noises, which is optimal. We get back such a sharp estimate in the present, much more complicated, setting in the next section. Note that the global Lipschitz bound satisfied by the natural map $\Phi$ giving the solution to equation \eqref{EqCoghiEtCo} as a fixed point of $\Phi$ actually allows to deal with reflected dynamics, as the bounded variation part needed for the reflection happens to be a Lipschitz function of the non-reflected path, in Skorokhod formulation of the problem. We do not have such a strong continuity result for our solution map; see Theorem \ref{theoremContinuity}. See also the previous work \cite{FrizMaurelli} of the authors.   }
\end{rem}


\section{Rate of Convergence}
\label{SubsectionConvergenceRate}

The goal of this section is to elucidate the rate of convergence in the convergence result stated in Theorem \ref{theorem:prop:of:chaos}. 

The analysis is based upon a variation of Sznitman's original coupling argument, see \cite{Sznitman}. To make its principle clear, we recall that, on the space $(\Omega,{\mathcal F},\PP)$, the triples 
$
\big(X_{0}^1(\cdot),W^1(\cdot),{\mathbb W}^1(\cdot)\bigr),\cdots,\bigl(X_{0}^n(\cdot),W^n(\cdot),{\mathbb W}^n(\cdot)\big)
$ 
are $n$ independent copies of the original triple $\big(X_{0}(\cdot),W(\cdot),{\mathbb W}(\cdot)\big)$. For each $i \in \{1,\cdots,n\}$, the pair $\big(W^i(\cdot),{\mathbb W}^i(\cdot)\big)$ is completed into a rough set-up 
\begin{equation}
\label{eq:rough:set-up:barwi}
\begin{split}
&\overline {\boldsymbol W}^i(\cdot) := \big(W^i(\cdot),{\mathbb W}^i(\cdot),{\mathbb W}^{i,\indep}(\cdot,\cdot)\big), 
\\
&{\mathbb W}^{i,\indep}(\omega,\omega') = {\mathcal I}\bigl( W^i(\omega), W^i(\omega')\bigr), \quad (\omega,\omega') \in \Omega^2.
\end{split}
\end{equation}
Here we put a bar on the symbol $\overline{\boldsymbol W}^i$ in order to distinguish it from the finite-dimensional rough set-up ${\boldsymbol W}^{(n)}(\omega)$ that lies above $\big(W^1(\omega),\cdots,W^n(\omega)\big)$. In comparison, the second-order level of ${\boldsymbol W}^{(n)}$ is made of $({\mathbb W}^i)_{1 \leq i \leq n}$ and of $\big({\mathbb W}^{i,j}={\mathcal I}(W^i,W^j)\big)_{1 \leq i \not = j \leq n}$, see \eqref{eq:W:2:empirical}. To make the notations more homogeneous, we sometimes write {${\mathbb W}^{i,i}(\omega)$ for ${\mathbb W}^i(\omega)$}.

\smallskip

With each $\bigl(X_{0}^i(\cdot),\overline{\boldsymbol W}^i(\cdot)\bigr)$, we associate the corresponding solution $\overline{X}^i(\cdot)$ to the mean field equation \eqref{EqRDE}. The {$5$}-tuples 
$$
\Omega \ni \omega \mapsto  \Big(X_{0}^i(\omega),W^i(\omega),{\mathbb W}^i(\omega),{\mathbb W}^{i,\indep}(\cdot,\omega),\overline X^i(\omega)\Big)_{1 \leq i \leq n}
$$ 
are independent and identically distributed, $\Omega \ni \omega \mapsto \bigl({\mathbb W}_{t}^{i,\indep}(\cdot,\omega)\bigr)_{0 \leq t \leq T}$ being regarded as a process with values in $\LL^q(\Omega,{\mathcal F},\PP;\RR^d)$. 
Recalling that $X^{(n)}(\omega)=
\big(X^{1,(n)}({\omega}),\cdots,X^{{n},(n)}({\omega})\big)$ is the solution to \eqref{eq:particle:system}, we then let
\begin{equation}
\label{eq:empirical:measures}
\mu^n_{t}(\omega) = \frac1n \sum_{i=1}^n \delta_{X_{t}^{i,(n)}(\omega)}, 
\quad 
\overline{\mu}^n_{t}(\omega) = \frac1n \sum_{i=1}^n \delta_{\overline X_{t}^i(\omega)}, \quad t \in [0,T], \quad \omega \in \Omega.
\end{equation}

Here is now the main result. Note the use of the ${\mathbf d}_1$-distance {(see \eqref{eq:wasserstein:distance})} in the assumption required from F in the statement below, 
${\mathbf d}_{1}$-continuity being stronger than ${\mathbf d}_{2}$-continuity.

\begin{thm}
\label{theorem:rate:cv}   \label{theoremConvergenceRate}
We make the following assumptions.   

\begin{enumerate}
   \item Assumptions (a)--(d) in the statement of Theorem 
   \ref{theorem:prop:of:chaos} are satisfied for the same parameters $p \in [2,3)$ and $q \geq 8$ as in Section 
   \ref{SectionRoughStructure}, for some control 
   $w$
satisfying \eqref{eq:w:s:t:omega:ineq} and \eqref{eq:useful:inequality:wT}
    and some time horizon $T>0$. 
   \item The first and second derivatives of $\textrm{\rm F}$,  
   $
   (x,\mu) \mapsto \partial_{x} \textrm{\rm F}(x,\mu)$, $(x,\mu,z) \mapsto {\bigl(} D_{\mu} \textrm{\rm F}(x,\mu)(z), \partial_{x} D_{\mu} \textrm{\rm F}(x,\mu)({z}) {\bigr)}$,
   and $(x,\mu,z,z') \mapsto D_{\mu}^2 \textrm{\rm F}(x,\mu)({z,z'})$, are bounded on the whole space and are Lipschitz continuous with respect to all the variables, the Lipschitz property in the direction $\mu$ being understood with respect to ${\mathbf d}_{1}$.   

   \item Last, for any $\alpha >0$,  there exists a constant $\varepsilon_{2}>0$ such that, for {some} $p' \in 
   {[p,3)}$, and any random variables $\uptau,\uptau' : \Omega \rightarrow [0,T]$, with $\PP(\uptau<\uptau') = 1$, we have
\end{enumerate}
\begin{equation}
\label{eq:third:item}
\sup_{n \geq 1} \sup_{1 \leq i \leq n}  {\mathbb E} \biggl[ \exp \biggl[ \biggl( \frac{\widehat N^{i,n}\bigl([\uptau,\uptau'],\omega,\alpha\bigr)}{\sqrt{\uptau' - \uptau}} \biggr)^{1+\varepsilon_{2}}
\biggr] \biggr] < \infty,
\end{equation}
\begin{enumerate}
\item[]
where $\widehat N^{i,n}\bigl([\uptau,\uptau'],\omega,\alpha\bigr)$ is defined as the accumulation 
$N_{\varpi}\bigl([\uptau,\uptau'],\alpha\bigr)$ 
when $\varpi = (\widehat w_{p'}^{i,n}(\omega))^{1/p'}$ with
\end{enumerate}
\begin{equation}
\label{eq:w:widehat:w:v:widehat:v}
\begin{split}
\widehat{w}^{i,n}_{p'}(s,t,\omega) &:= \bigl( w^{i,n}_{p'} + 
\widehat{v}^{i,n}_{p'} \bigr)(s,t,\omega)
+
 {}^{(n)} \hspace{-1pt}\big\lgroup  
\widehat v^{\bullet,n}_{p'}(\omega)  \big\rgroup_{q ; [s,t],1-\textrm{\rm v}}  
 + (t-s),
\\
w_{p'}^{i,n}(s,t,\omega) &:= v_{p'}^{i,n}(s,t,\omega) +  {}^{(n)} \hspace{-1pt} \big\lgroup  v_{p'}^{\bullet,n}(\omega) \big\rgroup_{q ; [s,t],1-\textrm{\rm v}},  
\\
\widehat{v}_{p'}^{i,n}(s,t,\omega) &:=  \big\langle \WW^{i,\indep}(\omega,\cdot) \big\rangle_{q ; [s,t],p'/2-\textrm{\rm v}}^{p'/2}   
 + \big\langle \WW^{i,\indep}(\cdot,\omega) \big\rangle_{q ; [s,t],p'/2-\textrm{\rm v}}^{p'/2}.  
\end{split}
\end{equation}

Then, for any $r \geq 1$, there exists an exponent $q(r) \geq 8$ such that, 
if $X_{0}(\cdot)$ is in $\LL^{q(r)}$, then, for any $n \geq 1$, 
\begin{equation}   \label{EqConvergenceRate}
\sup_{1 \leq i \leq n} {\mathbb E} \left[ \sup_{0 \leq t \leq T} \big\vert \overline{X}^i_{t} - X^{i,(n)}_{t} \big\vert^r
 \right]^{1/r}
 + {\mathbb E} 
 \left[ \sup_{0 \leq t \leq T} {\mathbf d}_{1} \bigl( \mu^n_{t},\overline \mu^n_{t} \bigr)^r
  \right]^{1/r}
  \leq C \varsigma_{n},
\end{equation}
for a constant $C$ independent of $n$, and $\varsigma_{n} = n^{-1/2}$ if $d=1$, $\varsigma_{n}=n^{-1/2} \ln(1+n)$ if $d=2$ and $\varsigma_{n}=n^{-1/d}$ if $d \geq 3$.
\end{thm}

\begin{rem}
\label{rem:5:1}
 Let us make a few remarks on this statement before embarking on its proof.

\begin{enumerate}   

   \item[$\textcolor{gray}{\bullet}$] We refer to  \cite[Chapter 5]{CarmonaDelarue_book_I} for examples of a function $\textrm{\rm F}$ satisfying item \textit{(b)} in the assumptions of the statement. Importantly, we recall that a function $G : {\mathcal P}_{2}(\RR^d) \ni \mu \mapsto G(\mu) \in \RR$, whose derivative $D_{\mu} G : {\mathcal P}_{2}(\RR^d) \times \RR^d \ni (\mu,z) \mapsto D_{\mu} G(\mu)(z) \in \RR^d$ is uniformly bounded on the whole ${\mathcal P}_{2}(\RR^d) \times \RR^d$, is Lipschitz continuous with respect to the ${\mathbf d}_{1}$-Wasserstein distance. In particular, under the assumptions of the statement, F itself is Lipschitz continuous on $\RR^d \times {\mathcal P}_{2}(\RR^d)$, the Lipschitz property in the direction $\mu$ being understood with respect to ${\mathbf d}_{1}$.   

	\item[$\textcolor{gray}{\bullet}$]
	Obviously,  
condition
\eqref{eq:third:item}
depends on $p'$. We let the reader check that 
if 
\eqref{eq:third:item}
holds for some $p' \in [p,3)$, then it holds for any other 
$p'' \in [p',3)$.  

   \item[$\textcolor{gray}{\bullet}$] By inspecting the proof of Theorem \ref{theoremConvergenceRate}, we could make explicit the value of $q(r)$ (in the condition $X_{0}(\cdot) \in \LL^{q(r)}$), but we feel that it would not be so useful.

	\item[$\textcolor{gray}{\bullet}$] The {convergence rate $\varsigma_n$ in \eqref{EqConvergenceRate}} corresponds to the usual rate for the convergence in the 1-Wasserstein distance of an empirical sample of independent, identically distributed, random variables toward the limiting common distribution; see \cite{FournierGuillin} together with Lemma \ref{lem:W1:cv:empirical}.   

   \item[$\textcolor{gray}{\bullet}$]  Theorem \ref{theorem:rate:cv} applies when $W$ is a continuous centred Gaussian process defined over $[0,T]$ as {in Example \ref{Gaussian}.}

\end{enumerate}
\end{rem}

\begin{proof}
Observe that, for each $i \in \{1,\cdots,n\}$ and any $\omega \in \Omega$, we can define the integral process
$
\bigl( \int_{0}^t \textrm{F}\bigl(\overline X_{s}^i(\omega),\overline \mu^n_{s}(\omega) \bigr) d {\boldsymbol W}^{i,(n)}_{s}(\omega) \bigr)_{0 \leq t \leq T}$ {using usual rough paths theory}, 
where the label $i$ in the notation ${\boldsymbol W}^{i,(n)}(\omega)$ is here to indicate that the integral only involves $\big(W^i(\omega),({\mathbb W}^{j,i}(\omega))_{1 \leq j \leq n}\big)$. Equivalently, ${\boldsymbol W}^{i,(n)}(\omega)$ must be seen as $\big(W^i(\omega),({\mathbb W}^{j,i}(\omega))_{1 \leq j \leq n}\big)$. The fact that the integral may be defined with respect to $\big(W^i(\omega),({\mathbb W}^{j,i}(\omega))_{1 \leq j \leq n}\big)$ follows from the fact that $\overline X^j(\omega)$, for each $j \in \{1,\cdots,n\}$ and each $\omega \in \Omega$, is controlled by the variations of the sole $W^j(\omega)$.  {So, whenever we expand locally $\textrm{F}\bigl(\overline X_{s}^i(\omega),\overline \mu^n_{s}(\omega) \bigr)$, 
we let appear increments of $\overline X_{s}^i(\omega)$, which are controlled by increments of $W^i(\omega)$, and also increments of 
$\overline X_{s}^j(\omega)$, for $j \not =i$, which are controlled by increments of $W^j(\omega)$: At the end of the day, it suffices to have the iterated integrals $({\mathbb W}^{j,i}(\omega))_{1 \leq j \leq n}$ to define the integral $
\bigl( \int_{0}^t \textrm{F}\bigl(\overline X_{s}^i(\omega),\overline \mu^n_{s}(\omega) \bigr) d {\boldsymbol W}^{i,(n)}_{s}(\omega) \bigr)_{0 \leq t \leq T}$.}

\medskip

\textbf{\textsf{Step 1.}}
The first step is to compare 
\begin{equation}
\label{eq:comparison:integrals}
\begin{split}
\int_{0}^t \textrm{\rm F}\Bigl(\overline X_{s}^i(\omega),{\mathcal L}(X_{s}) \Bigr) d \overline{\boldsymbol W}^{i}_{s}(\omega)
\ \
 \textrm{\rm and}
\ \ \int_{0}^t \textrm{\rm F}\Bigl(\overline X_{s}^i(\omega),\overline \mu^n_{s}(\omega) \Bigr) d {\boldsymbol W}^{i,(n)}_{s}(\omega),
\end{split}
\end{equation}
for $t \in [0,T]$. What makes the proof non-trivial is the fact that the rough set-ups used in the first and the second integrals are not the same. So, in order to compare the two of them, we need to come back to the original constructions of the two integrals. To simplify notations, and for $0 \leq t \leq T$, set 
\begin{equation}
\label{eq:barFti:Ftin}
\begin{split}
&\overline F^i_{t}(\omega) := \textrm{\rm F}\Big(\overline X_{t}^i(\omega),{\mathcal L}(X_{t})\Big),
\quad F_{t}^{i,n}(\omega) := \textrm{\rm F}\Big(\overline X_{t}^i(\omega),\overline \mu^n_{t}(\omega)\Big).
\end{split}
\end{equation}
For sure, $\big(\overline F^i_{t}(\omega)\big)_{0 \leq t \leq T}$ is $\omega$-controlled by $\overline {\boldsymbol W}^i(\omega)$, {see Definition \ref{definition:omega:controlled:trajectory} and Proposition 
\ref{prop:chaining}}, and the collection indexed by $\omega \in \Omega$ is a random path controlled by $\overline {\boldsymbol W}^i$, see Definition \ref{definition:random:controlled:trajectory} for a reminder. The corresponding Gubinelli derivatives are denoted by $\big(\delta_{x} \overline F_{t}^{i}(\omega),\delta_{\mu} \overline F_{t}^{i}(\omega,\cdot)\big)_{0 \leq t \leq T}$, see again Proposition \ref{prop:chaining}. Similarly, $(F^{i,n}_{t}(\omega))_{0 \leq t \leq T}$ is controlled by ${\boldsymbol W}^{i,(n)}(\omega)$ {and ${\boldsymbol W}^{\bullet,(n)}(\omega)$}
and Gubinelli derivatives are encoded in the form of a collection $\big(\delta_{x} F_{t}^{i,n}(\omega), \big(\delta_{\mu}  F_{t}^{i,j,n}(\omega)\big)_{1 \leq j \leq n}\big)_{0 \leq t \leq T}$, {see (\ref{eq:deltax:empirical}--\ref{eq:deltamu:empirical})}. To make it clear, set
\begin{equation}
\label{eq:delta:barFti}
\begin{split}
&\delta_{x} \overline F_{t}^{i}(\omega) := \partial_{x} \textrm{\rm F}\bigl(\overline X_{t}^i(\omega), {\mathcal L}(X_{t})\bigr)  \, \textrm{\rm F}\bigl(\overline X_{t}^i(\omega), {\mathcal L}(X_{t})\big),   
\\
&\delta_{\mu} \overline F_{t}^{i}(\omega,\cdot) := D_{\mu} \textrm{\rm F}\bigl(\overline X_{t}^i(\omega), {\mathcal L}(X_{t})\bigr)
\bigl(\overline X_{t}^i(\cdot) \bigr) \, \textrm{\rm F}\bigl( \overline X_{t}^i(\cdot), {\mathcal L}(X_{t})\bigr),
\end{split}
\end{equation}
where $X(\cdot)$ is the solution to \eqref{EqRDE} with ${\boldsymbol W}(\cdot) = \big(W(\cdot),{\mathbb W}(\cdot),
{\mathbb W}^{\indep}(\cdot,\cdot)\big)$. We also let
\begin{equation}
\label{eq:delta:Ftin}
\begin{split}
&\delta_{x} F_{t}^{i,n}(\omega)
:= \partial_{x} \textrm{\rm F}\bigl(\overline X_{t}^i(\omega),\overline \mu^n_{t}(\omega)\bigr) \, \textrm{\rm F}\bigl(\overline X_{t}^i(\omega),\overline \mu^n_{t}(\omega)\bigr), 
\\
&\delta_{\mu} F_{t}^{i,j,n}(\omega) := 
D_{\mu} \textrm{\rm F}\bigl(\overline X_{t}^i(\omega),
\overline \mu^n_{t}(\omega)
\bigr)
\bigl(\overline X_{t}^j(\omega)
\bigr) \,
\textrm{\rm F}\bigl(\overline X_{t}^j(\omega),
\overline \mu^n_{t}(\omega)
\bigr).
\end{split}
\end{equation}
For a subdivision $\Delta = \{ s = t_{0} < t_{1} < \cdots < t_{K}=t \}$, set
\begin{equation}
\label{eq:I:i,n,Delta}
\begin{split}
\overline {\mathcal I}^{i,\Delta}_{s,t}(\omega) &:= \sum_{k=0}^{K-1}
\Big\{ \overline F_{t_{k}}^i(\omega) W_{t_{k},t_{k+1}}^i(\omega) + \delta_{x} \overline F_{t_{k}}^i(\omega) {\mathbb W}_{t_{k},t_{k+1}}^{i}(\omega)
\\
&\hspace{15pt}+ 
{\mathbb E}
\bigl[ 
\delta_{\mu} \overline F_{t_{k}}^{i}(\omega,\cdot) {\mathbb W}_{t_{k},t_{k+1}}^{i,\indep}(\cdot,\omega)
\bigr]
\Bigr\},
\\
{\mathcal I}^{i,n,\Delta}_{s,t}(\omega) &:= \sum_{k=0}^{K-1}
\Big\{ F_{t_{k}}^{i,n}(\omega) W_{t_{k},t_{k+1}}^i(\omega) + \delta_{x} F_{t_{k}}^{i,n}(\omega) {\mathbb W}_{t_{k},t_{k+1}}^{i}(\omega)
\\
&\hspace{15pt} + 
\frac1n 
\sum_{j=1}^n
\delta_{\mu} F_{t_{k}}^{i,j,n}(\omega) {\mathbb W}_{t_{k},t_{k+1}}^{j,i}(\omega)
\Bigr\}.
\end{split}
\end{equation}
The two integrals in \eqref{eq:comparison:integrals} should be understood as the respective limits of the two Riemann sums right above as $K$ tends to $\infty$. In the sequel, we denote the summand in the first sum by $\overline{\mathcal I}^{i,\partial}_{\{t_{k},t_{k+1}\}}(\omega)$ and the summand in the second sum by ${\mathcal I}^{i,n,\partial}_{\{t_{k},t_{k+1}\}}(\omega)$. By {Lemma \ref{le:sewing} proved in Appendix \ref{AppendixAuxiliary}, we can find, for any $\varrho \geq 8$, an exponent $\varrho' \geq q$, 
 independent of $n$, $K$ and $\Delta$, 
 such that, whenever $X_{0}(\cdot) \in \LL^{\varrho'}$, it holds, 
 for a constant $C$, also independent of $n$, $K$ and $\Delta$ but depending on 
 $\langle X_{0}(\cdot) \rangle_{\varrho'}$, 
 and then}
 for any $k \in \{1,\cdots,K-1\}$ (provided $K \geq 2$),
\begin{equation*}
\Bigl\langle \Big\{{\mathcal I}^{i,n,\Delta}_{s,t}(\cdot) - {\mathcal I}^{i,n,\Delta'}_{s,t}(\cdot)\Big\} - \Big\{ \overline{\mathcal I}^{i,\Delta}_{s,t}(\cdot) - \overline{\mathcal I}^{i,\Delta'}_{s,t}(\cdot)\Big\} \Bigr\rangle_{\varrho} \leq C \varsigma_{n} \bigl\llangle w^+(t_{k-1},t_{k+1},\cdot,\cdot) \bigr\rrangle_{\varrho'}^{3/p},
\end{equation*}
where 
$
\Delta' := \Delta \setminus \{t_{k}\}
$
and 
$
w^+(s,t,\omega,\omega') := w(s,t,\omega) + \| {\mathbb W}^{\indep}(\omega,\omega') \|_{[s,t],p/2-\textrm{\rm v}}^{p/2},
$ 
({$w$ being here given by 
\eqref{eq:w:s:t:omega}}). {In the rest of the proof, it is implicitly understood that
the condition  
$X_{0}(\cdot) \in 
 \LL^{\varrho'}$ is satisfied for $\varrho'$ large enough and that the constant $C$ is allowed to 
 depend on $\langle X_{0}(\cdot)\rangle_{\varrho'}$}.

{Formulating \eqref{eq:v:p:i:N} and \eqref{eq:lln:holder} but for the limit (instead of empirical) rough set-up, we know that the right hand side in the above inequality} is less than 
\begin{equation*}
\begin{split}
&C \varsigma_{n} \Bigl[ \bigl\langle  
\|
W(\cdot)
\|_{[0,T],(1/p)-\textrm{\rm H}} \bigr\rangle_{p \varrho'}
+
\bigl\langle
\|{\mathbb W}(\cdot) \bigr\|_{[0,T],(2/p)-\textrm{\rm H}}
\bigr\rangle_{p \varrho'}^{1/2}
\\
&\hspace{120pt}+
\bigl\llangle
\|
{\mathbb W}^{\indep}(\cdot,\cdot) \bigr\|_{[0,T],(2/p)-\textrm{\rm H}}
\bigr\rrangle_{p \varrho'}^{1/2}
\Bigr]^{3} (t_{k+1}-{t_{k-1}})^{3/p},
\end{split}
\end{equation*}
but by assumption all the expectations are finite. 
Now we can choose ${k}$ such that 
$\vert t_{k+1} - t_{k-1} \vert \leq {4} \vert t-s \vert/K$ (if not, it means that 
${4} ( t-s ) (K-1)/K
< \sum_{k=1}^{K-1} \vert t_{k+1} - t_{k-1} \vert = 
\sum_{k=1}^{K-1} ( t_{k+1} -t_{k} + t_{k} - t_{k-1})
= 2(t-s) - (t_{K} - t_{K-1} + t_{1} - t_{0}) \leq {2( t-s)}$, which is a contradiction {since $K \geq 2$}).  
We get 
\begin{equation*}
\Bigl\langle \Big\{ {\mathcal I}^{i,n,\Delta}_{s,t}(\cdot) - {\mathcal I}^{i,n,\Delta'}_{s,t}(\cdot) \Big\} - \Big\{ \overline{\mathcal I}^{i,\Delta}_{s,t}(\cdot) - \overline{\mathcal I}^{i,\Delta'}_{s,t}(\cdot) \Big\} \Bigr\rangle_{\varrho} \leq C \varsigma_{n} \Bigl( \frac{t-s}{K} \Bigr)^{3/p},
\end{equation*}
the constant $C$ being allowed to increase from line to line as long as it remains independent of $n$, $K$, $\Delta$ and $\Delta'$. 
Letting $t^{(1)}=t_{k}$ and
applying iteratively the above bound to a sequence of meshes of the form $\Delta \setminus \{t^{(1)}\}$, 
$\Delta \setminus \{t^{(1)},t^{(2)}\}$, \dots, and then letting $K$ tend to $\infty$, we deduce that 
\begin{equation}
\label{eq:I:Delta}
\begin{split}
&\left\langle  
\int_{s}^t F_{r}^{i,n}(\cdot) d {\boldsymbol W}_{r}^{i,(n)}(\cdot) - \int_{s}^t \overline F_{r}^{i}(\cdot) d \overline{\boldsymbol W}_{r}^i(\cdot) - \Big\{ {\mathcal I}^{i,n,\partial}_{\{s,t\}} - \overline{\mathcal I}^{i,\partial}_{\{s,t\}} \Big\} \right\rangle_{\varrho}   
\\
&\hspace{15pt}  \leq C \varsigma_{n} (t-s)^{3/p}.
\end{split}
\end{equation}
By Lemma \ref{le:sewing}, we also have
$\bigl\langle 
{\mathcal I}^{i,n,\partial}_{\{s,t\}}
-
\overline{\mathcal I}^{i,\partial}_{\{s,t\}}
\bigr\rangle_{\varrho} \leq C \varsigma_{n} (t-s)^{1/p}$,
from which we deduce that 
\begin{equation*}
\begin{split}
&\Bigl\langle  
\int_{s}^t 
F^{i,n}_{r}(\cdot)
 d {\boldsymbol W}_{r}^{i,(n)}(\cdot)
 -
 \int_{s}^t 
\overline F^i_{r}(\cdot) d \overline{\boldsymbol W}_{r}^i(\cdot)
\Bigr\rangle_{\varrho}
 \leq C \varsigma_{n} (t-s)^{1/p}.
\end{split}
\end{equation*}
Similarly, Lemma \ref{le:sewing} says that
$\bigl\langle 
\bigl[ F^{i,n}(\cdot)
-\overline{F}^i(\cdot)  
\bigr]_{s,t}
\bigr\rangle_{\varrho}
\leq C \varsigma_{n} (t-s)^{1/p}$,
and, noting that 
\begin{equation*}
\begin{split}
R^{\int F^{i,n} d  {\boldsymbol W}^{i,(n)}}_{s,t}(\omega)
& = 
\int_{s}^t 
F_{r}^{i,n}(\omega) d  {\boldsymbol W}_{r}^{i,(n)}(\omega) - {\mathcal I}^{i,n,\partial}_{s,t}(\omega)
\\
&\hspace{15pt}
+ \delta_{x} F_{s}^{i,n}(\omega) {\mathbb W}_{s,t}^i(\omega) 
+ 
\frac1n \sum_{j=1}^n
\delta_{\mu} F_{s}^{i,j,n}(\omega) {\mathbb W}_{s,t}^{j,i}(\omega),
\\
R^{\int \overline F^i d \overline {\boldsymbol W}^i}_{s,t}(\omega)
&= 
\int_{s}^t 
\overline F_{r}^{i}(\omega) d \overline{\boldsymbol W}_{r}^i(\omega)
- \overline{\mathcal I}^{i,\partial}_{s,t}(\omega)
\\
&\hspace{15pt} + 
\delta_{x} \overline F_{s}^i(\omega) {\mathbb W}_{s,t}^i(\omega) 
+ 
\EE \Bigl[
\delta_{\mu} \overline F_{s}^i(\omega,\cdot) {\mathbb W}_{s,t}^{i,\indep}(\cdot,\omega)
\Bigr],
\end{split}
\end{equation*}
we deduce in a similar manner,  {using   
\eqref{eq:I:Delta} and Lemma \ref{le:sewing} once again},
that
\begin{equation*}
\Bigl\langle 
R^{\int  F^{i,n} d {\boldsymbol W}^{i,(n)}}_{s,t}(\cdot)
-
R^{\int \overline F^i d \overline{\boldsymbol W}^i}_{s,t}(\cdot)
\Bigr\rangle_{\varrho}
\leq C \varsigma_{n} (t-s)^{2/p}.
\end{equation*}
So, fixing $i \in \{1,\cdots,n\}$, choosing $\varrho$ large enough and applying a suitable version of Kolmogorov's theorem (see for instance Theorem 3.1 in \cite{FrizHairer}), we can find $p' \in (p,3)$, 
{which may be assumed to satisfy item $(c)$ in the assumption (see for instance 
the second bullet point in Remark 
\ref{rem:5:1}) 
and
the value of which is fixed until the end of the proof}, such that 
\begin{equation}
\label{eq:controls:empirical}
\begin{split}
&\biggl\vert 
\int_{s}^t 
F_{r}^{i,n}(\omega)
 d {\boldsymbol W}_{r}^{i,(n)}
 -
 \int_{s}^t 
\overline F_{r}^i(\omega) 
 d \overline{\boldsymbol W}_{r}^i(\omega)
\biggr\vert \leq \theta^{i,n}(\omega) (t-s)^{1/p'},
\\
&\Bigl\vert \Bigl[
 F^{i,n}(\omega) 
-
\overline F^i(\omega) 
\Bigr]_{s,t} \Bigr\vert \leq \theta^{i,n}(\omega) (t-s)^{1/p'},
\\
&\left\vert
R^{\int  F^{i,n} d {\boldsymbol W}^{i,(n)}}_{s,t}(\omega)
-
R^{\int \overline F^i d \overline{\boldsymbol W}^i}_{s,t}(\omega)
\right\vert \leq \theta^{i,n}(\omega) (t-s)^{2/p'},
\end{split}
\end{equation}
with $\big\langle \theta^{i,n}(\cdot)\big\rangle_{\varrho} \leq C \varsigma_{n}$, for a new value of the constant $C$.  

\smallskip

Observe now that the empirical control associated with our empirical rough set-up and with the exponent $p'$ reads 
(compare with \eqref{eq:w:s:t:omega})
\begin{equation*}
w_{p'}^{i,n}(s,t,\omega)
:= v_{p'}^{i,n}(s,t,\omega) + 
 {}^{(n)} \hspace{-1pt}\big\lgroup  
 v_{p'}^{\bullet,n}(\omega)   \big\rgroup_{q ; [s,t],1-\textrm{\rm v}},
\end{equation*}
where we used the same notation as in \eqref{eq:v:N}. In fact, there is no loss of generality in changing the definition of $w_{p'}^{i,n}$ into
\begin{equation}
\label{eq:w:i:n}
w_{p'}^{i,n}(s,t,\omega)
:= v_{p'}^{i,n}(s,t,\omega) + 
 {}^{(n)} \hspace{-1pt}\big\lgroup  
 v_{p'}^{\bullet,n}(\omega)  \big\rgroup_{q ; [s,t],1-\textrm{\rm v}} + (t-s),
\end{equation}
which permits to replace $(t-s)^{1/p'}$ by $w_{p'}^{i,n}(s,t,\omega)^{1/p'}$ in the inequalities \eqref{eq:controls:empirical}. Hence,
\begin{equation*}
\biggl\vvvert 
\int_{0}^{\cdot} 
F_{r}^{i,n}(\omega)
 d {\boldsymbol W}_{r}^{i,(n)}
 -
 \int_{0}^{\cdot} 
\overline F_{r}^i(\omega) 
 d \overline{\boldsymbol W}_{r}^i(\omega)
 \biggr\vvvert_{[0,T],w_{p'}^{i,n},p'} \leq \theta^{i,n}(\omega). 
\end{equation*}

\textbf{\textsf{Step 2.}} We now make use of Proposition \ref{theorem:fixed:1} to compare 
\begin{equation*}
\begin{split}
\int_{0}^t \textrm{\rm F}\bigl(X_{s}^{i,(n)}(\omega),\mu_{s}^n(\omega) \bigr) d {{\boldsymbol W}^{i,(n)}_{s}}(\omega)
\quad 
 \textrm{\rm and}
\quad \int_{0}^t \textrm{\rm F}\bigl(\overline X_{s}^i(\omega),\overline \mu^n_{s}(\omega) \bigr) d {\boldsymbol W}^{i,(n)}_{s}(\omega),
\end{split}
\end{equation*}
see \eqref{eq:empirical:measures}.
To simplify the notations, we just write $X^i$ for $X^{i,(n)}$
and ${\boldsymbol W}^{i}$ for ${\boldsymbol W}^{i,(n)}$. 
We then apply Proposition \ref{theorem:fixed:1} with 
\begin{equation}
\label{eq:X:Y:Xprime:Yprime}
\big(X(\omega),Y(\cdot)\big)=\big(X^{i}(\omega),X^{\bullet}(\omega)\big), 
\ 
\big(X'(\omega),Y'(\cdot)\big)=\big(\overline X^{i}(\omega),\overline X^{\bullet}(\omega)\big),
\end{equation}
the underlying set-up being understood as the empirical rough set-up for a given realization $\omega$ ({in particular, the various assumptions on the moments 
in 
Proposition \ref{theorem:fixed:1}
must be checked under the empirical distribution}). The difficulty here is that the variations of these two solutions are controlled by two different functionals
$w$, see 
\eqref{eq:omega:controlled:eq}; {in short, there is the control associated with the empirical set-up and the control associated with the theoretical one}. This is the rationale for 
introducing 
$\widehat{w}_{p'}^{i,n}$ in 
\eqref{eq:w:widehat:w:v:widehat:v}. Obviously, $\widehat{w}_{p'}^{i,n}(\cdot,\cdot,\omega)$ is not the \textit{natural} control functional associated with ${\boldsymbol W}^i(\omega)$, but it is greater than 
${w}^{i,n}_{p'}(s,t,\omega)$ and it satisfies 
\begin{equation}
\label{eq:bound:n:hat:w}
{}^{(n)} \hspace{-1pt}\big\lgroup 
\widehat w_{p'}^{\bullet,n}(s,t,\omega)  \big\rgroup_{q}
\leq 2 \widehat{w}_{p'}^{i,n}(s,t,\omega),
\end{equation}
which suffices to apply  
Proposition 
\ref{theorem:fixed:1}, see also \cite[Proposition 4.3]{BCD1}, with $w^{i,n}_{p'}(s,t,\omega)$ replaced by $\widehat{w}^{i,n}_{p'}(s,t,\omega)$. 
The resulting semi-norm that must be used to control the difference $\bigl(X(\omega)-X'(\omega),Y(\cdot) - Y'(\cdot)\bigr)
= \bigl( X^i(\omega) - \overline X^i(\omega),X^{\bullet}(\omega) - \overline X^{\bullet}(\omega) \bigr)$
on a given interval $[s,t]$
is $\vvvert \cdot \vvvert_{[s,t],\widehat w^{i,n},p'}$ ({we feel easier not to put the index $p'$ in 
$\widehat w^{i,n}$ but it is implicitly understood)}.
We use the corresponding local accumulation, which we denote by $\widehat N^{i,n}\bigl([0,T],\omega,\alpha\bigr) $.  
\smallskip

{We first check that the pair
$\big(X(\omega),Y(\cdot)\big)=\big(X^{i}(\omega),X^{\bullet}(\omega)\big)$
in \eqref{eq:X:Y:Xprime:Yprime}
satisfies the assumptions of Proposition \ref{theorem:fixed:1}, assuming that $T$ is upper bounded by $1$ (which may seem rather restrictive but which is in fact consistent with what we do in the sequel since we require, at least for a while, $T$ to be small enough).} 
By construction of the processes $\bigl(X^j(\omega)\bigr)_{j=1,\cdots,n}$ as the solution of the empirical rough equation, 
{we know from Theorem 
\ref{main:theorem:existence:small:time} 
and from \cite[Proposition 4.2]{BCD1} (which guarantees a form of stability in the iterative construction of the solution) that we can find three deterministic constants
$L_{0}$, $\eta_{0}$ and $\gamma_{0}$, all strictly greater than 1 and only depending on $\Lambda$ 
(we may forget the dependence upon $T$ since $T \leq 1$; in particular the three  
constants are independent of $n$ and, for sure, of the index $i$ as well), such that 
$X(\omega)=X^{i}(\omega)$ satisfies 
 \eqref{eq:fixed:2} with 
$w = \widehat{w}^{i,n}_{p'}$ 
(it being understood that the points 
$\bigl(t_{\ell}^0=\tau_{\ell}(0,T,\omega,1/(4L_{0}))\bigr)_{\ell=0,\cdots,N^0+1}$ in the statement of 
Proposition 
\ref{theorem:fixed:1} are constructed with respect to 
$\widehat{w}^{i,n}_{p'}$) 
and 
%
$Y(\cdot)=X^{\bullet}(\omega)$
satisfies 
 condition 
\eqref{eq:fixed:1} 
with respect to 
${}^{(n)} \hspace{-1pt} \bigl\lgroup \,
\cdot \,
\bigr\rgroup_{8}$, provided that $T$ satisfies} 
\begin{align}
\label{eq:NL0:coupling}
 &{}^{(n)} \hspace{-1pt}\Big\lgroup  \widehat N^{\bullet,n}\big([0,T],\omega,1/(4L_{0})\big) 
 \Big\rgroup_{8} \leq 1, 
\\  
 &{}^{(n)} \hspace{-1pt} \biggl\lgroup \Bigl[ \gamma_{0} \Bigl( 1 + \widehat w_{p'}^{\bullet,n}(0,T,\cdot)^{1/p'} \Bigr)\Bigr]^{\widehat N^{\bullet,n}([0,T],\cdot,1/(4L_{0}))} \biggr\rgroup_{32} \textcolor{black}{\leq \eta_{0}}.
\label{eq:NL0:coupling:supple}
\end{align}
\smallskip

{We now check that the pair
$\big(X'(\omega),Y'(\cdot)\big)=\big(\overline X^{i}(\omega),\overline X^{\bullet}(\omega)\big)$
in \eqref{eq:X:Y:Xprime:Yprime}
also satisfies the assumptions of Proposition \ref{theorem:fixed:1}.}
In fact, 
using the H\"older regularity of the paths, see
\eqref{eq:v:p:i:N}
for a similar use,
and using the additional $t-s$ in the definition 
\eqref{eq:w:widehat:w:v:widehat:v}, 
$\widehat{w}^{i,n}_{p'}$
dominates 
(up to a multiplicative constant)
the control $\overline w^i$ associated to $\overline{\boldsymbol W}^i$
(and $p'$)
through \eqref{eq:w:s:t:omega}  (see \eqref{eq:rough:set-up:barwi}
for the definition of 
$\overline{\boldsymbol W}^i$;
in short, the variations of $W^i(\omega)$ and ${\mathbb W}^i(\omega)$
are already included in $v^{i,n}_{p'}(\omega)$, the variations of ${\mathbb W}^{i,\indep}(\omega,\cdot)$
and ${\mathbb W}^{i,\indep}(\cdot,\omega)$ are precisely included in the definition of 
$\widehat{v}_{p'}^{i,n}(\omega)$ and, using the H\"older regularity of the paths, the variations of $W^i(\cdot)$ and ${\mathbb W}^{i,\indep}(\cdot,\cdot)$ (in 
$\LL^q$) are dominated by the additional $t-s$). Moreover, we have
({observe that, since we work here with a copy of the theoretical rough set-up, 
we use theoretical instead of empirical moments to check the various properties of a rough set-up})
\begin{equation*}
\bigl\langle \widehat{w}^{i,n}_{p'}(s,t,\cdot) \bigr\rangle_{q} \leq C(t-s) \leq C \widehat{w}^{i,n}_{p'}(s,t,\omega),
\end{equation*}
for a constant $C$ independent of $i$, $n$, $s$ and $t$. Although $C \geq 2$ ({compare with 
\eqref{eq:w:s:t:omega:ineq}}), this permits to use $\widehat{w}_{p'}^{i,n}(s,t,\cdot)$ as control functional when working with the rough set-up $\overline{\boldsymbol W}^i$ and, in particular, when invoking the solvability Theorem 
\ref{main:theorem:existence:small:time} -- the proof would be the same. This is an important point: the path $\overline{X}^{i}(\omega)$, defined right after \eqref{eq:rough:set-up:barwi}, is the solution of a mean-field rough equation driven by a signal that is controlled by $\widehat{w}_{p'}^{i,n}(\cdot)$. Hence, $X'(\omega)= \overline{X}^{i}(\omega)$ in \eqref{eq:X:Y:Xprime:Yprime} satisfies the second bound in \eqref{eq:fixed:2}
with $w = \widehat{w}^{i,n}_{p'}$, provided 
{\eqref{eq:NL0:coupling}
and
\eqref{eq:NL0:coupling:supple}
hold true but under the theoretical (instead of empirical) moments of order $8$ and $32$ respectively and $T$ therein is deterministic.
Of course, by exchangeability, 
the latter is in fact true automatically 
as soon as 
\eqref{eq:NL0:coupling}
and
\eqref{eq:NL0:coupling:supple}
themselves are satisfied (it suffices to take power $8$ in 
\eqref{eq:NL0:coupling}
and 
power $32$
in 
\eqref{eq:NL0:coupling:supple}
and then to take the theoretical expectation).
By the same argument, 
$\overline X^{i}(\cdot)$ satisfies the second condition in \eqref{eq:fixed:1} with respect to 
$\langle \cdot \rangle_{8}$; in turn, this, together with 
\eqref{eq:NL0:coupling}, make it possible to apply the first line in \cite[Proposition 4.2, (4.4)]{BCD1} for each $i \in \{1,\cdots,n\}$ and to deduce that 
$Y'(\cdot)=\overline X^{\bullet}(\omega)$ satisfies the second condition \eqref{eq:fixed:1} but with respect to ${}^{(n)} \hspace{-1pt} \bigl\lgroup 
\cdot \bigr\rgroup_{8}$. 
Due to the assumption in \cite[Proposition 4.2]{BCD1}, this may require to work with a larger value of the threshold $L_{0}$ in the statement of Proposition \ref{theorem:fixed:1}, but, as made clear in the statement of \cite[Proposition 4.2]{BCD1} itself, this is always possible.} Then, by  Proposition \ref{theorem:fixed:1}, we obtain, for a given $L \geq L_{0}$,
\begin{equation*}
\begin{split}
&\biggl\vvvert \int_{t_{k}}^{\cdot} \textrm{F}\bigl(X_{r}^i(\omega),\mu^n(\omega)
\bigr) d {\boldsymbol W}_{r}^i(\omega) - \int_{t_{k}}^{\cdot} \textrm{F}\bigl(\overline X_{r} ^i(\omega),
\overline \mu^n(\omega)\bigr) d {\boldsymbol W}_{r}^i(\omega) \biggr\vvvert_{[t_{k},t_{k+1}],\widehat w^{i,n},p'}   
\\
&\leq \gamma\, \widehat w^{i,n}_{p'}(0,t_{k},\omega)^{1/p'}\,
\biggl(\big\vvvert  
\bigl( X^i - \overline X^i \bigr)(\omega)
\big\vvvert_{[0,t_{k}],\widehat w^{i,n},p'} 
\\
&\hspace{100pt}+ 
 {}^{(n)} \hspace{-3pt} \Big\lgroup
  \bigl\vvvert \bigl( X^{\bullet} - \overline X^{\bullet} \bigr) (\omega) \bigr\vvvert_{[0,T],\widehat w^{\bullet,n},p'}
   \hspace{-1pt}\Big\rgroup_{8}
 \biggr)   
\\
&\hspace{15pt} + \frac{\gamma}{4L}  \biggl(
\big\vvvert \bigl( X^i - \overline X^i \bigr) (\omega)\big\vvvert_{[t_{k},t_{k+1}],\widehat w^{i,n},p'} 
\\
&\hspace{100pt}
+ 
 {}^{(n)} \hspace{-3pt} \Big\lgroup
  \bigl\vvvert \bigl( X^{\bullet} - \overline X^{\bullet} \bigr) (\omega) \bigr\vvvert_{[t_{k},t_{k+1}],\widehat w^{\bullet,n},p'}
  \hspace{-1pt} \Big\rgroup_{8}
 \biggr),
\end{split}
\end{equation*}
where $\widehat w^{i,n}_{p'}(t_{k},t_{k+1},\omega)^{1/p'} \leq 1/(4L)$ and for $k \leq 2\widehat N^{i,n}([0,T],\omega,1/(4L))$ ({since the sequence 
$(t_{i})_{i}$ must refine the sequence $(t_{j}^0)_{j}$, we may assume that the collection $(t_{i})_{i}$ counts  
$2\widehat N^{i,n}([0,T],\omega,1/(4L))+2$ points, including $t_{0}=0$})
and
{where 
$\gamma$ depends on $L_{0}$ and $\Lambda$}. 
The point now is to insert the conclusion of the first step (replacing for free $w^{i,n}_{p'}$ by $\widehat w^{i,n}_{p'}$ therein). We get
\begin{equation*}
\begin{split}
&\bigl\vvvert  \bigl( X^i - \overline X^i \bigr)(\omega) \bigr\vvvert_{[t_{k},t_{k+1}],\widehat w^{i,n},p'}   
\\
&\leq \gamma\, \widehat w_{p'}^{i,n}(0,t_{k},\omega)^{1/p'}\,
\biggl(\big\vvvert  
\bigl( X^i - \overline X^i \bigr)(\omega)
\big\vvvert_{[0,t_{k}],\widehat w^{i,n},p'} 
\\
&\hspace{100pt}+ 
 {}^{(n)} \hspace{-3pt}\Big\lgroup 
  \bigl\vvvert \bigl( X^{\bullet} - \overline X^{\bullet} \bigr) (\omega) \bigr\vvvert_{[0,T],\widehat w^{\bullet,n},p'}
  \hspace{-1pt}\Big\rgroup_{8}
 \biggr) 
+ \theta^{i,n}(\omega)
\\
&\hspace{15pt} + \frac{\gamma}{4L}  \biggl(
\big\vvvert \bigl( X^i - \overline X^i \bigr) (\omega)\big\vvvert_{[t_{k},t_{k+1}],\widehat w^{i,n},p'} 
\\
&\hspace{100pt}
+ 
 {}^{(n)} \hspace{-3pt}\Big\lgroup  
  \bigl\vvvert \bigl( X^{\bullet} - \overline X^{\bullet} \bigr) (\omega) \bigr\vvvert_{[t_{k},t_{k+1}],\widehat w^{\bullet,n},p'}
  \hspace{-1pt} \Big\rgroup_{8}
 \biggr).
\end{split}
\end{equation*}
If $\gamma/(4L) \leq 1/2$, we get
\begin{equation}
\label{eq:bound:Xi-barXi}
\begin{split}
&\bigl\vvvert  \bigl( X^i - \overline X^i \bigr)(\omega) \bigr\vvvert_{[t_{k},t_{k+1}],\widehat w^{i,n},p'}   
\\
&\leq 
2 \gamma\, \Bigl( \frac1{L}+ \widehat w^{i,n}_{p'}(0,t_{k},\omega)^{1/p'} \Bigr) \,
\biggl(\big\vvvert  
\bigl( X^i - \overline X^i \bigr)(\omega)
\big\vvvert_{[0,t_{k}],\widehat w^{i,n},p'} 
\\
&\hspace{80pt}+ 
 {}^{(n)} \hspace{-3pt}\Big\lgroup  
  \bigl\vvvert \bigl( X^{\bullet} - \overline X^{\bullet} \bigr) (\omega) \bigr\vvvert_{[0,T],\widehat w^{\bullet,n},p'}
  \hspace{-1pt} \Big\rgroup_{8}
 \biggr)  + 2\theta^{i,n}(\omega).
\end{split}
\end{equation}
{We deduce that there exists a constant $c \geq 1$, possibly depending on 
$L_{0}$ but independent of $n$ and $L$ and $T$ and whose value may increase from line to line,
such that} ({see for instance
\cite[footnote 5]{BCD1} for the concatenation of two intervals})
\begin{equation*}
\begin{split}
&\bigl\vvvert  \bigl( X^i - \overline X^i \bigr)(\omega) \bigr\vvvert_{[0,t_{k+1}],\widehat w^{i,n},p'}   
\\
&\leq 
c \Bigl( 
\bigl\vvvert  \bigl( X^i - \overline X^i \bigr)(\omega) \bigr\vvvert_{[0,t_{k}],\widehat w^{i,n},p'}  
+
\bigl\vvvert  \bigl( X^i - \overline X^i \bigr)(\omega) \bigr\vvvert_{[t_{k},t_{k+1}],\widehat w^{i,n},p'}
\Bigr)  
\\
&\leq  c \, \bigl(1+\zeta_{T}^{i,n}(\omega)\bigr) \big\vvvert  
\bigl( X^i - \overline X^i \bigr)(\omega)
\big\vvvert_{[0,t_{k}],\widehat w^{i,n},p'} 
\\
&\hspace{15pt}
+ c \, \zeta^{i,n}_{T}(\omega) \,
 {}^{(n)} \hspace{-3pt}\Big\lgroup  
  \bigl\vvvert \bigl( X^{\bullet} - \overline X^{\bullet} \bigr) (\omega) \bigr\vvvert_{[0,T],\widehat w^{\bullet,n},p'}
  \hspace{-1pt} \Big\rgroup_{8}
  +  c \theta^{i,n}(\omega),
\end{split}
\end{equation*}
with
$\zeta^{i,n}_{T}(\omega) :=  \frac1{L} +{\widehat w}^{i,n}_{p'}(0,T,\omega)^{1/p'}$. 
So, by induction,
\begin{equation*}
\begin{split}
&\bigl\vvvert  \bigl( X^i - \overline X^i \bigr)(\omega) \bigr\vvvert_{[0,t_{k+1}],\widehat w^{i,n},p'}   
\leq c
\Bigl( \sum_{\ell=0}^{k} \bigl[ c \bigl(1+ \zeta^{i,n}_{T}(\omega) \bigr) \bigr]^\ell \Bigr)
 \\
  &\hspace{15pt} \times
 \biggl( 
 \zeta^{i,n}_{T}(\omega)  {}^{(n)} \hspace{-3pt}\Big\lgroup  
   \bigl\vvvert \bigl( X^{\bullet} - \overline X^{\bullet} \bigr) (\omega) \bigr\vvvert_{[0,T],\widehat w^{\bullet,n},p'}
  \hspace{-1pt} \Big\rgroup_{8}  +   \theta^{i,n}(\omega) \biggr).
\end{split}
\end{equation*}
In the end,
\begin{equation}
\label{eq:passer:i:j}
\begin{split}
&\bigl\vvvert  \bigl( X^i - \overline X^i \bigr)(\omega) \bigr\vvvert_{[0,T],\widehat w^{i,n},p'}   
\\
&\leq c \Bigl[ c  \bigl( 1+ \zeta^{i,n}_{T}(\omega) \bigr) \Bigr]^{{2}\widehat N^{i,n}([0,T],\omega,1/(4L))+1}
\\
&\hspace{30pt} \times \biggl(   \zeta^{i,n}_{T}(\omega) \,
 {}^{(n)} \hspace{-3pt}\Big\lgroup  
  \bigl\vvvert \bigl( X^{\bullet} - \overline X^{\bullet} \bigr) (\omega) \bigr\vvvert_{[0,T],\widehat w^{\bullet,n},p'}
  \hspace{-1pt} \Big\rgroup_{8}
  +  \theta^{i,n}(\omega)   \biggr).
\end{split}
\end{equation}
Hence, using the shorten notation 
$\widehat N^{i,n}_{T}(\omega)$ for $\widehat N^{i,n}([0,T],\omega,1/(4L))$ and recalling that $c \geq 1$, we obtain
\begin{equation}
\label{eq:conclusion:rate:2ndstep}
\begin{split}
& {}^{(n)} \hspace{-3pt} \Big\lgroup
  \bigl\vvvert \bigl( X^{\bullet} - \overline X^{\bullet} \bigr) (\omega) \bigr\vvvert_{[0,T],\widehat w^{\bullet,n},p'}
  \hspace{-1pt} \Big\rgroup_{8}
\\
&\leq
 {}^{(n)} \hspace{-3pt} \Big\lgroup
 \bigl[ c^2 \,  \bigl( 1+ \zeta_{T}^{\bullet,n}(\omega) \bigr) \bigr]^{{2} \widehat N^{\bullet,n}_{T}(\omega)+1}
  \zeta_{T}^{\bullet,n}(\omega)  
  \hspace{-1pt}\Big\rgroup_{8}
 \\
&\hspace{100pt} \times  {}^{(n)} \hspace{-3pt}\Big\lgroup
  \bigl\vvvert \bigl( X^{\bullet} - \overline X^{\bullet} \bigr) (\omega) \bigr\vvvert_{[0,T],\widehat w^{\bullet,n},p'}
  \hspace{-1pt} \Big\rgroup_{8}
\\
&\hspace{15pt}
+
 {}^{(n)} \hspace{-3pt} \Big\lgroup
 \bigl[ c^2 \,  \bigl( 1+ \zeta_{T}^{\bullet,n}(\omega) \bigr) \bigr]^{{2} \widehat N^{\bullet,n}_{T}(\omega)+1}
  \theta^{\bullet,n}(\omega)  
  \hspace{-1pt} \Big\rgroup_{8}.
\end{split}
\end{equation}

\smallskip

\textbf{\textsf{Step 3.}} The key quantity of interest in \eqref{eq:conclusion:rate:2ndstep} is the multiplicative factor in the second line, which we denote by
\begin{equation*}
\begin{split}
&\Psi^n_{T}(\omega)
:=
 {}^{(n)} \hspace{-3pt}\Big\lgroup  
 \bigl[ c^2 \,  \bigl( 1+ \zeta_{T}^{\bullet,n}(\omega) \bigr) \bigr]^{{2} \widehat N_{T}^{\bullet,n}(\omega)+1}
  \zeta_{T}^{\bullet,n}(\omega)  
  \hspace{-1pt} \Big\rgroup_{8}.
  \end{split}
  \end{equation*}
In particular, letting 
\begin{equation*}
\begin{split}
&\Theta_{T}^n(\omega)
:=  {}^{(n)} \hspace{-3pt}\Big\lgroup   \bigl[ c^2 \,  \bigl( 1+ \zeta_{T}^{\bullet,n}(\omega) \bigr) \bigr]^{{2} \widehat N_{T}^{\bullet,n}(\omega)+1}
  \theta^{\bullet,n}(\omega)  
  \hspace{-1pt} \Big\rgroup_{8},
\end{split}
\end{equation*}
we rewrite 
\eqref{eq:conclusion:rate:2ndstep} in the form
\begin{equation}
\label{eq:conclusion:rate:2ndstep:2b}
\begin{split}
{}^{(n)} \hspace{-3pt}\Big\lgroup  
  &\bigl\vvvert \bigl( X^{\bullet} - \overline X^{\bullet} \bigr) (\omega) \bigr\vvvert_{[0,T],\widehat w^{\bullet,n},p'}
  \hspace{-1pt} \Big\rgroup_{8}
\\
&\leq
 \Psi^n_{T}(\omega) \,
  {}^{(n)} \hspace{-3pt} \Big\lgroup
  \bigl\vvvert \bigl( X^{\bullet} - \overline X^{\bullet} \bigr) (\omega) \bigr\vvvert_{[0,T],\widehat w^{\bullet,n},p'}
  \hspace{-1pt} \Big\rgroup_{8} + \Theta_{T}^n(\omega).
\end{split}
\end{equation}
{\it Here comes the key point.} The variable $\omega$ being frozen, assume that we are given a deterministic time horizon $T$ small enough and a deterministic $L\geq L_0$ large enough, such that 
$\gamma/(4L) \leq 1/2$, 
$\Psi^n_{T}(\omega) \leq 1/2$, and \eqref{eq:NL0:coupling} and 
\eqref{eq:NL0:coupling:supple} hold true. Then,
\begin{equation*}
 {}^{(n)} \hspace{-3pt} \Big\lgroup
  \bigl\vvvert \bigl( X^{\bullet} - \overline X^{\bullet} \bigr) (\omega) \bigr\vvvert_{[0,T],\widehat w^{\bullet,n},p'}
  \hspace{-1pt} \Big\rgroup_{8}
\leq 
2 \Theta_{{T}}^n(\omega).
\end{equation*}
The above inequality sounds really close to the desired result, but it is on a small interval $[0,T]$ only. 
The purpose is thus to iterate it in order to cover any given time interval. 
The explicit construction of $T$ and $L$ being postponed to Step 5. 

\medskip

\textbf{\textsf{Step 4.}} In order to iterate in a proper way, we change our notation. While we keep the notation $T$ for the deterministic {global} time horizon given in the statement (without any further requirement that $T \leq 1$), we use the letter $\uptau$ instead of $T$ in the previous analysis. Put differently, $\uptau$ will stand for  {a deterministic} time horizon less than 1 such that $\Psi_{\uptau}$ is small enough. {And then}, we let $\uptau_{0}=\uptau$ and consider a {deterministic} dissection $0=\uptau_{0}<\uptau_{1} < \cdots < \uptau_{M} = T$ of the interval $[0,T]$ into $M$ subintervals.  {\textit{The goal of this step is to 
explain how the error on the interval $[\uptau_{\ell},\uptau_{\ell+1}]$, $\ell=1,\cdots,M-1$
can be controlled in terms of 
the error on the preceding  interval $[0,\uptau_{\ell}]$, the construction of the dissection being achieved in Step 5.}}
\smallskip

To do so, we need to revisit the statement of Proposition \ref{theorem:fixed:1}. 
Assume indeed that we have a bound for 
\begin{equation*}
{\mathcal E}^{i,n}_{\uptau_{\ell}}(\omega)
:= 
\Bigl( 1+ \widehat w^{i,n}_{p'}(0,T,\omega)^{1/p'} \Bigr) \, 
\big\vvvert  
\bigl( X^i - \overline X^i \bigr)(\omega)
\big\vvvert_{[0,\uptau_{\ell}],\widehat w^{i,n},p'},
\end{equation*}
for some $\ell \leq M$. Then, in order to duplicate the  previous two steps, we must consider a new dissection $\uptau_{\ell} = t_{0} < t_{1} < \cdots < t_{K}=\uptau_{\ell+1}$ of the interval $[\uptau_{\ell},\uptau_{\ell+1}]$ with the property that  $K= {2} \widehat N^{i,n}\big([\uptau_{\ell},\uptau_{\ell+1}],\omega,1/(4L)\big) + 1$ and that $\widehat w^{i,n}_{p'}(t_{k},t_{k+1},\omega)  {\leq} 1/(4L)$ if ${k} < K$. The key point is to apply  the first inequality in \eqref{eq:3.16:BCD1} on $[t_{k},t_{k+1}]$ with $(X(\omega),Y(\cdot)) = (X^i(\omega),X^i(\cdot))$ and $(X'(\omega),Y'(\cdot)) = (\overline X^i(\omega),\overline X^i(\cdot))$, but with $\uptau_{\ell}$ instead of $0$ as initial time. {Upper bounding the second line in \eqref{eq:3.16:BCD1}
by ${\mathcal E}^{i,n}_{\uptau_{\ell}}(\omega) + 
{}^{(n)} \hspace{-1pt} \big\lgroup  {\mathcal E}^{\bullet,n}_{\uptau_{\ell}}(\omega)
\big\rgroup_{8}$, we obtain}
\begin{equation*}
\begin{split}
&\biggl\vvvert \int_{t_{k}}^{\cdot} \textrm{F}\bigl(X_{r}^i(\omega),\mu^n_{r}(\omega)
\bigr) d {\boldsymbol W}_{r}^i(\omega) - \int_{t_{k}}^{\cdot} \textrm{F}\bigl(\overline X_{r} ^i(\omega),
\overline \mu^n_{r}(\omega)\bigr) d {\boldsymbol W}_{r}^i(\omega) \biggr\vvvert_{[t_{k},t_{k+1}],\widehat w^{i,n},p'}   
\\
&\leq 
\gamma \, 
\widehat w^{i,n}_{p'}(\uptau_{\ell},\uptau_{\ell+1},\omega)^{1/p'} 
\biggl\{ \big\vvvert  
\bigl( X^i - \overline X^i \bigr)(\omega)
\big\vvvert_{[\uptau_{\ell},t_{k}],\widehat w^{i,n},p'} 
\\
&\hspace{100pt} + 
 {}^{(n)} \hspace{-3pt} \Big\lgroup
  \bigl\vvvert \bigl( X^{\bullet} - \overline X^{\bullet} \bigr) (\omega) \bigr\vvvert_{[\uptau_{\ell},\uptau_{\ell+1}],\widehat w^{\bullet,n},p'}
  \hspace{-1pt} \Big\rgroup_{8}
 \biggr\}   
\\
&\quad+ 
 {\frac{\gamma}{4L}}
  \biggl\{
\big\vvvert \bigl( X^i - \overline X^i \bigr) (\omega)\big\vvvert_{[t_{k},t_{k+1}],\widehat w^{i,n},p'}  
\\
&\hspace{100pt}+ 
 {}^{(n)} \hspace{-3pt}\Big\lgroup 
  \bigl\vvvert \bigl( X^{\bullet} - \overline X^{\bullet} \bigr) (\omega) \bigr\vvvert_{[t_{k},t_{k+1}],\widehat w^{\bullet,n},p'}
  \hspace{-1pt} \Big\rgroup_{8}
 \biggr\}
 \\
 &\quad+ \gamma 
  \biggl[ {\mathcal E}_{\uptau_{\ell}}^{i,n}(\omega) +
 {}^{(n)} 
  \hspace{-1pt}\big\lgroup  
 {\mathcal E}_{\uptau_{\ell}}^{\bullet,n}(\omega) 
  \hspace{-1pt}   \big\rgroup_{8} 
  \biggr],
\end{split}
\end{equation*}
provided the analogues of \eqref{eq:fixed:1} and \eqref{eq:fixed:2} hold true. We may argue as in the second step to check the latter two: They are consequences of Theorem \ref{main:theorem:existence:small:time}, if 
 {$\uptau_{\ell+1}-\uptau_{\ell} \leq 1$ and 
the analogues of \eqref{eq:NL0:coupling}
and
\eqref{eq:NL0:coupling:supple} hold true, namely}
\begin{align}
\label{eq:N:Lambda}
&{}^{(n)} \hspace{-1pt}\Big\lgroup  \widehat N^{\bullet,n}\big(\textcolor{black}{[\uptau_{\ell},\uptau_{\ell+1}]},\omega,1/(4L_{0})\big)    \Big\rgroup_{8} \leq 1,
\\
 &{{}^{(n)} \hspace{-1pt} \biggl\lgroup \Bigl[ \gamma_{0} \Bigl( 1 + \widehat w_{p'}^{\bullet,n}(\uptau_{\ell},\uptau_{\ell+1},\cdot)^{1/p'} \Bigr)\Bigr]^{\widehat N^{\bullet,n}([\uptau_{\ell},\uptau_{\ell+1}],\cdot,1/(4L_{0}))} \biggr\rgroup_{32} \textcolor{black}{\leq \eta_{0}}}.
\label{eq:N:Lambda:supple}
\end{align}
Then, proceeding as in the second step,
\begin{equation*}
\begin{split}
&\bigl\vvvert  \bigl( X^i - \overline X^i \bigr)(\omega) \bigr\vvvert_{[t_{k},t_{k+1}],\widehat w^{i,n},p'}   
\\
&\leq c \, 
\Bigl(  {\frac1L} + 
\widehat w^{i,n}_{p'} (\uptau_{\ell},\uptau_{\ell+1},\omega)^{1/p'} \Bigr) 
\,
\biggl\{\big\vvvert  
\bigl( X^i - \overline X^i \bigr)(\omega)
\big\vvvert_{[\uptau_{\ell},t_{k}],\widehat w^{i,n},p'} 
\\
&\hspace{100pt}+ 
 {}^{(n)} \hspace{-3pt}\Big\lgroup  
  \bigl\vvvert \bigl( X^{\bullet} - \overline X^{\bullet} \bigr) (\omega) \bigr\vvvert_{[\uptau_{\ell},\uptau_{\ell+1}],\widehat w^{\bullet,n},p'}
  \hspace{-1pt} \Big\rgroup_{8}
 \biggr\}   
 \\
 &\hspace{30pt}
 +  {c} \, \biggl\{ {\mathcal E}_{\uptau_{\ell}}^{i,n}(\omega) + {}^{(n)} 
  \hspace{-1pt}\big\lgroup  
 {\mathcal E}_{\uptau_{\ell}}^{\bullet,n}(\omega) 
  \hspace{-1pt}   \big\rgroup_{8} 
 +  \theta^{i,n}(\omega) \biggr\}.
\end{split}
\end{equation*}
In the end, we are in the same situation as in Step 2, but with new $\zeta^{i,n}_{T}$ and 
$\widehat N_{T}^{i,n}$. Here, we let (pay attention that, to be consistent with the notations $\zeta^{i,n}_{T}$ and 
$\widehat N_{T}^{i,n}$, we should use $[\uptau_{\ell},\uptau_{\ell+1}]$ instead of 
$\ell$ as subscript below, but, for simplicity, we prefer to use $\ell$ only)
\begin{equation*}
\begin{split}
&\zeta^{i,n}_{{\ell}}(\omega) :=    {\frac1L} + \widehat w_{p'}^{i,n}(\uptau_{\ell},\uptau_{\ell+1},\omega)^{1/p'},   \quad
 \widehat N^{i,n}_{\ell}(\omega) := \widehat N^{i,n}\left([\uptau_{\ell},\uptau_{\ell+1}],\omega,\frac{1}{4L}\right).
\end{split}
\end{equation*}
Following
\eqref{eq:passer:i:j}, we obtain
\begin{align}
&\bigl\vvvert  \bigl( X^i - \overline X^i \bigr)(\omega) \bigr\vvvert_{[\uptau_{\ell},\uptau_{\ell+1}],\widehat w^{i,n},p'}   
\nonumber
\\
&\leq c \, \bigl[ c  \bigl( 1+ \zeta^{i,n}_{\ell}(\omega) \bigr) \bigr]^{{2}\widehat N^{i,n}_{\ell}(\omega)+1}
\nonumber
\\
&\hspace{15pt} \times 
 \biggl\{ 
 \zeta^{i,n}_{\ell}(\omega) \,
 {}^{(n)} \hspace{-3pt}\Big\lgroup  
  \bigl\vvvert \bigl( X^{\bullet} - \overline X^{\bullet} \bigr) (\omega) \bigr\vvvert_{[\uptau_{\ell},\uptau_{\ell+1}],\widehat w^{\bullet,n},p'}
  \hspace{-1pt} \Big\rgroup_{8} \label{eq:passer:i:j:67}
  \\
 &\hspace{100pt} +   \theta^{i,n}(\omega)
+  {\mathcal E}_{\uptau_{\ell}}^{i,n}
+
{}^{(n)} 
  \hspace{-1pt}\big\lgroup   {\mathcal E}_{\uptau_{\ell}}^{\bullet,n}(\omega) 
  \hspace{-1pt} \big\rgroup_{8} 
\biggr\}.
\nonumber
\end{align}
Hence,
\begin{equation*}
\begin{split}
& {}^{(n)} \hspace{-3pt} \Big\lgroup
  \bigl\vvvert \bigl( X^{\bullet} - \overline X^{\bullet} \bigr) (\omega) \bigr\vvvert_{[\uptau_{\ell},\uptau_{\ell+1}],\widehat w^{\bullet,n},p'}
  \hspace{-1pt} \Big\rgroup_{8}  
\\
&\hspace{15pt} \leq \Psi^n_{\ell}(\omega) \times 
{}^{(n)} \hspace{-3pt}\Big\lgroup 
  \bigl\vvvert \bigl( X^{\bullet} - \overline X^{\bullet} \bigr) (\omega) \bigr\vvvert_{[\uptau_{\ell},\uptau_{\ell+1}],\widehat w^{\bullet,n},p'}
  \hspace{-1pt} \Big\rgroup_{8}  
  + \Theta^n_{\ell}(\omega),
\end{split}
\end{equation*}
with
\begin{equation*}
\begin{split}
&\Psi^n_{\ell}(\omega)
:=
 {}^{(n)} \hspace{-3pt}\Big\lgroup  
 \bigl[ c^2 \,  \bigl( 1+ \zeta_{{\ell}}^{\bullet,n}(\omega) \bigr) \bigr]^{{2} \widehat N^{\bullet,n}_{\ell}(\omega)+1}
  \zeta_{\ell}^{\bullet,n}(\omega)  
  \hspace{-1pt} \Big\rgroup_{8},
  \\
&\Theta^n_{\ell}(\omega)
\\
&\hspace{5pt} :=  {}^{(n)} \hspace{-3pt}\Big\lgroup  
 \bigl[ c^2 \,  \bigl( 1+ \zeta_{\ell}^{\bullet,n}(\omega) \bigr) \bigr]^{{2} \widehat N^{\bullet,n}_{\ell}(\omega)+1}
 \Bigl(
  \theta^{\bullet,n}(\omega)  
  +  {\mathcal E}_{\uptau_{\ell}}^{\bullet,n}(\omega)
+ 
  {}^{(n)} 
  \hspace{-1pt}\big\lgroup
 {\mathcal E}_{\uptau_{\ell}}^{\bullet,n}(\omega) 
  \hspace{-1pt}   \big\rgroup_{8} 
\Bigr)  
  \hspace{-1pt} \Big\rgroup_{8}.
\end{split}
\end{equation*}
{Following 
\eqref{eq:conclusion:rate:2ndstep:2b}}, if we can choose 
{$L$ large enough and then 
$\uptau_{\ell+1}- \uptau_{\ell}$ small enough} such that $\Psi^n_{\ell}(\omega) \leq 1/2$, then we get
\begin{equation*}
{}^{(n)} \hspace{-3pt}\Big\lgroup  
\bigl\vvvert  \bigl( X^{\bullet} - \overline X^{\bullet} \bigr)(\omega) \bigr\vvvert_{[\uptau_{\ell},\uptau_{\ell+1}],\widehat w^{\bullet,n},p'}
  \hspace{-1pt} \Big\rgroup_{8}
\leq 
2 \, 
 \Theta^n_{\ell}(\omega).
\end{equation*}
Eventually, returning to 
\eqref{eq:passer:i:j:67} and modifying the value of the constant $c$, we deduce
\begin{equation*}
\begin{split}
&\bigl\vvvert  \bigl( X^i - \overline X^i \bigr)(\omega) \bigr\vvvert_{[\uptau_{\ell},\uptau_{\ell+1}],\widehat w^{i,n},p'}   
\\
&\hspace{15pt} \leq c \, \bigl[ c  \bigl( 1+ \zeta^{i,n}_{\ell}(\omega) \bigr) \bigr]^{{2} \widehat N^{i,n}_{\ell}(\omega)+1}
\\
&\hspace{30pt}  
 \times
 \biggl( 
 \zeta^{i,n}_{\ell}(\omega) \,
\Theta^n_{\ell}(\omega)
  +   \theta^{i,n}(\omega)
+  {\mathcal E}_{\uptau_{\ell}}^{i,n}
+
{}^{(n)} 
  \hspace{-1pt}\big\lgroup  
 {\mathcal E}_{\uptau_{\ell}}^{\bullet,n}(\omega) 
  \hspace{-1pt}   \big\rgroup_{8} 
\biggr),
\end{split}
\end{equation*}
and then
\begin{equation*}
\begin{split}
&{\mathcal E}_{\uptau_{\ell+1}}^{i,n}(\omega)
 \leq  \kappa^{i,n}_{\ell}(\omega) \biggl( 
 \zeta^{i,n}_{\ell}(\omega) \,
\Theta^n_{\ell}(\omega)
  +   \theta^{i,n}(\omega)
+  {\mathcal E}_{\uptau_{\ell}}^{i,n}(\omega)
+
{}^{(n)} 
  \hspace{-1pt} \big\lgroup
 {\mathcal E}_{\uptau_{\ell}}^{\bullet,n}(\omega) 
  \hspace{-1pt}  \big\rgroup_{8} 
\biggr),
\end{split}
\end{equation*}
with
\begin{equation}
\label{eq:kappa:i:n:ell:omega}
\kappa^{i,n}_{\ell}(\omega) := c^2 \,\Bigl( 1+ \widehat w^{i,n}_{p'}(0,T,\omega)^{1/p'} \Bigr) \, \Bigl[ c^2  \bigl( 1+ \zeta_{\ell}^{i,n}(\omega) \bigr) \Bigr]^{{2} \widehat N^{i,n}_{\ell}(\omega)+1},
\end{equation}
using the fact that $c \geq 1$. By induction, we get the following global bound:
\begin{equation}
\label{eq:estimate:Ein}
\begin{split}
& {\mathcal E}^{i,n}_{\uptau_{\ell+1}}(\omega) 
\leq 
\sum_{k=0}^{\ell}
 {\mathcal K}^{i,n}_{k,\ell}(\omega)   \,
  \biggl[
 \zeta^{i,n}_{k}(\omega) 
 \Theta^n_{k}(\omega)
  +   \theta^{i,n}(\omega)
+
{}^{(n)} 
  \hspace{-1pt} \big\lgroup
 {\mathcal E}_{\uptau_{k}}^{\bullet,n}(\omega) 
  \hspace{-1pt}   \big\rgroup_{8} 
  \biggr],
\end{split}
\end{equation}
with
\begin{equation}
\label{eq:mathcal:kappa:i,n:k,ell}
{\mathcal K}^{i,n}_{k,\ell}(\omega) := 
\prod_{j=k}^{\ell } \kappa^{i,n}_{j}(\omega).
\end{equation}
We deduce that for any $r > 8$, we can find a constant $q(r)$ ({which has nothing to do with $q$ in the assumption}) such that 
\begin{equation*}
\begin{split}
{}^{(n)} 
  \hspace{-1pt}\big\lgroup  
 {\mathcal E}_{\uptau_{\ell+1}}^{\bullet,n}(\omega) 
  \hspace{-1pt} \big\rgroup_{r}
\leq 
\sum_{k=0}^{\ell}
&\biggl\{  {}^{(n)} 
\hspace{-1pt}\big\lgroup  
 {\mathcal K}^{\bullet,n}_{k,\ell} \,  \hspace{-1pt} \big\rgroup_{q(r)}
 \times 
\Bigl(
 1+  
{}^{(n)} 
  \hspace{-1pt}\big\lgroup  
 \widehat w^{\bullet,n} (0,T,\omega)^{1/p'}
  \hspace{-1pt}  \big\rgroup_{q(r)} 
\Bigr)
\\
&\hspace{15pt} \times \Bigl( 1+
{}^{(n)} \hspace{-3pt}\Big\lgroup 
 \bigl[ c^2 \,  \bigl( 1+ \zeta_{k}^{\bullet,n}(\omega) \bigr) \bigr]^{{2} \widehat N^{\bullet,n}_{k}(\omega)+1}
  \hspace{-1pt} \Big\rgroup_{q(r)}
  \Bigr)
  \\
&\hspace{15pt}  \times
\Bigl(
{}^{(n)} 
  \hspace{-1pt}\big\lgroup  
  \theta^{\bullet,n}(\omega)  
  \hspace{-1pt}  \big\rgroup_{q(r)} 
+
{}^{(n)} 
  \hspace{-1pt}\big\lgroup  
 {\mathcal E}_{\uptau_{k}}^{\bullet,n}(\omega) 
  \hspace{-1pt}   \big\rgroup_{r} 
  \Bigr) \biggr\}.
\end{split}
\end{equation*}
{Using the fact that} 
\begin{equation*}
\begin{split}
&{}^{(n)} 
  \hspace{-1pt} \big\lgroup  
 {\mathcal K}^{\bullet,n}_{k,\ell} \,  \hspace{-1pt} \big\rgroup_{q(r)}
\\
&\geq \max \Bigl( 1, 
{}^{(n)} 
  \hspace{-1pt}\big\lgroup  
 \widehat w^{\bullet,n} (0,T,\omega)^{1/p'}
  \hspace{-1pt}  \big\rgroup_{q(r)},{}^{(n)} 
  \hspace{-3pt} \Big\lgroup 
 [ c^2 \,  ( 1+ \zeta_{k}^{\bullet,n}(\omega) ) ]^{{2} \widehat N^{\bullet,n}_{k}(\omega)+1}
  \hspace{-1pt} \Big\rgroup_{q(r)} \Bigr),
\end{split}
\end{equation*}
{we obtain a bound of the form}
$a_{\ell+1} \leq  \sum_{k=0}^{\ell} g_{k,\ell} \bigl( 
b
+a_k 
\bigr)$,
with
\begin{equation*}
\begin{split}
&a_{\ell} := 
{}^{(n)} 
  \hspace{-1pt}\big\lgroup   {\mathcal E}_{\uptau_{\ell}}^{\bullet,n}(\omega) 
  \hspace{-1pt}  \big\rgroup_{r},
\quad g_{k,\ell}:=  
4 \times \Bigl( {}^{(n)} 
  \hspace{-1pt}\big\lgroup  
 {\mathcal K}^{\bullet,n}_{k,\ell} \,
  \hspace{-1pt}   \big\rgroup_{q(r)} \Bigr)^3,
\quad b :=  
{}^{(n)} 
  \hspace{-1pt}\big\lgroup  
  \theta^{\bullet,n}(\omega)  
  \hspace{-1pt}   \big\rgroup_{q(r)}.
\end{split}
\end{equation*}
Hence,
\begin{equation}
\label{eq:bound:for:aell}
a_{\ell}
\leq  b
\sum_{j=1}^{\ell}
\sum_{0 \leq k_{1} \leq \cdots \leq k_{j} \leq k_{j+1} = \ell}
\prod_{h=1}^{j} g_{k_{h},k_{h+1}}.
\end{equation}
Back to \eqref{eq:mathcal:kappa:i,n:k,ell}, we will use below the bound
\begin{equation}
\label{eq:mathcal:K}
\begin{split}
&{\mathcal K}^{i,n}_{k,\ell}(\omega) 
\\
&\leq c^{2(\ell+1- k)}
\prod_{j=k}^{\ell}
\Bigl\{
\bigl( 1+  \widehat w^{i,n}(0,T,\omega)^{1/p'}
\bigr)
 \bigl[ c^2  \bigl( 1+ \zeta^{i,n}_{j}(\omega) \bigr) \bigr]^{{2}\widehat N^{i,n}_{j}(\omega)+1}
 \Bigr\}
 \\
 &\leq c^{4(\ell+1- k) + {4}\widehat N^{i,n}_{k,\ell+1}(\omega)}
 \bigl( 1+ 
\widehat w^{i,n}(0,T,\omega)^{1/p'}
\bigr)^{\ell+1- k + {2} \widehat N^{i,n}_{k,\ell+1}(\omega)},
\end{split}
\end{equation}
with the shortened notation 
$
\widehat N^{i,n}_{k,\ell}(\omega) := \widehat N^{i,n}\big([\uptau_{k},\uptau_{\ell}],\omega,1/(4L)\big)$, 
 {where we have used the fact that 
$\sum_{j=k}^\ell \widehat{N}_{j}^{i,n}(\omega) \leq 
\widehat N^{i,n}_{k,\ell+1}(\omega)$.}

\medskip

\textbf{\textsf{Step 5.}} We now recall that the parameter $L$ and the sequence of  {deterministic} times $0 = \uptau_{0} < \uptau_{1} < \cdots < \uptau_{M}$ must satisfy
 {$\gamma/(4L) \leq 1/2$}, 
 \eqref{eq:N:Lambda},
\eqref{eq:N:Lambda:supple}
and 
{$\Psi^n_{\ell}(\omega) \leq 1/2$,  together with $\uptau_{\ell+1}- \uptau_{\ell} \leq 1$, for $\ell \in \{0,\cdots,M-1\}$}. 

In order to proceed, we let
\begin{equation*}
\begin{split}
M_{1} &= 
\Bigl\langle
\bigl\|
W(\cdot)
\bigr\|_{[0,T],(1/{p'})-\textrm{\rm H}}^{{p'}}
\Bigr\rangle_{q}
+
\Bigl\langle
\bigl\|{\mathbb W}(\cdot) \bigr\|_{[0,T],(2/p')-\textrm{\rm H}}^{p'/2}
\Bigr\rangle_{q}
\\
&\hspace{15pt}
+
\Bigl\langle
\bigl\|
{\mathbb W}^{\indep}(\cdot,\cdot) \bigr\|_{[0,T],(2/p')-\textrm{\rm H}}^{p'/2}
\Bigr\rangle_{q},
\end{split}
\end{equation*}
and we consider the events 
\begin{equation*}
\begin{split}
A^{n}_{1} &= \biggl\{ 
\omega \in \Omega : 
{}^{(n)} \hspace{-1pt} \bigl\lgroup 
\bigl\|
W^{\bullet}(\omega) \bigr\|_{[0,T],(1/p')-\textrm{\rm H}}^{p'}
 \bigr\rgroup_{q}  
 + 
{}^{(n)} \hspace{-1pt} \bigl\lgroup 
\bigl\|
 {\mathbb W}^{\bullet}(\omega) \bigr\|_{[0,T],(2/p')-\textrm{\rm H}}^{p'/2}
   \bigr\rgroup_{q}
   \\
&\hspace{30pt}
 +
 {}^{(n)} \hspace{-1pt}\big\lgroup \hspace{-2pt}\big\lgroup
 \bigl\|
 {\mathbb W}^{\bullet,\bullet}(\omega) \bigr\|_{[0,T],(2/p')-\textrm{\rm H}}^{p'/2}   
   \big\rgroup \hspace{-2pt} \big\rgroup_{q}
\leq M_{1}+1 \biggr\},
\\
A^{n}_{2} &= \biggl\{ 
\omega \in \Omega :  
\forall i \in \{1,\dots,n\}, \ 
\forall (s,t) \in {\mathcal S}_{2}^T, 
\\
&\hspace{30pt} {{}^{(n)} \hspace{-1pt} \big\lgroup} \WW^{i,\bullet}(\omega) \big\rgroup_{q;[s,t],p'/2-\textrm{\rm v}}
 +
 {{}^{(n)} \hspace{-1pt} \big\lgroup} \WW^{\bullet,i}(\omega) \big\rgroup_{q;[s,t],p'/2-\textrm{\rm v}}
    \\
&\hspace{30pt}\leq  
2c_{p,p',q} (t-s)^{2/p'}
\\
&\hspace{45pt} + \big\langle \WW^{i,\indep}(\omega,\cdot) \big\rangle_{q ; [s,t],{p'}/2-\textrm{\rm v}}
+ \big\langle \WW^{i,\indep}(\cdot,\omega) \big\rangle_{q ; [s,t],{p'}/2-\textrm{\rm v}}
\biggr\},
\end{split}
\end{equation*}
{with 
$c_{p,p',q}$, 
as in  \eqref{eq:conclusion:supplementaire}}.
On the event $A^{n}_{1} \cap A^{n}_{2}$, we have (compare with 
\eqref{eq:v:N}
and 
\eqref{eq:v:p:i:N}) for $s,t \in [0,T]^2$, $s < t$, 
\begin{equation*}
\begin{split}
v^{i,n}_{p'}(s,t,\omega) 
&\leq 
\big\| W^i(\omega) \big\|_{[s,t],p'-\textrm{\rm v}}^{p'} 
 + \big\| \WW^i(\omega) \big\|_{[s,t],p'/2-\textrm{\rm v}}^{p'/2} 
 \\
&\hspace{15pt} + 
3^{p'/2-1}
 \biggl( 
  \big\langle \WW^{i,\indep}(\omega,\cdot) \big\rangle_{q ; [s,t],{p'}/2-\textrm{\rm v}}^{p'/2}
+ \big\langle \WW^{i,\indep}(\cdot,\omega) \big\rangle_{q ; [s,t],{p'}/2-\textrm{\rm v}}^{p'/2}
\biggr)
\\ 
&\hspace{15pt}
+
 \bigl( c_{p,p',q}'  +  M_{1} 
 \bigr)(t-s),
\end{split}
\end{equation*}
for a new constant $c_{p,p',q}'$ only depending on $p$, $p'$ and $q$. 
Therefore, introducing the new event
\begin{equation*}
\begin{split}
A_{3}^n &= \biggl\{ {}^{(n)}\hspace{-3pt} \Big\lgroup
\Bigl\langle
\bigl\|
{\mathbb W}^{\bullet,\indep}(\omega,\cdot) \bigr\|_{[0,T],(2/p')-\textrm{\rm H}}^{p'/2}
\Bigr\rangle_{q} 
\hspace{-1pt} \Big\rgroup_{q}
\\
&\hspace{15pt} +
{}^{(n)} \hspace{-3pt} \Big\lgroup
\Bigl\langle
\bigl\| {\mathbb W}^{\bullet,\indep}(\cdot,\omega) \bigr\|_{[0,T],(2/p')-\textrm{\rm H}}^{p'/2}
\Bigr\rangle_{q}
\hspace{-1pt} \Big\rgroup_{q} \leq 2 \bigl( 1+ M_{1} \bigr)
\biggr\},
\end{split}
\end{equation*}
we get, on 
$A^{n}_{1}  \cap A_{2}^n \cap A^n_{3}$,
\begin{equation*}
\begin{split}
 {}^{(n)} \hspace{-1pt} \big\lgroup  v_{p'}^{\bullet,n}(\omega) \big\rgroup_{q ; [s,t],1-\textrm{\rm v}} \leq  
 c_{p,p',q,M_{1}}'
(t-s),
\end{split}
\end{equation*}
for a new constant $c_{p,p',q,M_{1}}'$ only depending on $p$, $p'$, $q$ and $M_{1}$. 

Recall
now the definition of $\widehat{v}_{p'}^{i,n}$ in 
 \eqref{eq:w:widehat:w:v:widehat:v}. We have
\begin{equation*}
\begin{split}
&\widehat{v}_{p'}^{i,n}(s,t,\omega) 
\\
&=   \big\langle \WW^{i,\indep}(\omega,\cdot) \big\rangle_{q ; [s,t],p'/2-\textrm{\rm v}}^{p'/2}   
 + \big\langle \WW^{i,\indep}(\cdot,\omega) \big\rangle_{q ; [s,t],p'/2-\textrm{\rm v}}^{p'/2}
 \\
 &\leq 
\Bigl[ \Bigl\langle 
 \bigl\|
{\mathbb W}^{i,\indep}(\omega,\cdot) \bigr\|_{[0,T],(2/p')-\textrm{\rm H}}^{p'/2}
\Bigr\rangle_{q}
+
\Bigl\langle 
 \bigl\|
{\mathbb W}^{i,\indep}(\cdot,\omega) \bigr\|_{[0,T],(2/p')-\textrm{\rm H}}^{p'/2}
\Bigr\rangle_{q}
\Bigr] (t-s).
\end{split}
\end{equation*}
And then,
\begin{equation*}
\begin{split}
 {}^{(n)} \hspace{-1pt} \big\lgroup  \widehat{v}_{p'}^{\bullet,n}(\omega) \big\rgroup_{q ; [s,t],1-\textrm{\rm v}}
&\leq  
\biggl(
{}^{(n)} \hspace{-3pt} \Big\lgroup
\Bigl\langle
\bigl\|
{\mathbb W}^{\bullet,\indep}(\omega,\cdot) \bigr\|_{[0,T],(2/p')-\textrm{\rm H}}^{p'/2}
\Bigr\rangle_{q} 
\hspace{-1pt} \Big\rgroup_{q}
\\
&\hspace{15pt} +
{}^{(n)} \hspace{-3pt} \Big\lgroup
\Bigl\langle
\bigl\| {\mathbb W}^{\bullet,\indep}(\cdot,\omega) \bigr\|_{[0,T],(2/p')-\textrm{\rm H}}^{p'/2}
\Bigr\rangle_{q}
\hspace{-1pt} \Big\rgroup_{q}
\biggr) (t-s).
\end{split}
\end{equation*}
Therefore, on the event $ A^{n}_{1}  \cap A_{2}^n \cap A^n_{3}$, we have 
$$ 
{}^{(n)} \hspace{-1pt} \big\lgroup  \widehat v_{p'}^{\bullet,n}(\omega) \big\rgroup_{q ; [s,t],1-\textrm{\rm v}}
\leq 2(1+M_{1}) (t-s).
$$ 
Using the same notations as in  \eqref{eq:w:widehat:w:v:widehat:v}, we end-up with 
 \begin{equation}
 \label{eq:widetile:w}
 \begin{split}
 \widehat{w}^{i,n}_{p'}(s,t,\omega) 
 &\leq 
\widetilde{w}^{i,n}_{p'}(s,t,\omega),
\end{split}
\end{equation}
for $\omega \in 
A^{n}_{1}  \cap A_{2}^n \cap A^n_{3}$, 
where we let (using the fact that $3^{p'/2-1} \leq \sqrt{3} \leq 2$)
 \begin{equation*}
 \begin{split}
\widetilde{w}^{i,n}_{p'}(s,t,\omega)
&:= 
 \big\| W^i(\omega) \big\|_{[s,t],p'-\textrm{\rm v}}^{p'} 
 + \big\| \WW^i(\omega) \big\|_{[s,t],p'/2-\textrm{\rm v}}^{p'/2}
 + C_{p,p',q,M_{1}} (t-s)
\\
&\hspace{15pt}
+ 3 \big\langle \WW^{i,\indep}(\omega,\cdot) \big\rangle_{q ; [s,t],p'/2-\textrm{\rm v}}^{p'/2}   
 + 3 \big\langle \WW^{i,\indep}(\cdot,\omega) \big\rangle_{q ; [s,t],p'/2-\textrm{\rm v}}^{p'/2},
 \end{split} 
 \end{equation*}
with 
$C_{p,p',q,M_{1}} :=c_{p,p',q,M_{1}}'+c_{p,p',q}'+3 M_{1}+3$.
Using the notation \eqref{eq:N:s:t:omega}, we also let 
$
\widetilde{N}^{i,n}([\uptau,\uptau'],\omega,\alpha) := N_{\varpi}([\uptau,\uptau'],\alpha),
$
with $\varpi := (\widetilde w_{p'}^{i,n}(\omega))^{1/p'}$. By \eqref{eq:widetile:w}, 
\begin{equation}
\label{eq:widetildeNgeqwidehatN}
\widehat{N}^{i,n}([\uptau,\uptau'],\omega,\alpha) \leq \widetilde{N}^{i,n}([\uptau,\uptau'],\omega,\alpha),
\end{equation}
for $\omega \in A^{n}_{1}  \cap  A^n_{2} \cap A^n_{3}$. The good point here is that the variables  $\bigl(\widetilde{w}^{i,n}_{p'} \bigr)_{1 \leq i \leq n}$ are independent whilst the variables $\bigl( \widehat{w}^{i,n}_{p'} \bigr)_{1 \leq i \leq n}$
are not. Similarly, whenever $\uptau$ and $\uptau'$ are deterministic, the variables $\bigl(\widetilde{N}^{i,n}([\uptau,\uptau'],\cdot,\alpha)
 \bigr)_{1 \leq i \leq n}$ are independent. Moreover, it is not difficult to see that 
\begin{equation}
\label{eq:domination:widetildewin:widehatwin}
\widetilde{w}^{i,n}_{p'}(s,t,\omega) \leq 3 \widehat{w}^{i,n}_{p'}(s,t,\omega) + 
C_{p,p',q,M_{1}} (t-s),
\end{equation}
from which we deduce, see for instance \eqref{eq:N1:N2}, that, for any $\alpha >0$,  
\begin{equation}
\label{eq:domination:widetildewin:widehatwin:2}
\widetilde{N}^{i,n}\bigl([\uptau,\uptau'],\omega,\alpha\bigr) \leq  \widehat{N}^{i,n}\Bigl([\uptau,\uptau'],\omega,\frac{\alpha}{2 \cdot 3^{1/p'}} \Bigr) +  {C_{\alpha} \bigl( \uptau' - \uptau \bigr)},
\end{equation}
for a constant $C_{\alpha}$ only depending on $\alpha$ and on $p$, $p'$, $q$ and $M_{1}$ (we feel easier not to indicate the dependence on $p$, $p'$, $q$ and $M_{1}$). In particular, we can easily replace $\widehat{N}^{i,n}$ by $\widetilde{N}^{i,n}$ in the third item of the assumption of the statement. Moreover, by 
{\eqref{eq:exp:integrability}}, we deduce that each $\widetilde{w}^{i,n}_{p'}$ satisfies the first bound in \eqref{EqTailAssumptions}, uniformly in $i$ and $n$. 

\smallskip

We now claim that we can choose $L$ and the sequence $0=\uptau_{0} < \cdots < \uptau_{M}=T$, independently of $n$, such that 
\begin{align}
&\Bigl\langle \widetilde{N}^{1,n} \bigl([\uptau_{\ell},\uptau_{\ell+1}],\cdot,1/(4L_{0}) \bigr) \Bigr\rangle_{8} \leq \frac{1}{2}, \label{eq:dissection:1stbound:tocheck}
\\
&\biggl\langle \Bigl[ \gamma_{0} \Bigl( 1 + \widetilde w^{1,n}_{p'}(0,T,\cdot)^{1/{p'}} \Bigr)\Bigr]^{\widetilde N^{1,n}([\uptau_{\ell},\uptau_{\ell+1}],\cdot,1/(4L_{0}))} \biggr\rangle_{32} \leq \frac{1+\eta_{0}}2,\label{eq:dissection:2ndbound:tocheck}
\\
&\biggl\langle  
 \Bigl[ c^2 \,  \Bigl( 1 + \widetilde w^{1,n}_{p'}(0,T,\cdot)^{1/{p'}} \Bigr) \Bigr]^{{2} \widetilde N^{1,n}([\uptau_{\ell},\uptau_{\ell+1}],\cdot,1/(4L_{0}))+1} \nonumber
 \\
 &\hspace{45pt} \times
\Bigl( \frac1L + \widetilde w^{1,n}_{p'}(\uptau_{\ell},\uptau_{\ell+1},\cdot)^{1/{p'}}
\Bigr)   
 \biggr\rangle_{8} \leq \frac14,
 \label{eq:dissection:3rdbound:tocheck}
\end{align} 
for $\gamma_{0}$, $\eta_{0}$ and $c$ as in Steps 2 and 4 (in particular, $c$ is independent of $n$ and $L$).
{We then make use of the third item in the assumption of the statement, see 
\eqref{eq:third:item}; basically, it says that, in all the three constraints, we should normalize $\widetilde{N}^{1,n}\bigl([\uptau_{\ell},\uptau_{\ell+1}],\cdot,1/(4L_{0})\bigr)$ by the root of $\uptau_{\ell+1}-\uptau_{\ell}$. This indeed makes sense since inequality  
\eqref{eq:domination:widetildewin:widehatwin}
insures that 
$\widetilde{N}^{1,n}\bigl([\uptau_{\ell},\uptau_{\ell+1}],\cdot,1/(4L_{0})\bigr)$ satisfies a similar 
estimate as
$\widehat{N}^{1,n}\bigl([\uptau_{\ell},\uptau_{\ell+1}],\cdot,1/(4L_{0})\bigr)$
in 
\eqref{eq:third:item}. 
As for 
\eqref{eq:dissection:1stbound:tocheck}, 
\eqref{eq:third:item}
says that the left-hand side can be bounded by $C \sqrt{\uptau_{\ell+1} - \uptau_{\ell}}$, for a constant $C$ independent of $n$, which makes it possible to choose $\uptau_{\ell+1}-\uptau_{\ell}$ small enough (independently of $n$) such that 
the constraint 
\eqref{eq:dissection:1stbound:tocheck}
is indeed satisfied. 
Recalling that $\uptau_{\ell+1}-\uptau_{\ell}$ is less than 1
and invoking Cauchy-Schwarz inequality, the left-hand side in
\eqref{eq:dissection:2ndbound:tocheck}
can be bounded by $C^{\sqrt{\uptau_{\ell+1}-\uptau_{\ell}}}$; 
the way the constant $C$ here shows up is made clear in 
\eqref{eq:wibull:exp:00} below, with $\delta_{\ell}'$ therein 
being here understood as $\uptau_{\ell+1}-\uptau_{\ell}$ 
and the function $f$ being lower bounded by the identity. 
Importantly, the application of 
\eqref{eq:wibull:exp:00} is made possible by the upper bounds we have for 
$\widetilde w^{1,n}_{p'}$ and 
$\widetilde N^{1,n}$ in terms 
of $\widehat w^{1,n}_{p'}$
and 
$\widehat N^{1,n}$, see 
\eqref{eq:domination:widetildewin:widehatwin} and 
\eqref{eq:domination:widetildewin:widehatwin:2}.  
As for 
\eqref{eq:dissection:3rdbound:tocheck}, it may be bounded, using the same argument together with an additional Cauchy-Schwarz argument, by a product of the form $C^{\sqrt{\uptau_{\ell+1}-\uptau_{\ell}}}  \times 
\bigl( \frac1L 
+ \bigl\langle  \widetilde w^{1,n}_{p'}(\uptau_{\ell},\uptau_{\ell+1},\cdot)^{1/{p'}}
 \bigr\rangle_{16} \bigr)$. The first factor 
 $C^{\sqrt{\uptau_{\ell+1}-\uptau_{\ell}}}$ can be made smaller than $2$ by 
choosing $\uptau_{\ell+1}-\uptau_{\ell}$ small enough. Then, we can take $L \geq 16$ so that 
$C^{\sqrt{\uptau_{\ell+1}-\uptau_{\ell}}} /L$ is less than $1/8$. It then remains to decrease 
$\uptau_{\ell+1}-\uptau_{\ell}$ if necessary to render 
$\bigl\langle  \widetilde w^{1,n}_{p'}(\uptau_{\ell},\uptau_{\ell+1},\cdot)^{1/{p'}}
 \bigr\rangle_{16}$ less than $1/16$; this is possible by using the 
 analogue of 
 \eqref{eq:v:p:i:N} 
but for $\widehat w^{i,n}_{p'}$, as it says that 
$\widetilde w^{1,n}_{p'}(\uptau_{\ell},\uptau_{\ell+1},\cdot)^{1/ {p'}}$ scales like
$(\uptau_{\ell+1}-\uptau_{\ell})^{1/p'}$, the (random) scaling factor having moments of any order that are bounded independently of $n$.}
Importantly, this discussion says that the number $M$ of intervals in the dissection can be chosen independently of $n$. For sure, the index $1$ in the left-hand side
in the three constraints 
\eqref{eq:dissection:1stbound:tocheck},
\eqref{eq:dissection:2ndbound:tocheck}
and
\eqref{eq:dissection:3rdbound:tocheck}
can be replaced by any $i \in \{1,\cdots,n\}$. We then consider the family of events 
\begin{equation*}
\begin{split}
&A_{4}^{\ell,n} = A_{4,1}^{\ell,n} \cap A_{4,2}^{\ell,n} \cap A_{4,3}^{\ell,n},
\quad \ell=0,\cdots,M-1, 
\\
&A_{4,1}^{\ell,n} = \Bigl\{ 
{}^{(n)} \hspace{-3pt} \hspace{-1pt}\Big\lgroup 
\widetilde{N}^{\bullet,n} \bigl([\uptau_{\ell},\uptau_{\ell+1}],\cdot,1/(4L_{0}) \bigr)
\hspace{-1pt}
\Big\rgroup_{8}
 \leq 1
 \Bigr\}, 
 \\
 &A_{4,2}^{\ell,n} = \biggl\{ 
{}^{(n)} \hspace{-3pt} \hspace{-1pt}\bigg\lgroup
  \Bigl[ \gamma_{0} \Bigl( 1 + \widetilde w_{p'}^{\bullet,n}(0,T,\cdot)^{1/ {p'}} \Bigr)\Bigr]^{\widetilde N^{\bullet,n}([\uptau_{\ell},\uptau_{\ell+1}],\cdot,1/(4L_{0}))} 
\hspace{-1pt}
\bigg\rgroup_{32} \leq 
1+\eta_{0}
 \biggr\}, 
 \\
 &A_{4,3}^{\ell,n} = 
  \biggl\{ 
{}^{(n)} \hspace{-3pt} \hspace{-1pt}\bigg\lgroup
 \Bigl[ c^2 \,  \Bigl( 1 + \widetilde w_{p'}^{\bullet,n}(0,T,\cdot)^{1/ {p'}} \Bigr) \Bigr]^{{2} \widetilde N^{\bullet,n}([\uptau_{\ell},\uptau_{\ell+1}],\cdot,1/(4L_{0}))+1} \nonumber
 \\
 &\hspace{145pt} \times
\Bigl( \frac1L + \widetilde w_{p'}^{\bullet,n}(\uptau_{\ell},\uptau_{\ell+1},\cdot)^{1/ {p'}}
\Bigr)   
\hspace{-1pt}
\bigg\rgroup_{32} \leq 
\frac12
 \biggr\}.
 \end{split}
 \end{equation*}
By 
 \eqref{eq:widetile:w}
and
\eqref{eq:widetildeNgeqwidehatN}, 
on $A^{n}_{1} \cap   A^{n}_{2} \cap A^n_{3} \cap \bigl( \bigcap_{\ell=0}^{M-1} A^{\ell,n}_{4}\bigr)$, 
the upper bounds \eqref{eq:N:Lambda}, 
\eqref{eq:N:Lambda:supple}
and
$\Psi^n_{\ell} \leq 1/2$
are satisfied and then the conclusion of the fourth step holds true. 

Following \eqref{eq:kappa:i:n:ell:omega}, this prompts us to set:
\begin{equation*}
\begin{split}
\widetilde \kappa^{i,n}_{\ell}(\omega) &:= c^2 \,\Bigl( 1+ \widetilde w^{i,n}_{p'}(0,T,\omega)^{1/p'} \Bigr) 
\\
&\hspace{15pt} \times \Bigl[ c^2  \Bigl( 1+ {\frac1L} + \widetilde w^{i,n}_{p'}(\uptau_{\ell},\uptau_{\ell+1},\omega)^{1/p'} \Bigr) \Bigr]^{{2} \widetilde N^{i,n}_{\ell}(\omega)+1},
\end{split}
\end{equation*}
and then
$\widetilde {\mathcal K}^{i,n}_{k,\ell}(\omega) := 
\prod_{j=k}^{\ell } \widetilde\kappa^{i,n}_{j}(\omega)$. 
Returning to the conclusion of the fourth step, we get, for $
\omega \in A^n:=  A^{n}_{1}  \cap A^{n}_{2} \cap A^n_{3} \cap  \bigl( \bigcap_{\ell=0}^{M-1} A^{\ell,n}_{4}\bigr)$, 
\begin{equation*}
\begin{split}
&{}^{(n)} 
  \hspace{-1pt}\big\lgroup   {\mathcal E}_{\uptau_{\ell}}^{\bullet,n}(\omega) 
  \hspace{-1pt}  \big\rgroup_{r}
\\
&\leq  {}^{(n)} 
  \hspace{-1pt}\big\lgroup  
  \theta^{\bullet,n}(\omega)  
  \hspace{-1pt}   \big\rgroup_{q(r)}
\sum_{j=1}^{\ell}
\sum_{0 \leq k_{1} \leq \cdots \leq k_{j} \leq k_{j+1} = \ell}
\prod_{h=1}^{j}
4 \times \Bigl( {}^{(n)} 
  \hspace{-1pt}\big\lgroup  
 \widetilde {\mathcal K}^{\bullet,n}_{k_{h},k_{h+1}} (\omega) \,
  \hspace{-1pt}   \big\rgroup_{q(r)} \Bigr)^3
 \\ 
  &\leq \ell 2^{2\ell+1} \times 4^{\ell} \times {}^{(n)} 
 \hspace{-1pt}\big\lgroup  
  \theta^{\bullet,n}(\omega)  
  \hspace{-1pt}   \big\rgroup_{q(r)}
\times \Bigl( {}^{(n)} 
  \hspace{-1pt}\big\lgroup  
 \widetilde {\mathcal K}^{\bullet,n}_{0,\ell} (\omega) \,
  \hspace{-1pt}   \big\rgroup_{q(r)} \Bigr)^{3\ell}. 
   \end{split}
\end{equation*}
The key fact here is that $\widetilde {\mathcal K}^{i,n}_{0,M} (\omega)$, for any $i \in \{1,\cdots,n\}$, has finite moments of any order, independently of $i$ and $n$. The proof follows from \eqref{eq:mathcal:K}, from \eqref{eq:exp:integrability} and  from the 
third item in the assumption of the statement of Theorem \ref{theorem:rate:cv}, the last two properties implying that 
$\bigl( 1+ \widetilde w^{i,n}_{p'}(\uptau_{\ell},\uptau_{\ell+1},\omega)^{1/p'} \bigr)^{\widetilde N^{i,n}_{\ell}(\omega)}$
has finite moments of any order, independently of $i$ and $n$, see for instance \eqref{eq:wibull:exp:22}
 below. Hence, for a constant $C$, independent of $n$ but possibly depending on $M$, we get 
 (using an obvious exchangeability argument)
\begin{equation*}
\begin{split}
\Bigl\langle {\mathbf 1}_{A^n}(\cdot) \, {}^{(n)} 
  \hspace{-1pt}\big\lgroup   {\mathcal E}_{T}^{\bullet,n}(\cdot) 
  \hspace{-1pt}  \big\rgroup_{r}
  \Bigr\rangle_{r}
&\leq C  
 \Bigl\langle {}^{(n)} \hspace{-1pt}\big\lgroup  
  \theta^{\bullet,n}(\cdot)  
  \hspace{-1pt}   \big\rgroup_{q(r)}
  \Bigr\rangle_{2r} \leq C  
  \bigl\langle \theta^{1,n}(\cdot) \bigr\rangle_{q(r)},
   \end{split}
\end{equation*}
where we took, without any loss of generality, $q(r) \geq 2r$. Taking $\varrho = q(r)$ in \eqref{eq:controls:empirical}, we get that, for a constant $C$ independent of $n$, but depending on $r$, 
\begin{equation}
\label{eq:conclusion:5th}
\sup_{1 \leq i \leq n}
\Bigl\langle 
{\mathbf 1}_{A^n}(\cdot)
\big\vert   
\bigl( X^i - \overline X^i \bigr)(\cdot)
\bigr\vert
\Big\rangle_{r}  \leq C \varsigma_{n}.
\end{equation}

\textbf{\textsf{Step 6.}} From the law of large of numbers and from \eqref{eq:exp:integrability},  we claim that ${\mathbb P}((A_{1}^n)^{\complement})$ decays faster than any $n^{-s}$, for $s>0$. The first step of the proof is to notice that 
\begin{equation*}
\begin{split}
&{\mathbb P}\Bigl(\bigl(A_{1}^n\bigr)^{\complement}\Bigr)
\\
&\leq {\mathbb P} \Bigl(  
\omega : 
{}^{(n)} \hspace{-1pt} \bigl\lgroup 
\bigl\|
W^{\bullet}(\omega) \bigr\|_{[0,T],(1/p')-\textrm{\rm H}}^{p'}
 \bigr\rgroup_{q}  
 - 
\Bigl\langle 
\bigl\|
W(\cdot) \bigr\|_{[0,T],(1/p')-\textrm{\rm H}}^{p'}
 \Bigr\rangle_{q}
 \geq \frac13 \Bigr)
 \\
 &\hspace{5pt} + 
{\mathbb P} \Bigl( \omega : 
{}^{(n)} \hspace{-1pt} \bigl\lgroup 
\bigl\|
 {\mathbb W}^{\bullet}(\omega) \bigr\|_{[0,T],(2/p')-\textrm{\rm H}}^{p'/2}
   \bigr\rgroup_{q}
   -
\Bigl\langle 
\bigl\|
 {\mathbb W}(\cdot) \bigr\|_{[0,T],(2/p')-\textrm{\rm H}}^{p'/2}
\Bigr\rangle_{q}
 \geq \frac13 \Bigr)
 \\
 &\hspace{5pt} + {\mathbb P} \Bigl(\omega : 
  {}^{(n)} \hspace{-1pt}\big\lgroup \hspace{-2pt}\big\lgroup
 \bigl\|
 {\mathbb W}^{\bullet,\bullet}(\omega) \bigr\|_{[0,T],(2/p')-\textrm{\rm H}}^{p'/2}   
   \big\rgroup \hspace{-2pt} \big\rgroup_{q}
-
\Bigl\langle  
 \bigl\|
 {\mathbb W}^{\indep}(\cdot,\cdot) \bigr\|_{[0,T],(2/p')-\textrm{\rm H}}^{p'/2}   
\Bigr\rangle_{q}
 \geq \frac13 \Bigr)
 \\
 &=: \pi_{1,1}^n + \pi_{1,2}^n + \pi_{1,3}^n.
\end{split}
\end{equation*}
Since the most difficult term is the last one, we just explain how to handle it. The other two terms may be treated in the same way. Since $q \geq 1$, we first observe that 
\begin{equation*}
\begin{split}
&\pi_{1,3}^n 
\\
&\leq {\mathbb P} \Bigl(\omega :
  {}^{(n)} \hspace{-1pt}\big\lgroup \hspace{-2pt}\big\lgroup
 \bigl\|
 {\mathbb W}^{\bullet,\bullet}(\omega) \bigr\|_{[0,T],(2/p')-\textrm{\rm H}}^{p'/2}   
   \big\rgroup \hspace{-2pt} \big\rgroup_{q}^q
\geq 
\Bigl\langle  
 \bigl\|
 {\mathbb W}^{\indep}(\cdot,\cdot) \bigr\|_{[0,T],(2/p')-\textrm{\rm H}}^{p'/2}   
\Bigr\rangle_{q}^q
 + \frac1{3^q} \Bigr)
\\
 &\leq {\mathbb P} \biggl( \omega : 
\frac1{n^2} \sum_{i \not= j}   
\Bigl(  
 \bigl\|
 {\mathbb W}^{i,j}(\omega) \bigr\|_{[0,T],(2/p')-\textrm{\rm H}}^{qp'/2}   
\hspace{-2pt} -
\Bigl\langle  
 \bigl\|
 {\mathbb W}^{\indep}(\cdot,\cdot) \bigr\|_{[0,T],(2/p')-\textrm{\rm H}}^{qp'/2}   
\Bigr\rangle_{1}
\Bigr)
\geq  {\frac1{3^{q+1}}} \hspace{-2pt} \biggr)
 \\
 &\hspace{15pt} + 
 {\mathbb P} \biggl( \omega : 
\frac1{n^2} \sum_{i=1}^n    
 \bigl\|
 {\mathbb W}^{i}(\omega) \bigr\|_{[0,T],(2/p')-\textrm{\rm H}}^{qp'/2}   
\geq 
 \frac1{3^{q+1}} \biggr).
\end{split}
\end{equation*}
By \eqref{eq:exp:integrability},  the last term in the right-hand side is easily handled. As for the first one, 
Markov's inequality yields, for any $s>1$, 
\begin{equation*}
\begin{split}
&{\mathbb P} \biggl( \omega : 
\frac1{n^2} \sum_{j \not = i}  
\Bigl(  
 \bigl\|
 {\mathbb W}^{i,j}(\omega) \bigr\|_{[0,T],(2/p')-\textrm{\rm H}}^{qp'/2}   
-
\Bigl\langle  
 \bigl\|
 {\mathbb W}^{\indep}(\cdot,\cdot) \bigr\|_{[0,T],(2/p')-\textrm{\rm H}}^{qp'/2}   
\Bigr\rangle_{1}
\Bigr)
\geq {\frac1{3^{q+1}}} \biggr)
\\
&\leq {\frac{3^{s(q+1)}}{n^{s+1}}} \sum_{i=1}^n 
{\mathbb E}
\biggl[ \biggl\vert \sum_{{j : j \not = i}}
\Bigl(
 \bigl\|
 {\mathbb W}^{i,j}(\cdot) \bigr\|_{[0,T],(2/p')-\textrm{\rm H}}^{qp'/2}   
\hspace{-2pt} -
\Bigl\langle  
 \bigl\|
 {\mathbb W}^{\indep}(\cdot,\cdot) \bigr\|_{[0,T],(2/p')-\textrm{\rm H}}^{qp'/2}   
\Bigr\rangle_{1}
\Bigr) \biggr\vert^s \biggr].
\end{split}
\end{equation*}
By \eqref{eq:exp:integrability} again and by  Rosenthal's inequality, see \cite{rosenthal1972}, we deduce that the right-hand side is less than $C n^{-s/2}$, for a constant $C$ independent of $n$. This completes the proof of our claim. 

\smallskip

The same result holds for ${\mathbb P}\bigl((A_{3}^n)^{\complement}\bigr)$. Also, since $\bigl(\widetilde{N}^{i,n}([\uptau_{\ell},\uptau_{\ell+1}],\cdot,\alpha)
 \bigr)_{1 \leq i \leq n}$, are independent for any $\ell=0,\cdots,M-1$ and have finite moments of any order that can be bounded independently of $n$, we also have that ${\mathbb P}\bigl((A_{4,1}^{\ell,n})^{\complement}\bigr)$ decays faster than any $n^{-s}$, for any $\ell =0,\cdots,M-1$.  
 Invoking once again 
 \eqref{eq:wibull:exp}, the same holds for 
 ${\mathbb P}\bigl((A_{4,2}^{\ell,n})^{\complement}\bigr)$
and
 ${\mathbb P}\bigl((A_{4,3}^{\ell,n})^{\complement} \bigr)$ and hence
 for
  ${\mathbb P}\bigl((A_{4}^{\ell,n})^{\complement} \bigr)$.
  Since $M$ is finite,  
  ${\mathbb P}\bigl((A_{4}^{n})^{\complement}\bigr)$
  also decays faster than any $n^{-s}$. 
\smallskip

We finally check that the same is true for $A_{2}^{n}$. In fact, this
is a consequence 
of \eqref{eq:conclusion:supplementaire}, choosing therein 
$\delta$ first in terms of $p$, $p'$ and $q$, and then 
$a$ in terms of $\delta$ such that $a \delta/4=s$.  
{In the end, 
$C_{a}$ 
in 
\eqref{eq:conclusion:supplementaire}
depends on $p$, $p'$, $q$ and $s$.}

\smallskip

All in all, back to the definition of $A^n$ at the end of the fifth step, see \eqref{eq:conclusion:5th}, we deduce that, for any $s>0$, 
${\mathbb P}\bigl((A^{n})^{\complement}\bigr) \leq C n^{-s}$. Therefore, in order to conclude, it suffices to prove that, for any $r \geq 1$, we can choose $q(r) \geq 8$ such that, if $X_{0}(\cdot)$ is in $\LL^{q(r)}$, then 
\begin{equation}
\label{eq:moment:n:lp}
\sup_{1 \leq i \leq n} {\mathbb E} \Bigl[ \sup_{0 \leq t \leq T} \vert X_{t}^{i} \vert^r 
\Bigr] \leq C(r),
\end{equation}
for a constant $C(r)$ depending on $r$ and on $\langle X_{0}(\cdot) \rangle_{q(r)}$ but independent of $n$. 
{Obviously, 
\eqref{eq:moment:n:lp}
implies the same inequality but for 
$\overline X^i$
by (say for $i=1$)
letting $n$ tend to $\infty$ and then invoking 
Theorem \ref{theorem:prop:of:chaos}}.

The proof of \eqref{eq:moment:n:lp} relies on the final estimate in the statement of 
Theorem \ref{main:theorem:existence:small:time}. 
To make it clear, we consider a new random dissection $0=\uptau_{0} < \uptau_{1} < \cdots < \uptau_{M}
=T$ of $[0,T]$ (for simplicity, we use the same notation as in the previous step, but the new dissection has in fact nothing to do with the first one; in particular, it is random) such that 
\begin{equation}
\label{eq:constraint:3:b}
 \begin{split} 
&\Gamma_{1}^{(n)}\bigl(\omega,[\uptau_{\ell},\uptau_{\ell+1}] \bigr)
:= {}^{(n)} \hspace{-3pt} \Big\lgroup
\widehat N^{\bullet,n}\bigl(\textcolor{black}{[\uptau_{\ell},\uptau_{\ell+1}]},\omega,1/(4L_{0})\bigr)
 \Big\rgroup_{8} \leq  {1},
\\
&\Gamma_{2}^{(n)}\bigl(\omega,[\uptau_{\ell},\uptau_{\ell+1}] \bigr)
\\
&:={}^{(n)} \hspace{-3pt}\Big\lgroup  
 \bigl[ \gamma_{0}^2 \,  \bigl( 1+ \widehat w^{\bullet,n}(0,T,\omega)^{1/p'} \bigr) \bigr]^{\widehat N^{\bullet,n}\bigl(\textcolor{black}{[\uptau_{\ell},\uptau_{\ell+1}]},\omega,1/(4L)\bigr)
}
  \hspace{-1pt}\Big\rgroup_{32} \leq \eta_{0}, 
  \end{split}
  \end{equation}
 for the same constants as in the statement of Theorem  \ref{main:theorem:existence:small:time}. We deduce from Theorem  \ref{main:theorem:existence:small:time} that  there exists a constant $C$ (independent of $n$) such that, for any $i \in \{1,\cdots,n\}$ and $\ell \in \{0,\cdots,M-1\}$,
\begin{equation}
\label{eq:vvvertXivvvert}
\begin{split}
&\vvvert X^{i}(\omega) \vvvert_{[\uptau_{\ell},\uptau_{\ell+1}],\widehat w^{i,n},p'} 
\\
&\hspace{15pt} \leq \Bigl[C 
 \Bigl( 1 + \widehat w^{i,n}_{p'}(0,T,\omega)^{1/p'} \Bigr)\Bigr]^{2 \widehat N^{i,n}([0,T],\omega,1/(4L))}.
 \end{split}
 \end{equation}
 Observe now that, for any $i \in \{1,\cdots,n\}$,
 \begin{equation*}
 \begin{split}
\sup_{0 \le t \le T} \vert X_{t}^{i} - X_{0}^i \vert &\leq \sum_{\ell=0}^{M-1} \Bigl( \vvvert X^i(\omega)\vvvert_{[\uptau_{\ell},\uptau_{\ell+1}],\widehat w^{i,n},p'}
 \widehat w^{i,n}_{p'}(\uptau_{\ell},\uptau_{\ell+1},\omega)^{1/p'}\Bigr)
 \\
 &\leq  M \Bigl[C 
 \Bigl( 1 + \widehat w^{i,n}_{p'}(0,T,\omega)^{1/p'} \Bigr)\Bigr]^{2 \widehat N^{i,n}([0,T],\omega,1/(4L))+1}.
 \end{split} 
 \end{equation*}
The second factor in the right-hand side has finite moments of any order, see \eqref{eq:wibull:exp:22} below, replacing therein $\widehat N^{i,n}_{\ell}/\sqrt{\delta_{\ell}}$ by 
$
N([0,T],\omega,1/(4L))/\sqrt{T}.$ 
Moreover, we prove below that $M$ has sub-exponential tails, \textit{i.e.,} ${\mathbb P}(M > a) \leq c \exp(-a^{\varepsilon})$, for $c,\varepsilon >0$. This suffices to prove \eqref{eq:moment:n:lp}.

\smallskip

We now prove that {$(\uptau_{\ell})_{\ell=0,\cdots,M}$ in \eqref{eq:constraint:3:b} may be constructed in such a way that $M$ has} indeed sub-exponential tails. Obviously, see for instance \eqref{eq:N1:N2}, 
it suffices to construct, for each constraint 
in 
\eqref{eq:constraint:3:b}, 
a subdivision {$(\uptau_{\ell})_{\ell=0,\cdots,M}$ 
of $[0,T]$, for which the corresponding constraint in \eqref{eq:constraint:3:b} (and only this one) holds true 
and the number of points $M$ has} sub-exponential tails. 

We start with the second constraint in \eqref{eq:constraint:3:b}. By induction, 
we define the sequence $(\uptau_{\ell}')_{\ell=0,\cdots,M'}$, letting 
$\uptau_{0}':=0$ and $\uptau_{\ell+1}':= \inf \{ t \geq \uptau_{\ell}' :  \Gamma_{2}^{(n)}(\omega,[\uptau_{\ell}',t]) \geq \eta_{0}
\} \wedge T$\footnote{\label{foo:scs}The reader may compare with 
\eqref{eq:stopping:times}, paying attention to the fact that, here, 
$t \mapsto \Gamma_{2}^{(n)}(\omega,[\uptau_{\ell}',t]) $ is not continuous but just right upper semi-continuous,
namely $\lim_{\varepsilon \searrow 0, \varepsilon >0}
\Gamma_{2}^{(n)}(\omega,[\uptau_{\ell}',t+\varepsilon]) \leq 
\Gamma_{2}^{(n)}(\omega,[\uptau_{\ell}',t])$ for any $t \in [\uptau_{\ell}',T]$.
In order to check the latter, it suffices to prove that 
$t \mapsto \widehat N^{i,n}([\uptau_{\ell}',t],\omega,1/(4L))$ is also right upper semi-continuous, for any 
$i \in \{1,\cdots,n\}$. 
Assume indeed that, for an index $i \in \{1,\cdots,n\}$, for a time $t \geq \uptau_{\ell}'$ and for an integer
$\ell \geq 0$, it holds that 
 $\widehat N^{i,n}([\uptau_{\ell}',t+\varepsilon],\omega,1/(4L)) \geq \ell$
 for any $\varepsilon >0$. Then,
 for any $\varepsilon>0$, we can find 
$\ell+1$ reals $\uptau_{\ell}'=:t_{0}^{(\varepsilon)}< 
t_{1}^{(\varepsilon)}<\cdots<t_{\ell}^{(\varepsilon)}\leq t+\varepsilon$
such that  $\widehat w^{i,n}_{p'}(t_{j}^{(\varepsilon)},t_{j+1}^{(\varepsilon)},\omega)^{1/p'}
\geq 1/(4L)$. By an obvious compactness argument and by continuity of 
$\widehat w^{i,n}_{p'}$, we deduce that there exists $\ell+1$ points 
$\uptau_{\ell}'=:t_{0}^{(0)}< 
t_{1}^{(0)}<\cdots<t_{\ell}^{(0)}\leq t$
such that  $\widehat w^{i,n}_{p'}(t_{j}^{(0)},t_{j+1}^{(0)},\omega)^{1/p'}
\geq 1/(4L)$, which in turn implies that 
 $\widehat N^{i,n}([\uptau_{\ell}',t],\omega,1/(4L)) \geq \ell$.  } (we recall that $\eta_{0}>1$), 
with $M' := \inf_{\ell \in {\mathbb N}} \{ \ell \in {\mathbb N} : \uptau_{\ell}' = T\}$.
We claim that we can choose $M =2M'$. Indeed, 
since the counter $\widehat{N}^{i,n}$ appearing in \eqref{eq:constraint:3:b} is the local accumulation of a continuous function on ${\mathcal S}_{2}^T$, there exists $\delta_{L}>0$ such that, for any $t \in [0,T]$ and any 
$i \in \{1,\cdots,n\}$, $\widehat{N}^{i,n}([t,(t+\delta_L)\wedge T],\omega,1/(4L))=0$. (Of course, 
$\delta_{L}$ depends on $n$ and $\omega$, but this is not a problem in the rest of the proof.)
Then, for any point $t \in [\uptau_{\ell}',\uptau_{\ell+1}')$, we have, by definition of 
$\uptau_{\ell+1}'$, $ \Gamma_{2}^{(n)}(\omega,[\uptau_{\ell}',t]) < \eta$. 
Moreover, if $\vert \uptau_{\ell+1}' - t \vert \leq \delta_{L}$, then 
$\Gamma_{2}^{(n)}(\omega,[t,\uptau_{\ell+1}']) =1 \leq \eta$. Therefore, 
we may choose $\uptau_{2\ell}= \uptau_{\ell}'$ for $\ell \in \{0,\cdots,M\}$ and 
then $\vert \uptau_{2\ell+1} - \uptau_{\ell+1}' \vert \leq \delta_{L}$. 
The sequence $(\uptau_{\ell})_{\ell=0,\cdots,2M}$ satisfies the second constraint in \eqref{eq:constraint:3:b}.

We now prove that $M'$ has sub-exponential tails (which implies that $M=2M'$ also has sub-exponential tails). 
Letting $\delta_{\ell}':=\uptau_{\ell+1}'-\uptau_{\ell}'$, for any $\ell \in {\mathbb N}$, we have, 
for any {$A>1$ (recalling $\gamma_{0},\eta_{0}>1$)},
\begin{equation*}
\begin{split}
\pi' &:={\mathbb P} \left( \delta_{\ell}' < \frac1A, \, \ell < M'-1 \right) 
\\
&\leq 
{\mathbb P}
\left(   {}^{(n)} \hspace{-3pt}\Big\lgroup  
 \bigl[ \gamma^2 \,  \bigl( 1+ \widehat w_{p'}^{\bullet,n}(0,T,\omega)^{1/p'} \bigr) \bigr]^{\widehat N^{\bullet,n}_{\ell}(\omega)/\sqrt{\delta_{\ell}'}}
  \hspace{-1pt} \Big\rgroup_{32}^{1/\sqrt{A}}
  \geq \eta_{0}
\right)
\\
&= {\mathbb P}
\left( {}^{(n)} \hspace{-3pt} \Big\lgroup
 \bigl[ \gamma^2 \,  \bigl( 1+ \widehat w_{p'}^{\bullet,n}(0,T,\omega)^{1/p'} \bigr) \bigr]^{\widehat N^{\bullet,n}_{\ell}(\omega)/\sqrt{\delta_{\ell}'}}
  \hspace{-1pt} \Big\rgroup_{32}
  \geq \eta_{0}^{\sqrt{A}}
\right),
\end{split}
\end{equation*}
with 
the shorten notation 
$\widehat N^{\bullet,n}_{\ell}(\omega)=
\widehat N^{\bullet,n}\bigl(\textcolor{black}{[\uptau_{\ell}',\uptau_{\ell+1}']},\omega,1/(4L)\bigr)$; 
 {in the second line, we used the fact that 
${}^{(n)} \hspace{-1pt}\big\lgroup  
 \bigl[ \gamma^2 \,  \bigl( 1+ \widehat w_{p'}^{\bullet,n}(0,T,\omega)^{1/p'} \bigr) \bigr]^{\widehat N^{\bullet,n}_{\ell}(\omega)}
  \hspace{-1pt} \big\rgroup_{32} \geq \eta_{0}$, see footnote \eqref{foo:scs}}. 
We now introduce the function
$f(x) = \exp \bigl( \ln(x)^{1+\epsilon} \bigr)$, $x > 1$; 
it is non-decreasing on $[1,\infty)$ and convex on $[e,\infty)$. By Markov inequality,
\begin{equation*}
\begin{split}
\pi' &\leq 
e^{ - \bigl( \ln [\eta_{0}^{32 \sqrt{A}}] \bigr)^{1+\varepsilon}
}\, {\mathbb E} \biggl[ f \biggl( \frac1n \sum_{i=1}^n e \Bigl[ {\gamma}^2 \Bigl(1+ \widehat w^{i,n}_{p'}(0,T,\cdot)^{1/p'} \Bigr) \Bigr]^{32 \widehat N^{i,n}_{\ell}/\sqrt{\delta_{\ell}'}} \biggr) \biggr]
\\
&\leq e^{ - \bigl( \ln [\eta_{0}^{32 \sqrt{A}}] \bigr)^{1+\varepsilon} } \, \frac1n \sum_{i=1}^n
  {\mathbb E}
\biggl[ f \biggl(
e \Bigl[ {\gamma}^2 \Bigl( 
 1+
 \widehat w^{i,n}_{p'}(0,T,\cdot)^{1/p'} 
  \Bigr) \Bigr]^{32 \widehat N^{i,n}_{\ell}/\sqrt{\delta_{\ell}'}}
  \biggr) 
\biggr],
\end{split}
\end{equation*}
with $e=\exp(1)$. We prove in \eqref{eq:wibull:exp} below that, for $\varepsilon$ small enough 
 {(independently of $n$)},  
\begin{equation}
\label{eq:wibull:exp:00}
\sup_{i=1,\cdots,n}  {\mathbb E}
\biggl[ f \biggl(
e \Bigl[ {\gamma}^2 \Bigl( 
 1+
 \widehat w_{p'}^{i,n}(0,T,\cdot)^{1/p'} 
  \Bigr) \Bigr]^{32 \widehat N^{i,n}_{\ell}/\sqrt{\delta_{\ell}'}}
  \biggr) 
\biggr] \leq C,
\end{equation}
for $C$ independent of $n$. As a result,  
$\pi'
\leq  C \exp \bigl( - \bigl( 32 \ln (\eta_{0}) \bigr)^{1+\varepsilon} A^{(1+\varepsilon)/2} \bigr)$,
and then,
\begin{equation*}
\begin{split}
{\mathbb P} \bigl( M' > \ell +1 \bigr) &= {\mathbb P} \big( \delta_{1}' + \cdots + \delta_{\ell}'  < T, \ell+1 < M' \big)  
\\
&\leq \sum_{i=1}^{\ell}  {\mathbb P} \left( \delta_{i}' < \frac{T}{\ell}, \, i + 1 < M' \right)   
\leq C \ell e^{ - \bigl( 32 \ln (\eta_{0}) \bigr)^{1+\varepsilon} (\ell/T)^{(1+\varepsilon)/2} },
\end{split}
\end{equation*}
which shows that $M'$ has sub-exponential tails. 

\smallskip

We now check what happens when handling the first constraint in \eqref{eq:constraint:3:b}. 
We may define $M'$ as before, 
that is
$M' := \inf_{\ell \in {\mathbb N}} \{ \ell \in {\mathbb N} : \uptau_{\ell}' = T\}$ 
with $\uptau_{0}':=0$ and $\uptau_{\ell+1}':= \inf \{ t \geq \uptau_{\ell}' :  \Gamma_{1}^{(n)}(\omega,[\uptau_{\ell}',t]) \geq 1
\} \wedge T$.
%
%
Then, we can repeat the same proof as above by using the fact that 
\begin{equation*}
\Bigl\{ \delta_{\ell}' < \frac1{A},
\, \ell < M'-1
\Bigr\}
\subset \biggl\{ {}^{(n)} \hspace{-3pt} \Big\lgroup \frac{\widehat N^{\bullet,n}\bigl(\textcolor{black}{[\uptau_{\ell}',\uptau_{\ell+1}']},\omega,1/(4L_{0})\bigr)}{\sqrt{\delta_{\ell}'}} \hspace{-1pt} \Big\rgroup_{8}
\geq    \sqrt{A} \biggr\}
\end{equation*}
and by recalling that
\begin{equation*}
{}^{(n)} \hspace{-3pt} \Big\lgroup \frac{\widehat N^{\bullet,n}\bigl(\textcolor{black}{[\uptau_{\ell}',\uptau_{\ell+1}']},\omega,1/(4L_{0})\bigr)}{\sqrt{\delta_{\ell}'}} \hspace{-1pt} \Big\rgroup_{8}
\end{equation*}
has Weibull tails with shape parameter strictly greater than $1$\footnote{\label{foot:100}Recall that a positive random variable $A$ has a Weibull tail with shape parameter $2/\varrho$ if $A^{1/\rho}$ has a Gaussian tail.}, uniformly in the choice of the dissection $0=\uptau_{0}<\cdots<\uptau_{M'}=T$, which follows from the third item in the assumption of Theorem \ref{theorem:rate:cv} together with the convexity of the function 
$[0,+\infty) \ni x \mapsto \exp (x^{1+\varepsilon})$, for $\varepsilon \geq 0$. This permits to provide an upper bound for $\PP(\delta_{\ell}' < 1/A, \ell < M'-1)$ and then to deduce as before that $M'$ has sub-exponential tails.  

\smallskip

It now remains to prove \eqref{eq:wibull:exp:00}. By \eqref{eq:empirical:exp:v:000} and \eqref{eq:empirical:exp:v:001}, we can find a real $\varepsilon_{1}>0$, independent of $n$, such that 
$
\sup_{i=1,\cdots,n} {\mathbb E}\bigl[ \exp \bigl( \widehat w^{i,n}_{p'}(0,T,\cdot)^{\varepsilon_{1}} \bigr) \bigr] \leq C,
$ 
for $C$ independent of $n$. Hence, combining with the third item in the assumption of the statement,
we get, for any $n \geq 1$, $i \in \{1,\cdots,n\}$, $a>1$ and $K>0$, 
\begin{equation}
\label{eq:wibull:exp}
\begin{split}
&\PP \biggl( \Bigl(     1+ 
\widehat w_{p'}^{i,n} (0,T,\cdot)^{1/p'}
\Bigr)^{\widehat N^{i,n}_{{\ell}}/\sqrt{{\delta_{\ell}'}}} \geq a   \biggr) 
\\
&\leq \PP \biggl( 
\frac{\widehat N^{i,n}_{\ell}}{\sqrt{\delta_{\ell}'}}
\geq K\biggr) + \PP \biggl(        1+ 
\widehat w_{p'}^{i,n} (0,T,\cdot)^{1/p'}
  \geq a^{1/K} \biggr)   
\\
&\leq \textcolor{black}{c} e^{ - K^{1+\varepsilon_2} }+ c e^{ -  a^{\varepsilon_1 p'/K}},
\end{split}
\end{equation}
for a  {new} constant $c$ independent of $n$ and $i$.
Choosing $K= (\ln a)^{1/(1+\textcolor{black}{\varepsilon_{2}}/2)}$,  
we deduce that 
there exist a constant $c>1$ and an exponent $\varepsilon >0$ such that, for any $a>0$, 
\begin{equation}
\label{eq:wibull:exp:22}
{\mathbb P}
   \Bigl(  \bigl(1+ 
\widehat w_{p'}^{i,n} (0,T,\cdot)^{1/p'}
\bigr)^{\widehat N^{i,n}_{\ell}/\sqrt{\delta_{\ell}'}}
\geq  a \Bigr) \leq c e^{- c^{-1} \ln(a)^{1+2\varepsilon} },
\end{equation}
from which we obtain \eqref{eq:wibull:exp:00}. 
\end{proof}

\appendix

\section{Integrability and Auxiliary Estimates}
\label{SectionIntegrability}

We prove in this appendix auxiliary results that we left aside {in the body of the text} to keep focused on the main problems at hand.  {In Appendix \ref{subse:appendix:2}, we show that assumption (c) in Theorem  
\ref{theoremConvergenceRate} holds true for
 interacting particle system driven by Gaussian rough paths satisfying Example 
 \ref{Gaussian}, see Remark \ref{rem:5:1}}.
Appendix \ref{AppendixAuxiliary} is dedicated to proving a crucial moment estimate for some quantity of interest in Step 1 of the proof of Theorem \ref{theorem:rate:cv}. \textit{This is where the convergence rate $\varsigma_n$ appears}, see for instance \eqref{eq:I:Delta}.  {In the last
Appendix \ref{subse:lln}, we elaborate on the versions of law of large numbers used in the text.}

\subsection{Gaussian Case}
\label{subse:appendix:2}   \label{AppendixIntegrability}

 {Remark \ref{rem:5:1}} asserts that the assumptions of Theorem \ref{theorem:prop:of:chaos} are satisfied in the Gaussian framework specified in Example \ref{Gaussian}.  {Since  
the derivation of 
\eqref{eq:exp:integrability} is already justified in the latter example,  
we only prove here that we can control the empirical local accumulation as in the requirement $(c)$ of Theorem 
\ref{theoremConvergenceRate}} with $p'=p$ therein. Following the proof of \cite[Theorem 2.4]{BCD1}, we may focus on the local accumulation of each of the various terms in \eqref{eq:w:widehat:w:v:widehat:v}. To make it clear, we have the following property: For a given threshold $\alpha >0$ and for any two continuous functions $v_{1} : {\mathcal S}_{2}^T \rightarrow \RR_{+}$ and $v_{2}: {\mathcal S}^T_{2} \rightarrow \RR_{+}$, set 
$
N_{i}(\alpha) := N_{v_{i}}\bigl([0,T],\alpha\bigr)$, for $1\leq i\leq 2$, and 
$N(\alpha) := N_{v_{1}+v_{2}}\bigl([0,T],\alpha\bigr)$, see \eqref{eq:N:s:t:omega} for the original definition, then 
\begin{equation}
\label{eq:N1:N2}
\max \Bigl( N_{1}\left(\frac{\alpha}{2}\Bigr),N_{2}\Bigl(\frac{\alpha}{2}\Bigr) \right) \geq
 {N(\alpha)}.
\end{equation}

Throughout the proof, we choose $\Omega$ as the  space ${\mathcal W}= {\mathcal C}([0,T];\RR^d)$, equipped with the 
law ${\mathbb P}$ of the Gaussian process addressed in Example \ref{Gaussian}. We call ${\mathcal H}$ the corresponding Cameron-Martin space and we regard $({\mathcal W},{\mathcal H},{\mathbb P})$ as an abstract Wiener space. We then regard $(W^1,\cdots,W^n)$ as the canonical process on $\Omega^n$ equipped with the product measure ${\mathbb P}^{\otimes n}$. We recall from \cite[Theorem 10.4]{FrizHairer} that the processes $({\mathbb W}^i)_{1 \leq i \leq n}$ and $({\mathbb W}^{i,j})_{1 \leq i,j \leq n}$ may be regarded as random variables on $\Omega^n$. 
We first perform the proof when $[\uptau,\uptau']$ in 
the requirement $(c)$ of Theorem 
\ref{theoremConvergenceRate}
is the interval $[0,T]$ itself; we explain in the last step of the proof why this may be generalized to any 
(possibly random) subinterval $[\uptau,\uptau']$ of $[0,T]$.

\smallskip

\textbf{\textsf{Step 1.}} The first step is to consider, for a given $\alpha >0$, the accumulation $\widetilde N^{i}\big([0,T],\omega,\alpha\big)$ associated with 
$\big\| W^i(\omega)\big\|_{[s,t],p-\textrm{\rm v}}^p + \big\| {\mathbb W}^i(\omega) \big\|_{[s,t],p/2-\textrm{\rm v}}^{p/2}$, see \eqref{eq:v:N}, 
namely 
\begin{equation*}
\widetilde N^{i}([0,T],\omega,\alpha) := N_{\varpi}\big([0,T],\alpha\big), 
\end{equation*}
when
\begin{equation*}
\varpi(s,t) ^p = \big\| W^i(\omega) \big\|_{[s,t],p-\textrm{\rm v}}^p + \big\| {\mathbb W}^i(\omega) \big\|_{[s,t],p/2-\textrm{\rm v}}^{p/2},
\end{equation*} 
{but this follows from  \cite[Theorem 2.4]{BCD1} and from an obvious exchangeability argument.
The term 
$\widehat{v}_{p'}^{i,n}(s,t,\omega)$ in \eqref{eq:w:widehat:w:v:widehat:v} is handled in the same way.} 
%
%
 
\medskip

\textbf{\textsf{Step 2.}} We now focus on the local accumulation of the fourth and fifth terms in \eqref{eq:v:N}. For simplicity, we just explain what happens for the fourth term. The fifth term may be handled in the same way. 

\smallskip

We use the same notation as in Subsection \ref{subse:empirical} and proceed as in the proof of \cite[Theorem 2.4]{BCD1}. The Gaussian process $(W^1,\cdots,W^n)$ has $\big({\mathcal W}^n,{\mathcal H}^{\oplus n},{\mathbb P}^{\otimes n}\big)$ as abstract Wiener space. For $\omega=(\omega_{i})_{i=1}^n \in \Omega^n$ and for ${\boldsymbol h}= \oplus_{i=1}^n h_{i} \in {\mathcal H}^{\oplus n}$, we let 
$$
T_{\boldsymbol h} {\boldsymbol W}^{(n)}(\omega) = T_{\oplus_{i=1}^n h_{i}} {\boldsymbol W}^{(n)}(\omega)
$$ 
for the translated rough path along ${\boldsymbol h}$ (see \cite[(11.5)]{FrizHairer}).  {By \cite[Lemma 11.4]{FrizHairer}} and by Young's inequality, with probability 1 under ${\mathbb P}^{\otimes n}$, for all ${\boldsymbol h} \in {\mathcal H}^{\oplus n}$, 
\begin{equation*}
\begin{split}
\big\| {\mathbb W}^{i,j}(\omega) \big\|_{[s,t],(p/2)-\textrm{\rm v}}^{p/2}
&\leq
c \, 
\Bigl( \big\| (T_{\boldsymbol h}{\mathbb W})^{i,j}(\omega) \big\|_{[s,t],(p/2)-\textrm{\rm v}}^{p/2}
+ \big\| (T_{\boldsymbol h} W)^i(\omega) \big\|_{[s,t],p-\textrm{\rm v}}^p 
\\
&\hspace{15pt} + \big\| (T_{\boldsymbol h} W)^j(\omega) \big\|_{[s,t],p-\textrm{\rm v}}^p   
+
\| h_{i} \|_{[s,t],\varrho-\textrm{\rm v}}^p
+
\| h_{j} \|_{[s,t],\varrho-\textrm{\rm v}}^p
\Bigr).
\end{split}
\end{equation*}
Importantly, the constant $c$ is independent of $n$. Below, it is allowed to increase from line to line as long as it remains independent of $n$. So,
\begin{equation}
\label{eq:Th:n:0}
\begin{split}
&{}^{(n)} \hspace{-1pt}\Big\lgroup 
\big\| {\mathbb W}^{i,\bullet}(\omega) \big\|_{[s,t],(p/2)-\textrm{\rm v}}^{p/2}
 \Big\rgroup_{q}   \\
&\leq
c \, \biggl\{
{}^{(n)} \hspace{-3pt}\Big\lgroup 
\big\| (T_{\boldsymbol h}{\mathbb W})^{i,\bullet}(\omega) \big\|_{[s,t],(p/2)-\textrm{\rm v}}^{p/2}
 \Big\rgroup_{q}
+
{}^{(n)} \hspace{-1pt}\Big\lgroup 
 \big\| (T_{\boldsymbol h} W)^{\bullet}(\omega) \big\|_{[s,t],p-\textrm{\rm v}}^p
 \Big\rgroup_{q}   \\
&\hspace{15pt} + \big\| (T_{\boldsymbol h} W)^i(\omega) \big\|_{[s,t],p-\textrm{\rm v}}^p + \big\| h_{i} \big\|_{[s,t],\varrho-\textrm{\rm v}}^p
+
{}^{(n)} \hspace{-1pt}\Big\lgroup \big\| h_{\bullet} \big\|_{[s,t],\varrho-\textrm{\rm v}}^p \Big\rgroup_{q}
\biggr\}
\\
&\leq
c \, \biggl\{
{}^{(n)} \hspace{-3pt}\Big\lgroup \talloblong\hspace{-2pt} (T_{\boldsymbol h}{\boldsymbol W})^{i,\bullet}(\omega) 
\hspace{-2pt} \talloblong_{[0,T],(1/p)-\textrm{\rm H}}^p \Big\rgroup_{q}
(t-s)
+ \big\| h_{i} \big\|_{[s,t],\varrho-\textrm{\rm v}}^p
\\
&\hspace{15pt} +
{}^{(n)} \hspace{-1pt}\Big\lgroup \big\| h_{\bullet} \big\|_{[s,t],\varrho-\textrm{\rm v}}^p \Big\rgroup_{q}
\biggr\}, 
\end{split}
\end{equation}
where for any $i,j \in \{1,\cdots,n\}$, we let 
\begin{equation*}
\talloblong {\boldsymbol W}^{i,j}(\omega) \talloblong_{[s,t],(1/p)-\textrm{\rm H}}
:= \| (W^i,W ^j)(\omega) \|_{[s,t],(1/p)-\textrm{\rm H}}
+ \sqrt{
\| {\mathbb W}^{i,j}(\omega) \|_{[s,t],(2/p)-\textrm{\rm H}}},
\end{equation*}
and similarly 
for
$\talloblong (T_{\boldsymbol h}{\boldsymbol W})^{i,j}(\omega) \talloblong_{[0,T],(1/p)-\textrm{\rm H}}$.

The tricky term in \eqref{eq:Th:n:0} is the last one on the last line. 
The key point is to notice that, for a given $\epsilon \in (0,2-\rho)$, 
\begin{equation*}
\begin{split}
&{}^{(n)} \hspace{-1pt}\Big\lgroup \big\| h_{\bullet} \big\|_{[s,t],\varrho-\textrm{\rm v}}^p \Big\rgroup_{q}
= \biggl[ \frac1n \sum_{j=1}^n \big\| h_{j} \big\|_{[s,t],\varrho-\textrm{\rm v}}^{p q}
\biggr]^{1/q}
\\
&= n^{-1/q}\biggl\{ \biggl[   \sum_{j=1}^n \big\| h_{j}\big\|_{[s,t],\varrho-\textrm{\rm v}}^{p q}
\biggr]^{(2-\epsilon)/(pq)} \biggr\}^{p/(2-\epsilon)}
\leq n^{-1/q}  
\biggl[ \sum_{j=1}^n \big\| h_{j} \big\|_{[s,t],\varrho-\textrm{\rm v}}^{2-\epsilon}
\biggr]^{p/(2-\epsilon)},
\end{split}
\end{equation*}
where we used the fact that $2-\epsilon < pq$.
Observe in particular that, whenever 
$\sum_{j=1}^n \big\| h_{j} \big\|_{[s,t],\varrho-\textrm{\rm v}}^{2-\epsilon}
 \leq n^{(2-\epsilon)/(pq)}$, it holds 
 \begin{equation*}
\begin{split}
&{}^{(n)} \hspace{-1pt}\Big\lgroup \big\| h_{\bullet} \big\|_{[s,t],\varrho-\textrm{\rm v}}^p \Big\rgroup_{q}
\leq n^{-1/q}  
\biggl[ \sum_{j=1}^n \big\| h_{j} \big\|_{[s,t],\varrho-\textrm{\rm v}}^{2-\epsilon}
\biggr]^{p/(2-\epsilon)}
\\
&\leq n^{-1/q} \bigl( n^{(2-\epsilon)/(pq)} \bigr)^{p/(2-\epsilon)-1}
 \sum_{j=1}^n \big\| h_{j} \big\|_{[s,t],\varrho-\textrm{\rm v}}^{2-\epsilon}
 = n^{-(2-\epsilon)/(pq)}  \sum_{j=1}^n \big\| h_{j} \big\|_{[s,t],\varrho-\textrm{\rm v}}^{2-\epsilon},
\end{split}
 \end{equation*}
 where, in the second line, we used the fact that $p/(2-\epsilon)>1$.
Returning to \eqref{eq:Th:n:0}, we deduce that, whenever 
$\| h_i \|_{[s,t],\varrho-\textrm{\rm v}} \leq 1$
and 
$\sum_{j=1}^n \big\| h_{j} \big\|_{[s,t],\varrho-\textrm{\rm v}}^{2-\epsilon}
 \leq n^{(2-\epsilon)/(pq)}$, 
\begin{align}
{}^{(n)} \hspace{-1pt}\Big\lgroup 
\big\| {\mathbb W}^{i,\bullet}(\omega) \big\|_{[s,t],(p/2)-\textrm{\rm v}}^{p/2}
 \Big\rgroup_{q}   
&\leq
c \, \biggl\{{}^{(n)} \hspace{-3pt}\Big\lgroup \talloblong\hspace{-2pt} (T_{\boldsymbol h}{\boldsymbol W})^{i,\bullet}(\omega) 
\hspace{-2pt} \talloblong_{[0,T],(1/p)-\textrm{\rm H}}^p \Big\rgroup_{q}
(t-s) \nonumber
\\
&\hspace{15pt} + \big\| h_{i} \big\|_{[s,t],\varrho-\textrm{\rm v}}^{2-\epsilon}
+
 n^{-(2-\epsilon)/(pq)}  \sum_{j=1}^n \big\| h_{j} \big\|_{[s,t],\varrho-\textrm{\rm v}}^{2-\epsilon}
\biggr\}. \label{eq:Th:n:0:bbb}
\end{align}
When the left-hand side is less than or equal to $\alpha^p$, we can modify the constant $c$ in such a way that the inequality remains true when $\| h_{i} \|_{[s,t],\varrho-\textrm{\rm v}} \geq 1$ or $\sum_{j=1}^n \big\| h_{j} \big\|_{[s,t],\varrho-\textrm{\rm v}}^{2-\epsilon}
\geq n^{(2-\epsilon)/(pq)}$.  Noticing that $2-\epsilon > \rho$,  \eqref{eq:Th:n:0:bbb} remains true with 
${}^{(n)} \lgroup  {\mathbb W}^{i,\bullet}(\omega) \rgroup_{q;[s,t],(p/2)-\textrm{\rm v}}^{p/2}   $ in the left-hand side. 
 
 Define now 
$N^{i,n,\indep}\big([0,T],\omega,\alpha\big) := N_{\varpi}\big([0,T],\alpha\big)$,
when
\begin{equation*}
\varpi(s,t)^p = {}^{(n)} \hspace{-1pt}\big\lgroup {\mathbb W}^{i,\bullet}(\omega) \big\rgroup_{q;[s,t],(p/2)-\textrm{\rm v}}^{p/2}.
\end{equation*}
Then, 
 \eqref{eq:Th:n:0:bbb} (together with $2-\epsilon > \rho$) yields 
\begin{equation*}
\begin{split}
N^{i,n,\indep}\big([0,T],\omega,\alpha\big) \alpha^p 
&\leq 
c \, \biggl\{{}^{(n)} \hspace{-3pt}\Big\lgroup \talloblong\hspace{-2pt} (T_{\boldsymbol h}{\boldsymbol W})^{i,\bullet}(\omega) 
\hspace{-2pt} \talloblong_{[0,T],(1/p)-\textrm{\rm H}}^p \Big\rgroup_{q}
T
 \\
&\hspace{15pt}+ \big\| h_{i}\big\|_{[0,T],\varrho-\textrm{\rm v}}^{2-\epsilon}
 +
 n^{-(2-\epsilon)/(pq)}  \sum_{j=1}^n \big\| h_{j} \big\|_{[0,T],\varrho-\textrm{\rm v}}^{2-\epsilon}
\biggr\}
\\
&\hspace{0pt}\leq 
c \, \biggl\{{}^{(n)} \hspace{-3pt}\Big\lgroup \hspace{-3pt}\talloblong\hspace{-2pt} (T_{\boldsymbol h}{\boldsymbol W})^{i,\bullet}(\omega) 
\hspace{-2pt} \talloblong_{[0,T],(1/p)-\textrm{\rm H}}^p \hspace{-3pt}\Big\rgroup_{q}
\hspace{-5pt} T
 + \big\| h_{i} \big\|_{[0,T],\varrho-\textrm{\rm v}}^{2-\epsilon}
\\
&\hspace{30pt} +
 n^{-(2-\epsilon)/(pq)+\epsilon/2}  \biggl[ \sum_{j=1}^n \big\| h_{j} \big\|_{[0,T],\varrho-\textrm{\rm v}}^{2} \biggr]^{(2-\epsilon)/2}
\biggr\},
\end{split}
 \end{equation*}
where we applied H\"older's inequality to handle the last term. 
By choosing $\varepsilon$ small enough such that $(2-\varepsilon)/(pq) - \varepsilon/2>0$ and by applying Proposition 11.2 in \cite{FrizHairer}, we get, for a possibly new value of the constant $c$,
\begin{equation}
\label{eq:Th:n}
\begin{split}
&N^{i,n,\indep}\big([0,T],\omega,\alpha\big) \alpha^p 
\\
&\hspace{15pt}\leq
c \, \biggl\{{}^{(n)} \hspace{-3pt}\Big\lgroup \talloblong\hspace{-2pt} (T_{\boldsymbol h}{\boldsymbol W})^{i,\bullet}(\omega) 
\hspace{-2pt} \talloblong_{[0,T],(1/p)-\textrm{\rm H}}^p \Big\rgroup_{q}
T
+  
\| {\boldsymbol h} \|_{{\mathcal H}^{\oplus n}}^{2-\epsilon}  
T^{(2-\epsilon)/(2\rho)}
\biggr\},
\end{split}
\end{equation}
with 
$\| {\boldsymbol h} \|_{{\mathcal H}^{\oplus n}}^2 = \sum_{i=1}^n \| h_{i} \|_{{\mathcal H}}^2$. We then notice that $(2-\varepsilon)/(2 \varrho) > 1/2$ since $2-\varepsilon > \varrho$. We deduce that $T^{(2-\varepsilon)/(2 \varrho)} \leq c T^{1/2}$ for a possibly new value of the constant $c$. We then apply Theorems 11.5 and 11.7 in \cite{FrizHairer} but on the space $({\mathcal W}^{\otimes n},{\mathcal H}^{\oplus n},{\mathbb P}^{\otimes n})$. Importantly, we observe that 
\begin{equation*}
{\mathbb E} \Bigl[ {}^{(n)} \hspace{-3pt}\Big\lgroup \talloblong \hspace{-2pt}  {{\boldsymbol W}^{i,\bullet}}(\omega) \hspace{-2pt} \talloblong_{[0,T],(1/p)-\textrm{\rm H}}^p  \Big\rgroup_{q} \Bigr]
\end{equation*}
is bounded by a constant $c$, independent of $i$ and $n$, which proves that  $N^{i,n,\indep}([0,T],\cdot,\alpha)/\sqrt{T}$ has a Weibull distribution with shape parameter $2/(2-\epsilon)$, independently of $n$.  

\medskip

\textbf{\textsf{Step 3.}} 
We now turn to the local accumulation of the sixth term in \eqref{eq:v:N}. 
Taking the norm   
${}^{(n)} \hspace{-1pt} \lgroup \, \cdot \, \rgroup_{q}$ in 
\eqref{eq:Th:n:0}, we get, 
with probability 1 under ${\mathbb P}^{\otimes n}$, for all ${\boldsymbol h} \in {\mathcal H}^{\oplus n}$, 
\begin{equation*}
\begin{split}
&{}^{(n)} \hspace{-3pt}
\Big\lgroup \hspace{-6pt}
\Big\lgroup 
\big\| {\mathbb W}^{\bullet,\bullet}(\omega) \big\|_{[s,t],(p/2)-\textrm{\rm v}}^{p/2}
\Big\rgroup
\hspace{-6pt}
\Big\rgroup_{q} 
\\
&\hspace{15pt} \leq
c \, \biggl\{
{}^{(n)} 
 \hspace{-3pt}
\Big\lgroup \hspace{-6pt}
\Big\lgroup
\talloblong\hspace{-2pt} (T_{\boldsymbol h}{\boldsymbol W})^{\bullet,\bullet}(\omega) 
\hspace{-2pt} \talloblong_{[0,T],(1/p)-\textrm{\rm H}}^p 
\Big\rgroup
\hspace{-6pt}
\Big\rgroup_{q}
(t-s)
+
{}^{(n)} \hspace{-1pt}\Big\lgroup \big\| h_{\bullet} \big\|_{[s,t],\varrho-\textrm{\rm v}}^p \Big\rgroup_{q}
\biggr\}.
\end{split}
\end{equation*}
Following the proof of \eqref{eq:Th:n:0:bbb}, we deduce that
\begin{align*}
{}^{(n)} \hspace{-1pt}\big\lgroup \hspace{-2pt}\big\lgroup
  \WW^{\bullet,\bullet}(\omega) \big\rgroup \hspace{-2pt} \big\rgroup_{q ; [s,t],p/2-\textrm{\rm v}}^{p/2}
 &\leq
c \, \biggl\{
{}^{(n)} 
 \hspace{-3pt}
\Big\lgroup \hspace{-6pt}
\Big\lgroup
\talloblong\hspace{-2pt} (T_{\boldsymbol h}{\boldsymbol W})^{\bullet,\bullet}(\omega) 
\hspace{-2pt} \talloblong_{[0,T],(1/p)-\textrm{\rm H}}^p 
\Big\rgroup
\hspace{-6pt}
\Big\rgroup_{q}
(t-s)
\\
&\hspace{30pt}+
 n^{-(2-\epsilon)/(pq)}  \sum_{j=1}^n \big\| h_{j} \big\|_{[s,t],\varrho-\textrm{\rm v}}^{2-\epsilon}
\biggr\},
\nonumber
\end{align*}
at least when the left-hand side is less than or equal to $\alpha^p$. Importantly, there is no need to distinguish the coordinate $i$ of ${\boldsymbol h}$ from the other coordinates $j \not = i$ since the coefficient in front of any $\| h_{j} \|_{[s,t],\varrho-\textrm{\rm v}}$, $j=1,\cdots,n$, 
has the same power decay  
as $n$ tends to $\infty$.  {So, the context is simpler than in the previous step
and we may conclude in the same way.}

 {Local accumulations associated to 
 the second term in 
\eqref{eq:v:N} and to 
   $(s,t) \mapsto {}^{(n)} \hspace{-1pt}\big\lgroup  
 v^{\bullet,n}_{p'}(\omega)  \big\rgroup_{q ; [s,t],1-\textrm{\rm v}}$
and  $(s,t) \mapsto {}^{(n)} \hspace{-1pt}\big\lgroup  
\widehat v^{\bullet,n}_{p'}(\omega)  \big\rgroup_{q ; [s,t],1-\textrm{\rm v}}$
in 
\eqref{eq:w:widehat:w:v:widehat:v}
are handled in the same way.} (As for the latter one, the reader may refer to 
the proof of \cite[Theorem 2.4]{BCD1}.)
\smallskip

\textbf{\textsf{Step 4.}} 
 {The proof has been here achieved on the interval $[0,T]$. Importantly, the fact that $T$ is deterministic does not play any role in the proof. It is
in particular quite 
easy to see that  
the interval $[0,T]$ can be replaced by any (random) sub-interval $[\uptau,\uptau'] \subset [0,T]$ as in
the requirement $(c)$ of Theorem 
\ref{theoremConvergenceRate}}. \qed
%

\bigskip

\subsection{An Auxiliary Estimate}
\label{AppendixAuxiliary}

We prove in this appendix some auxiliary estimates that were used in Step 1 of the proof of Theorem \ref{theorem:rate:cv}. This is where the convergence rate $\varsigma_n$ 
in Theorem \ref{theoremConvergenceRate} appears.  {Recall we set $\varsigma_n =  {n^{-1/2}}$ if $d=1$, and $\varsigma_n = n^{-1/2}\ln(1+n)$, if $d=2$, and $\varsigma_n = n^{-1/d}$, if $d\geq 3$. Recall also} definitions \eqref{eq:barFti:Ftin}, \eqref{eq:delta:barFti}, \eqref{eq:delta:Ftin} and \eqref{eq:I:i,n,Delta}.

\medskip

\begin{lem}
\label{le:sewing}
Fix $\varrho \geq 8$. There exists an exponent $\varrho'$ 
such that, whenever $X_{0}(\cdot) \in \LL^{\varrho'}$, 
we can find another constant $C$, depending on $\langle X_{0}(\cdot)\rangle_{\varrho'}$ and satisfying, for any integers $1\leq i\leq n$ and  any $0 \leq r \leq s \leq t \leq T$, 
\begin{equation*}
\begin{split}
&\bigl\langle 
\bigl[ F^{i,n}(\cdot)
-\overline{F}^i(\cdot)  
\bigr]_{s,t}
\bigr\rangle_{\varrho}
\leq C \varsigma_{n} \big\llangle w^+(s,t,\cdot,\cdot)\big\rrangle_{\varrho'}^{1/p},
\\
&\Bigl\langle
{\mathcal I}^{i,n,\partial}_{\{s,t\}}(\cdot)
- 
\overline {\mathcal I}^{i,\partial}_{\{s,t\}}(\cdot)
 \Bigr\rangle_{\varrho}
 \leq C \, \varsigma_{n} \,
 \Bigl( 
 \big\llangle w^+(s,t,\cdot,\cdot)\big\rrangle_{\varrho'}^{1/p}
 +
 \big\llangle w^+(s,t,\cdot,\cdot)\big\rrangle_{\varrho'}^{2/p} \Bigr),
 \\
&\biggl( \int_{\Omega} \Bigl\vert  
 \frac1n \sum_{j=1}^n
\delta_{\mu} F_{s}^{i,j,n}(\omega) {\mathbb W}_{s,t}^{j,i}(\omega)
-
\EE \bigl[
\delta_{\mu} \overline F_{s}^i(\omega,\cdot) {\mathbb W}_{s,t}^{i,\indep}(\cdot,\omega)
\bigr]
\Bigr\vert^{\rho} d {\mathbb P}(\omega) \biggr)^{1/\rho}
\\
&\hspace{15pt} +
\Bigl\langle
\bigl( \delta_{x} F_{s}^{i,n}(\cdot)
-
\delta_{x} \overline F_{s}^{i}(\cdot)
\bigr) {\mathbb W}^i_{s,t}(\cdot) 
\Bigr\rangle_{\rho} \leq  C \, \varsigma_{n} \,\big\llangle w^+(s,t,\cdot,\cdot)\big\rrangle_{\varrho'}^{2/p},
\\
&\Bigl\langle
\Bigl\{
{\mathcal I}^{i,n,\partial}_{\{r,s\}}(\cdot)
+
{\mathcal I}^{i,n,\partial}_{\{s,t\}}(\cdot)
-
{\mathcal I}^{i,n,\partial}_{\{r,t\}}(\cdot)
\Bigr\} 
- 
\Bigl\{
\overline {\mathcal I}^{i,\partial}_{\{r,s\}}(\cdot)
+
\overline {\mathcal I}^{i,\partial}_{\{s,t\}}(\cdot)
-
\overline {\mathcal I}^{i,\partial}_{\{r,t\}}(\cdot)
\Bigr\}
 \Bigr\rangle_{\varrho}
 \\
&\hspace{15pt} 
 \leq C \, \varsigma_{n} \,\big\llangle w^+(r,t,\cdot,\cdot)\big\rrangle_{\varrho'}^{3/p},
\end{split}
\end{equation*}
where 
$
w^+(r,t,\omega,\omega') := w(r,t,\omega) + \| {\mathbb W}^{\indep}(\omega,\omega') \|_{[r,t],p/2-\textrm{\rm v}}^{p/2}$,  {with 
$w$ as in 
\eqref{eq:w:s:t:omega} for the same parameters $p$ and $q$ as therein}. 
\end{lem}

 {The reason the appearance of the quantity $w^+$ instead of $w$, in the above upper bounds, will appear at the beginning of \textbf{\textsf{Step 2}} in the proof.}

\begin{proof}
We directly prove the last inequality in the statement; the first three inequalities follow from similar computations. Throughout the proof, we use the following notations. For each $i \in \{1,\cdots,n\}$, we call $\overline w^i$ the control associated with $\overline {\boldsymbol W}^i(\cdot)$ through identity \eqref{eq:w:s:t:omega}. For $j \in \{1,\cdots,n\}$, we also let 
\begin{equation*}
\overline w^{i,j}(s,t,\omega) :=
\bigl\|  {{\mathbb W}^{i,j}}(\omega) \bigr\|_{[s,t],p-\textrm{\rm v}}^p.
\end{equation*}
We make {in the course of the proof} an intense use of Lemma \ref{lem:W1:cv:empirical} below, giving the convergence rate of the empirical measure of a sample of independent and identically distributed random variables towards their common law. 
{In this regard, a key fact is that the theoretical distribution driving the empirical one must be sufficiently integrable. By
a variant of 
\eqref{eq:moment:n:lp}, we already know that, for any $\rho \geq 1$,
there exists $\rho' \geq 8$ such that  
$\sup_{0 \leq t \leq T} \big\vert X_{t}(\cdot) \big\vert$ is in $\LL^{\rho}$ as soon as $X_{0}(\cdot)$ is in $\LL^{\rho'}$.
The proof of this variant is in fact simpler than the proof of \eqref{eq:moment:n:lp} itself, since 
we can directly invoke Theorem 
\ref{main:theorem:existence:small:time}
instead 
of 
\eqref{eq:vvvertXivvvert}, noticing that the analogue of $M$ 
in 
\eqref{eq:vvvertXivvvert} then becomes deterministic, see for instance footnote \eqref{foo:concatenation}. 
Importantly, the same holds true with $\big\vvvert X(\cdot) \big\vvvert_{[0,T],w,p}$:
it belongs to $\LL^{\rho}$ if $X_{0}(\cdot)$ is in $\LL^{\rho'}$, for a well-chosen $\rho'$.
The proof 
also follows from Theorem 
\ref{main:theorem:existence:small:time}, by concatenating a  deterministic finite number of intervals of the form 
$[S_{1},S_{2}]$, see \cite[footnote (5)]{BCD1} for some details about concatenation.} We then compute
\begin{equation*}
\begin{split}
&\Bigl\{
{\mathcal I}^{i,n,\partial}_{\{r,s\}}(\omega)
+
{\mathcal I}^{i,n,\partial}_{\{s,t\}}(\omega)
-
{\mathcal I}^{i,n,\partial}_{\{r,t\}}(\omega)
\Bigr\} 
- 
\Bigl\{
\overline {\mathcal I}^{i,\partial}_{\{r,s\}}(\omega)
+
\overline {\mathcal I}^{i,\partial}_{\{s,t\}}(\omega)
-
\overline {\mathcal I}^{i,\partial}_{\{r,t\}}(\omega)
\Bigr\}
\\
&= \Bigl( R_{r,s}^{F^{i,n}}(\omega) - R_{r,s}^{\overline F^i}(\omega) \Bigr) W_{s,t}^i(\omega)
+ 
\Bigl( \delta_{x} F_{r,s}^{i,n}(\omega) - \delta_{x} \overline F_{r,s}^i(\omega) \Bigr) {\mathbb W}^i_{s,t}(\omega)
 \\
&\hspace{15pt}+ 
\left( \frac1n \sum_{j=1}^n \delta_{\mu} F_{r,s}^{i,j,n}(\omega) {\mathbb W}^{j,i}_{s,t}(\omega) - {\mathbb E} \Bigl[ \delta_{\mu} \overline F_{r,s}^{i}(\omega,\cdot) {\mathbb W}_{s,t}^{i,\indep}(\cdot,\omega)\Bigr] \right),
\end{split}
\end{equation*}
where
\begin{equation}
\label{eq:R:Fin:RbarFi}
\begin{split}
&R_{r,s}^{F^{i,n}}(\omega) := F_{s}^{i,n}(\omega) - F_{r}^{i,n}(\omega) - \delta_{x} F_{r}^{i,n}(\omega) W_{r,s}^i(\omega) 
\\
&\hspace{30pt} - 
 \frac1n \sum_{j=1}^n \delta_{\mu} F_{r}^{i,j,n}(\omega)  W_{r,s}^{j}(\omega),   \\
&R_{r,s}^{\overline F^{i}}(\omega) := \overline F_{s}^{i}(\omega) - \overline F_{r}^{i}(\omega) - \delta_{x} \overline F_{r}^{i}(\omega) W_{r,s}^i(\omega) 
\\
&\hspace{30pt} 
-  \EE \Bigl[ \delta_{\mu} \overline F_{r}^{i}(\omega,\cdot)  W_{r,s}^i(\cdot) \Bigr].
\end{split}
\end{equation}
Following \eqref{eq:delta:barFti} and \eqref{eq:delta:Ftin}, we define differentiable functions $G_x$ and $G_\mu$ of their arguments setting 
\begin{equation*}
\begin{split}
&\delta_{x} F_{t}^{i,n}(\omega) =: G_{x}\bigl(\overline X_{t}^i(\omega),\overline \mu^n_{t}(\omega) \bigr),   \quad
\delta_{x} \overline F_{t}^{i}(\omega) =: G_{x}\bigl(\overline X_{t}^i(\omega),{\mathcal L}(X_{t}) \bigr),
\\
&\delta_{\mu} F_{t}^{i,j,n}(\omega) =: G_{\mu}\bigl(\overline X_{t}^i(\omega),\overline \mu^n_{t}(\omega) \bigr) \bigl( \overline X_{t}^j(\omega)  \bigr),   \\
&\delta_{\mu} \overline F_{t}^{i}(\omega,\cdot) =: G_{\mu}\bigl(\overline X_{t}^i(\omega),{\mathcal L}(X_{t}) \bigr) \bigl( \overline X^i_{t}(\cdot) \bigr).
\end{split}
\end{equation*}
Finally, we can write the whole difference in the form
\begin{equation}
 \label{eq:summand:analysis}
\begin{split}
&\Bigl\{
{\mathcal I}^{\partial}_{\{r,s\}}(\omega)
+
{\mathcal I}^{\partial}_{\{s,t\}}(\omega)
-
{\mathcal I}^{\partial}_{\{r,t\}}(\omega)
\Bigr\}
- 
\Bigl\{
\overline {\mathcal I}^{\partial}_{\{r,s\}}(\omega)
+
\overline {\mathcal I}^{\partial}_{\{s,t\}}(\omega)
-
\overline {\mathcal I}^{\partial}_{\{r,t\}}(\omega)
\Bigr\}  
\\
&= \bigl( R_{r,s}^{F^{i,n}}(\omega)
-R_{r,s}^{\overline F^i}(\omega)
\bigr) 
 W_{s,t}^i(\omega) 
 \\
&
+ 
\Bigl[ 
G_{x} \bigl(\overline X^i(\omega),\overline \mu^n(\omega) \bigr) 
-
G_{x} \bigl(\overline X^i(\omega),{\mathcal L}(X) \bigr) 
\Bigr]_{r,s}
 {\mathbb W}^i_{s,t}(\omega) 
  \\
&+
\frac1n \sum_{j=1}^n 
\Bigl[ 
G_{\mu} \bigl(\overline X^i(\omega),\overline \mu^n(\omega) \bigr)\bigl(\overline X^j(\omega)\bigr) 
-
G_{\mu} \bigl(\overline X^i(\omega),{\mathcal L}(X) \bigr) \bigl(\overline X^j(\omega)\bigr)
\Bigr]_{r,s}
{\mathbb W}^{j,i}_{s,t}(\omega)  
 \\
&+ \frac1n \sum_{j=1}^n  \Bigl[G_{\mu} \bigl(\overline X^i(\omega),{\mathcal L}(X) \bigr) \bigl(\overline X^j(\omega)\bigr)
\Bigr]_{r,s} {\mathbb W}^{j,i}_{s,t}(\omega) - {\mathbb E} \Bigl[ \delta_{\mu} \overline F_{r,s}^{i}(\omega,\cdot) {\mathbb W}_{s,t}^{i,\indep}(\cdot,\omega) \Bigr].  
\end{split}
\end{equation}
A key fact is that $G_{x}$ and $G_{\mu}$ are Lipschitz continuous in all the entries, the Lipschitz property 
in $\mu$ being understood with respect to ${\mathbf d}_{1}$. Moreover, similar to F itself, they are jointly continuously differentiable in 
all the arguments and the derivatives are Lipschitz continuous, the Lipschitz property in 
$\mu$ being again understood with respect to ${\mathbf d}_{1}$. 
\smallskip

\textbf{\textsf{Step 1.}} Observe that
\begin{equation}
\label{eq:Gx:xi:mun}
\begin{split}
&\Bigl[ 
G_{x} \bigl(\overline X^i(\omega),\overline \mu^n(\omega) \bigr) 
\Bigr]_{r,s}
\\
&= \int_{0}^1 \partial_{x} G_{x} \Bigl( \overline X_{r;(r,s)}^{i,(\lambda)}(\omega),
\overline \mu_{r;(r,s)}^{n,\lambda}(\omega)
\Bigr) \overline X_{r,s}^{i}(\omega) 
 d\lambda
 \\
 &\hspace{15pt}
 + 
 \frac1n \sum_{j=1}^n
 \int_{0}^1 D_{\mu} G_{x}
 \Bigl( \overline X_{r;(r,s)}^{i,(\lambda)}(\omega),
 \overline \mu_{r;(r,s)}^{n,\lambda}(\omega)\Bigr) 
\bigl( 
\overline X_{r;(r,s)}^{j,(\lambda)}(\omega)
\bigr)
\overline X_{r,s}^{j}(\omega) d \lambda
\\
&= \int_{0}^1 \partial_{x} G_{x} \Bigl(\overline X_{r;(r,s)}^{i,(\lambda)}(\omega),
\overline \mu_{r;(r,s)}^{n,\lambda}(\omega)
\Bigr) \overline X_{r,s}^{i}(\omega) 
 d\lambda
 \\
 &\hspace{15pt}
 + 
 \int_{\RR^{2d}} 
 \biggl[ \int_{0}^1
 D_{\mu} G_{x}
 \Bigl( \overline X_{r;(r,s)}^{i,(\lambda)}(\omega),
 \overline \mu_{r;(r,s)}^{n,\lambda}(\omega)\Bigr) 
(y)
 z d \lambda
 \biggr]
 d \overline \nu_{r;(r,s)}^{n,\lambda}(\omega;y,z)
\end{split}
\end{equation}
where
\begin{equation*}
 \overline \mu_{r;(r,s)}^{n,(\lambda)}(\omega) := \frac1n \sum_{j=1}^n \delta_{\overline X_{r;(r,s)}^{j,(\lambda)}(\omega)}, 
 \quad 
 \overline \nu_{s;(s,t)}^{n,(\lambda)}(\omega) := \frac1n \sum_{j=1}^n \delta_{\bigl(\overline X_{r;(r,s)}^{j,(\lambda)}(\omega),\overline X_{r,s}^{j}(\omega)\bigr)}, 
\end{equation*}
with
\begin{equation*}
\overline X_{r;(r,s)}^{j,(\lambda)}(\omega) := \overline X_{r}^j(\omega) + \lambda \overline X_{r,s}^j(\omega).
\end{equation*}
Proceeding similarly with $\big[ G_{x} \big(\overline X^i(\omega) , {\mathcal L}(X) \big) \big]_{r,s}$, we get
\begin{equation*}
\begin{split}
&\Bigl[ 
G_{x} \bigl( \overline X^i(\omega),\overline \mu^n(\omega) \bigr) 
-
G_{x} \bigl( \overline X^i(\omega) , {\mathcal L}(X) \bigr)
\Bigr]_{r,s}
\\
&=
\int_{0}^1 
\Bigl[
\partial_{x} G_{x} \Bigl( \overline X_{r;(r,s)}^{i,(\lambda)}(\omega),
\overline \mu_{r;(r,s)}^{n,(\lambda)}(\omega)
\Bigr)
\\
&\hspace{100pt}-\partial_{x} G_{x} \Bigl( \overline X_{r;(r,s)}^{i,(\lambda)}(\omega),
{\mathcal L}
\bigl(X_{r;(r,s)}^{(\lambda)}\bigr)
\Bigr)
\Bigr]
\, \overline  X_{r,s}^{i}(\omega) \,  d\lambda
 \\
 &\hspace{15pt}
 +
  \int_{\RR^{2d}} 
 \biggl[ \int_{0}^1
 D_{\mu} G_{x}
 \Bigl( \overline X_{r;(r,s)}^{i,(\lambda)}(\omega),
 \overline \mu_{r;(r,s)}^{n,(\lambda)}(\omega)\Bigr) 
(y)
 z d \lambda
 \biggr] \,
 d \overline \nu_{r;(r,s)}^{n,(\lambda)}(\omega;y,z)
 \\
 &\hspace{15pt}
 -
\int_{\RR^{2d}}
 \biggl[ \int_{0}^1
D_{\mu} G_{x}
 \Bigl( \overline X_{r;(r,s)}^{i,(\lambda)}(\omega),
{\mathcal L}\bigl(  X_{r;(r,s)}^{(\lambda)}
\bigr)
\Bigr) 
(y)
 z d \lambda
 \biggr] \, d {\mathcal L}\bigl( X_{r;(r,s)}^{(\lambda)},X_{r,s}\bigr) (y,z),
\end{split}
\end{equation*}
where, as before,
$X_{r;(r,s)}^{(\lambda)}(\omega)
=  X_{r}(\omega) + \lambda X_{r,s}(\omega)$. 
Splitting the last two terms in the above expansion into
\begin{equation*}
\begin{split}
 & \int_{\RR^{2d}} 
 \biggl[ \int_{0}^1
 D_{\mu} G_{x}
 \Bigl( \overline X_{r;(r,s)}^{i,(\lambda)}(\omega),
 \overline \mu_{r;(r,s)}^{n,(\lambda)}(\omega)\Bigr) 
(y)
 z d \lambda
 \biggr] \,
 d \overline \nu_{r;(r,s)}^{n,(\lambda)}(\omega;y,z)
 \\
 &\hspace{10pt}
 -
\int_{\RR^{2d}}
 \biggl[ \int_{0}^1
 D_{\mu} G_{x}
 \Bigl( \overline X_{r;(r,s)}^{i,(\lambda)}(\omega),
{\mathcal L}\bigl(X_{r;(r,s)}^{(\lambda)}
\bigr)
\Bigr) 
(y)
 z d \lambda
 \biggr] \, d  \overline \nu_{r;(r,s)}^{n,(\lambda)}(\omega;y,z)
 \\
 &\hspace{10pt}+ 
\int_{\RR^{2d}}
 \biggl[ \int_{0}^1
 D_{\mu} G_{x}
 \Bigl( \overline X_{r;(r,s)}^{i,(\lambda)}(\omega),
{\mathcal L}\bigl(X_{r;(r,s)}^{(\lambda)}
\bigr)
\Bigr) 
(y)
 z d \lambda
 \biggr] \, d\overline \nu_{r;(r,s)}^{n,(\lambda)}(\omega;y,z)
 \\
&\hspace{10pt}
  -   \int_{\RR^{2d}} \biggl[ \int_{0}^1  D_{\mu} G_{x}  \Bigl( \overline X_{r;(r,s)}^{i,(\lambda)}(\omega), {\mathcal L}\bigl(X_{r;(r,s)}^{(\lambda)} \bigr) \Bigr)  (y) z d \lambda \biggr] \, d {\mathcal L}\bigl(X_{r;(r,s)}^{(\lambda)},X_{r,s}\bigr) (y,z),
\end{split}
\end{equation*}
we get
\begin{equation*}
\begin{split}
&\Bigl\vert \Bigl[ 
G_{x} \bigl(\overline X^i(\omega),\overline \mu^n(\omega) \bigr) 
-
G_{x} \bigl( \overline X^i(\omega) , {\mathcal L}(X) \bigr)
\Bigr]_{r,s} \Bigr\vert
\\
&\leq 
c \int_{0}^1  {{\mathbf d}_{1}}
\Bigl( \overline \mu^{n,(\lambda)}_{r;(r,s)}(\omega),{\mathcal L}
\bigl( X_{r;(r,s)}^{(\lambda)}\bigr)
\Bigr) d\lambda
\\
&\hspace{5pt} \times \biggl( \vvvert \overline X^i(\omega) \vvvert_{[0,T],\overline w^i,p} \overline w^{i}(r,s,\omega)^{1/p} + \frac1n \sum_{k=1}^n \vvvert \overline X^k (\omega) \vvvert_{[0,T],\overline w^k,p} \overline w^k(r,s,\omega)^{1/p} \biggr)   \\
&\hspace{1pt} + c \, \Bigl\vert {\mathcal S}^{i,n}_{r,s}\bigl(\omega, \vert \overline X_{r,s}^\bullet(\omega) \vert\bigr)
\Bigr\vert,
\end{split}
\end{equation*}
where ${\mathcal S}^{i,n}_{r,s}\bigl(\omega, \vert \overline X_{r,s}^\bullet(\omega) \vert\bigr)$ is the $n$-empirical mean of 
$n$  variables that 
are dominated by $\big(\vert \overline X_{r,s}^j(\omega) \vert + \bigl\langle 
X_{r,s}(\cdot)
\bigr\rangle_{1}
\big)_{j=1,\cdots,n}$ 
and $n-1$ of which are conditionally centred and conditionally independent given
the realization of the path $(\overline{X}^i,W^i,{\mathbb W}^i)$.  Allowing the value of the constant $c$ to increase from line to line, we obtain
\begin{equation*}
\begin{split}
&\Bigl\vert \Bigl[ 
G_{x} \Bigl(\overline X^i(\omega),\overline \mu^n(\omega) \Bigr) 
-
G_{x} \Bigl( \overline X^i(\omega) , {\mathcal L}(X) \Bigr)
\Bigr]_{r,s} {\mathbb W}_{s,t}^i(\omega) \Bigr\vert   \\
&\leq c \int_{0}^1  {{\mathbf d}_{1}} \Bigl( 
\overline \mu_{r;(r,s)}^{n,(\lambda)}(\omega), 
{\mathcal L}\bigl(X_{r;(r,s)}^{(\lambda)}\bigr)
\Bigr) d\lambda
\\
&\hspace{30pt} \times \biggl[ \vvvert \overline X^i(\omega) \vvvert_{[0,T],\overline w^i,p} + \biggl( \frac1n \sum_{k=1}^n \vvvert \overline X^k(\omega) \vvvert_{[0,T],\overline w^k,p}^2 \biggr)^{1/2}\biggr]   \\
&\hspace{30pt} \times \biggl[ \overline w^i(r,t,\omega)^{3/p} +  \frac1n \sum_{k=1}^n \overline w^k(r,t,\omega)^{3/p} \biggr] 
\\
&\hspace{15pt} + c \, 
\Bigl\vert {\mathcal S}^{i,n}_{r,s}\bigl(\omega,\vert \overline X_{r,s}^\bullet(\omega) \vert\bigr) 
\Bigr\vert
\, \overline w^i(r,t,\omega)^{2/p}.
\end{split}
\end{equation*}
In order to conclude for the second term in the right-hand side of \eqref{eq:summand:analysis}, it suffices to recall from Rosenthal's inequality (applied under the conditional probability given 
the realization of the path $(\overline{X}^i,W^i,{\mathbb W}^i)$)
 that
\begin{equation*}
\begin{split}
\Bigl\langle {\mathcal S}^{i,n}_{r,s}\bigl(\cdot,\vert \overline X_{r,s}^\bullet(\cdot) \vert\bigr) \Bigr\rangle_{3 \varrho/2}
&\leq c\, n^{-1/2} \,\Bigl\langle \vvvert  X(\cdot) \vvvert_{[0,T],w,p} w(r,s,\cdot)^{1/p} \Big\rangle_{3\varrho/2}   \\
&\leq c \, n^{-1/2} \, \big\langle \vvvert  X(\cdot) \vvvert_{[0,T],w,p} \bigr\rangle_{{\chi}\varrho} \bigl\langle w(r,t,\cdot) \bigr\rangle_{{\chi}\varrho}^{1/p},
\end{split}
\end{equation*}
{where $\chi \geq 1$ is a universal constant whose value may change from line to line (as long as it remains universal)}.  
If $\rho$ is large enough, we deduce from Lemma 
\ref{lem:W1:cv:empirical} that
\begin{equation*}
\begin{split}
&\Bigl\langle  
\Bigl[
G_{x} \bigl(\overline X^i(\cdot),\overline \mu^n(\cdot) \bigr) 
-
G_{x} \bigl( \overline X^i(\cdot) , {\mathcal L}(X) \bigr)
\Bigr]_{r,s} {\mathbb W}_{s,t}^i(\cdot) 
\Bigr\rangle_{\varrho}
\\
&\leq 
c \, \biggl( \int_{0}^1 \Bigl\langle 
{{\mathbf d}_{1}} \Bigl( 
\overline \mu_{r;(r,s)}^{n,(\lambda)}(\cdot), 
{\mathcal L}\bigl(X_{r;(r,s)}^{(\lambda)}\bigr)
\Bigr)
\Bigr\rangle_{{\chi}\varrho}
d \lambda
\biggr) \,
\bigl\langle 
\vvvert X(\cdot) \vvvert_{[0,T],w,p}
\bigr\rangle_{{\chi} \varrho} \, 
\\
&\hspace{200pt} \times \bigl\langle 
w(r,t,\cdot)
\bigr\rangle_{{\chi} \varrho}^{3/p}
\\
&\hspace{10pt}
+ c \,n^{-1/2} \, \big\langle \vvvert  X(\cdot) \vvvert_{[0,T],w,p} \bigr\rangle_{{\chi}\varrho} \, \bigl\langle w(r,t,\cdot) \bigr\rangle_{{\chi}\varrho}^{3/p}
\\
&\leq 
c \, \varsigma_{n} \,
 \Bigl(1
 +
\bigl\langle
\sup_{0\leq u \leq T}
\vert X_{u}(\cdot) \vert
\bigr\rangle_{{\chi} \varrho}
\Bigr)
 \bigl\langle 
\vvvert X(\cdot) \vvvert_{[0,T],w,p}
\bigr\rangle_{{\chi} \varrho}
\,
\bigl\llangle 
w^+(r,t,\cdot,\cdot)
\bigr\rrangle_{{\chi} \varrho}^{3/p}.
\end{split}
\end{equation*}

\smallskip

\textbf{\textsf{Step 2.}} By the same argument, we have
\begin{equation*}
\begin{split}
&\Bigl\vert \Bigl[ 
G_{\mu} \Bigl(\overline X^i(\omega),\overline \mu^n(\omega) \Bigr)\bigl(\overline X^j(\omega)\bigr)  -  G_{\mu} \Bigl( \overline X^i(\omega) , {\mathcal L}(X) \Bigr) \bigl(\overline X^j(\omega)\bigr) \Bigr]_{r,s} {\mathbb W}_{s,t}^{j,i}(\omega) \Bigr\vert   \\
&\leq c \, \biggl( \int_{0}^1 {{\mathbf d}_{1}} \Bigl( \overline \mu_{r;(r,s)}^{n,(\lambda)}(\omega), {\mathcal L}\bigl(X_{r;(r,s)}^{(\lambda)}\bigr)\Bigr) d\lambda \biggr) \, \overline w^{j,i}(s,t,\omega)^{2/p}   
\\
&\hspace{5pt} \times \biggl[ \big\vvvert \overline X^i(\omega) \big\vvvert_{[0,T],\overline w^i,p} + \big\vvvert \overline X^j(\omega) \big\vvvert_{[0,T],\overline w^j,p} + \biggl( \frac1n \sum_{k=1}^n \vvvert \overline X^k(\omega) \vvvert_{[0,T],\overline w^k,p}^2 \biggr)^{1/2} \biggr]   
\\
&\hspace{5pt}
\times \biggl[ \overline w^i(r,s,\omega)^{1/p} + \overline w^j(r,s,\omega)^{1/p} + \biggl( \frac1n \sum_{k=1}^n
\overline w^k(r,s,\omega)^{2/p} \biggr)^{1/2} \biggr]    \\
&\hspace{5pt}+ c \, \Bigl\vert {\mathcal S}^{i,j,n}_{r,s} \bigl(\omega, \vert \overline X_{r,s}^\bullet(\omega) \vert \bigr) 
\Bigr\vert
\, \overline w^{j,i}(s,t,\omega)^{2/p},
\end{split}
\end{equation*}
where
\begin{equation*}
\begin{split}
\Bigl\langle {\mathcal S}^{i,j,n}_{r,s} 
\bigl(\cdot,
\vert \overline X_{r,s}^\bullet(\cdot) \vert
\bigr)
\Bigr\rangle_{3\varrho/2}
&\leq c \,n^{-1/2} \,\big\langle \vvvert  X \vvvert_{[0,T],w,p} w(r,s,\cdot)^{1/p} \big\rangle_{3\varrho/2}
\\
&\leq c \,n^{-1/2}\, \big\langle \vvvert  X \vvvert_{[0,T],w,p} \rangle_{{\chi}\varrho} \, \big\langle w(r,t,\cdot) \rangle_{{\chi}\varrho}^{1/p}.
\end{split}
\end{equation*}
Observing that $\langle \overline w^{j,i}(s,t,\cdot)^{2/p} \rangle_{{\chi} \varrho} \leq  \llangle  w^+(r,t,\cdot,\cdot) \rrangle_{{\chi} \varrho}^{2/p}$ -- this is the rationale for introducing $w^+$, and taking expectation,  we get
\begin{equation*}
\begin{split}
&\Bigl\langle  
\Bigl[ 
G_{\mu} \bigl(\overline X^i(\cdot),\overline \mu^n(\cdot) \bigr)\bigl(\overline X^j(\cdot)\bigr) 
-
G_{\mu} \bigl( \overline X^i(\cdot) , {\mathcal L}(X) \bigr)
\bigl(\overline X^j(\cdot)\bigr)
\Bigr]_{r,s} {\mathbb W}_{s,t}^{j,i}(\omega) 
\Bigr\rangle_{\varrho}
\\
&\leq 
c \ \biggl( \int_{0}^1 \Bigl\langle 
{{\mathbf d}_{1}} \Bigl( 
\overline \mu_{r;(r,s)}^{n,(\lambda)}(\cdot), 
{\mathcal L}\bigl(X_{r;(r,s)}^{(\lambda)}\bigr)
\Bigr)
\Bigr\rangle_{{\chi}\varrho}
d\lambda \biggr) \,
\bigl\langle 
\vvvert X(\cdot) \vvvert_{[0,T],w,p}
\bigr\rangle_{{\chi} \varrho} \,
\\ 
&\hspace{150pt} \times \bigl\llangle 
w^+(r,t,\cdot,\cdot)
\bigr\rrangle_{{\chi} \varrho}^{3/p}
\\
&\hspace{10pt}
+ c \,n^{-1/2}\, \big\langle \vvvert  X(\cdot) \vvvert_{[0,T],w,p} \bigr\rangle_{{\chi}\varrho} \, \bigl\llangle w^+(r,t,\cdot,\cdot) \bigr\rrangle_{{\chi}\varrho}^{3/p}.
\end{split}
\end{equation*}
Taking the mean over $j$, we obtain as upper bound for the third term in the right-hand side of \eqref{eq:summand:analysis} the quantity
\begin{equation*}
\begin{split}
&\left\langle  
\frac1n \sum_{j=1}^n
\Bigl[ 
G_{\mu} \bigl(\overline X^i(\cdot),\overline \mu^n(\cdot) \bigr)\bigl(\overline X^j(\cdot)\bigr) 
-
G_{\mu} \bigl( \overline X^i(\cdot) , {\mathcal L}(X) \bigr)
\bigl(\overline X^j(\cdot)\bigr)
\Bigr]_{r,s} {\mathbb W}_{s,t}^{j,i}(\omega) 
\right\rangle_{\varrho}
\\
&\leq 
c \ \biggl( \int_{0}^1 \Bigl\langle 
{{\mathbf d}_{1}} \Bigl( 
\overline \mu_{r;(r,s)}^{n,(\lambda)}(\cdot), 
{\mathcal L}\bigl(X_{r;(r,s)}^{(\lambda)}\bigr)
\Bigr)
\Bigr\rangle_{ {\chi}\varrho}
d\lambda \biggr) \,
\bigl\langle 
\vvvert X(\cdot) \vvvert_{[0,T],w,p}
\bigr\rangle_{{\chi} \varrho} 
\\
&\hspace{150pt} \times
\bigl\llangle 
w^+(r,t,\cdot,\cdot)
\bigr\rrangle_{{\chi} \varrho}^{3/p}
\\
&\hspace{10pt}
+ c \,n^{-1/2}\, \big\langle \vvvert  X(\cdot) \vvvert_{[0,T],w,p} \bigr\rangle_{{\chi}\varrho} \, \bigl\llangle w^+(r,t,\cdot,\cdot) \bigr\rrangle_{{\chi}\varrho}^{3/p}.
\end{split}
\end{equation*}
By Lemma \ref{lem:W1:cv:empirical}, we get the same bound as in the first step. 

\medskip

\textbf{\textsf{Step 3.}} We now turn to the last term in the right-hand side of \eqref{eq:summand:analysis}. It reads as the empirical mean of 
$n$ random variables, $n-1$ of which are conditionally centred and conditionally independent given
the realization of the paths $(\overline{X}^i,W^i,{\mathbb W}^i)$, namely
\begin{equation*}
\frac1n \sum_{j=1}^n 
\bigl[G_{\mu} \bigl(\overline X^i(\omega),{\mathcal L}(X) \bigr) \bigl(\overline X^j(\omega)\bigr)
\bigr]_{r,s}
{\mathbb W}^{j,i}_{s,t}(\omega)
-
{\mathbb E}
\bigl[ \delta_{\mu} \overline F_{r,s}^{i}(\omega,\cdot) {\mathbb W}_{s,t}^{i,\indep}(\cdot,\omega)
\bigr].
\end{equation*}
We can handle 
the above term by
invoking Rosenthal's inequality once again (in a conditional form). To do so, it suffices to compute the $\LL^\varrho$ norm of 
$
\bigl[G_{\mu} \bigl(\overline X^i(\omega),{\mathcal L}(X) \bigr) \bigl(\overline X^j(\omega)\bigr) \bigr]_{r,s} {\mathbb W}^{j,i}_{s,t}(\omega).$
Doing as before (see \eqref{eq:Gx:xi:mun}), it is less than $c \,\big\langle \vvvert \overline X(\cdot) \big\vvvert_{[0,T],w,p} \bigr\rangle_{{\chi}\varrho}\,\big\llangle  w^+(r,t,\cdot,\cdot) \big\rrangle_{{\chi}\varrho}^{3/p}$. So, 
\begin{equation*}
\begin{split}
&\left\langle \frac1n \sum_{j=1}^n \Bigl[G_{\mu} \bigl(\overline X^i(\omega),{\mathcal L}(X) \bigr) \bigl(\overline X^j(\omega)\bigr) \Bigr]_{r,s} {\mathbb W}^{j,i}_{s,t}(\omega)  -  {\mathbb E} \Bigl[ \delta_{\mu}\overline F_{r,s}^{i}(\omega,\cdot) {\mathbb W}_{s,t}^{i,\indep}(\cdot,\omega) \Bigr] \right\rangle_{\varrho}   \\
&\leq c \,n^{-1/2}\, \big\langle \vvvert  X(\cdot) \vvvert_{[0,T],w,p} \bigr\rangle_{{\chi}\varrho} \, 
\bigl\llangle w^+(r,t,\cdot,\cdot) \bigr\rrangle_{{\chi}\varrho}^{3/p},
\end{split}
\end{equation*}
which suffices to conclude. 

\medskip

\textbf{\textsf{Step 4.}} We now handle the remainders in \eqref{eq:summand:analysis}. By expanding \eqref{eq:R:Fin:RbarFi} and by using similar notations for the remainders in the expansion of each $\big(\overline X^j\big)_{j=1,\cdots,n}$, we have
{(see for instance the proof of \cite[Proposition 3.5]{BCD1} and in particular \cite[(3.9)]{BCD1})}
\begin{equation}
\label{eq:expansion:RFin}
\begin{split}
&R^{F^{i,n}}_{r,s}(\omega)
= \partial_{x}\textrm{F} \Bigl( \overline X_{r}^{i}(\omega),
\overline \mu_{r}^{n}(\omega)\Bigr) R_{r,s}^{\overline X^i}(\omega)
\\
&\hspace{-2pt}
+ \frac1n \sum_{j=1}^n 
D_{\mu}\textrm{F} \Bigl( \overline X_{r}^{i}(\omega), \overline \mu_{r}^{n}(\omega) \Bigr) \bigl( \overline X_{r}^j(\omega)\bigr) R_{r,s}^{\overline X^j}(\omega)   \\
&\hspace{-2pt}
+ \int_{0}^1 
\Bigl[ \partial_{x} \textrm{F} \Bigl( \overline X_{r;(r,s)}^{i,(\lambda)}(\omega),
\overline \mu_{r;(r,s)}^{n,(\lambda)}(\omega)
\Bigr) 
\hspace{-2pt}
-
\partial_{x} \textrm{F} \Bigl( \overline X_{r}^{i}(\omega),
\overline \mu_{r}^{n}(\omega)
\Bigr)  \Bigr] \, \overline X_{r,s}^i(\omega) \, d \lambda
\\
&\hspace{-2pt}
+ \frac1n \sum_{j=1}^n \int_{0}^1 \left[ D_{\mu} \textrm{F} \Bigl( \overline X_{r;(r,s)}^{i,(\lambda)}(\omega), \overline \mu_{r;(r,s)}^{n,(\lambda)} \Bigr)\Bigl( \overline X_{r;(r,s)}^{j,(\lambda)}(\omega) \Bigr)
 \right. \\
&\hspace{80pt}- \left.D_{\mu} \textrm{F} \Bigl( \overline X_{r}^{i}(\omega), \overline \mu_{r}^{n} \Bigr)\bigl(\overline X_{r}^{j}(\omega) \bigr)  \right] \, \overline X_{r,s}^j(\omega) \,d \lambda.
\end{split}
\end{equation}
Expanding $R^{\overline F^i}_{r,s}(\omega)$ in a similar way, we have to investigate four difference terms in order to estimate the difference $ R^{F^{i,n}}_{r,s}(\omega) - R^{\overline F^i}_{r,s}(\omega)$. The first difference term corresponds to the first term in the right-hand side of \eqref{eq:expansion:RFin}
\begin{equation*}
\begin{split}
&\left\vert \Bigl[ \partial_{x} {\textrm{F}}\Bigl(\overline X_{r}^i(\omega), \overline\mu^n_{r}(\omega) \Bigr)  -  \partial_{x} {\textrm{F}}\Bigl(\overline X_{r}^i(\omega), {\mathcal L}(X_{r}) \Bigr) \Bigr] \, R_{r,s}^{\overline X^i }(\omega) \right\vert   \\
&\leq c \,  {{\mathbf d}_{1}}\Bigl( \overline \mu^n_{r}(\omega),{\mathcal L}(X_{r}) \Bigr) 
\,
\vvvert \overline X^i (\cdot) \big\vvvert_{[0,T],\overline w^i,p}
\, \overline w^i (r,s,\omega)^{2/p}.
\end{split}
\end{equation*}
Then, we must recall that, in the first line of the right-hand side in \eqref{eq:summand:analysis}, the difference $ R^{F^{i,n}}_{r,s}(\omega) - R^{\overline F^i}_{r,s}(\omega)$ is multiplied by $W_{s,t}^i(\omega)$, which is less than $\overline w^i(s,t,\omega)^{1/p}$. In other words, we must multiply both sides in the above inequality by $\overline w^i(r,t,\omega)^{1/p}$. By Cauchy Schwarz inequality, the ${\mathbb L}^{\varrho}$ norm of the resulting bound is less than  
$$
c \, \big\langle  {{\mathbf d}_{1}}(\overline{\mu}^n_{r}(\cdot),{\mathcal L}(X_{r}) \big\rangle_{{\chi}\varrho} \, \big\langle \vvvert  X (\cdot) \big\vvvert_{[0,T],w,p} \big\rangle_{{\chi} \varrho} \, \big\langle w(r,t,\cdot) \big\rangle^{3/p}_{{\chi} \varrho}.
$$ 

\smallskip

The second difference term
that we have to handle
 corresponds to the second term in the right-hand side of \eqref{eq:expansion:RFin}. With an obvious definition for $R^{X}(\cdot)$, it reads
\begin{equation*}
\begin{split}
&\biggl\vert \frac1n \sum_{j=1}^n 
D_{\mu} \textrm{F}\Bigl( \overline X_{r}^i(\omega),\overline \mu^n_{r}(\omega) 
\Bigr) \bigl( \overline X_{r}^j(\omega) \bigr)
R_{r,s}^{\overline X^j}(\omega)
\\
&\hspace{15pt} - \Bigl\langle
D_{\mu} \textrm{F}\Bigl(\overline X_{r}^i(\omega), 
{\mathcal L}(X_{r}) 
\Bigr) \bigl( X_{r}(\cdot) \bigr)
R_{r,s}^{X}(\cdot)
\Bigr\rangle \biggr\vert.
\end{split}
\end{equation*}
Proceeding exactly as in the first step, the latter is bounded by 
\begin{equation*}
\begin{split}
&c \, {{\mathbf d}_{1}} \Bigl( \overline \mu^{n}_{r}(\omega) , {\mathcal L}( X_{r} )
\Bigr) \, 
 \left( \frac1n \sum_{j=1}^n
\bigl\vert 
R_{r,s}^{\overline{X}^j}
(\omega)
\bigr\vert  \right)
+ c 
\Bigl\vert {\mathcal S}^{i,n}_{r,s} \Bigl(\omega, \bigl\vert R_{r,s}^{\overline X^{\bullet}}(\omega) \bigr\vert \Bigr)
\Bigr\vert,
\end{split}
\end{equation*}
where ${\mathcal S}^{i,n}_{r,s} \bigl(\omega, \vert R_{r,s}^{\overline X^{\bullet}}(\omega) \vert \big)$
 is the $n$-empirical mean of 
$n$  variables that 
are dominated by $\big(\vert  R_{r,s}^{\overline X^j}(\omega) \vert
+\big\langle R_{r,s}^{X}(\cdot) \big\rangle_{1}
\big)_{j=1,\cdots,n}$ 
and $n-1$ of which are conditionally centred and conditionally independent given
the realization of the path $(\overline{X}^i,W^i,{\mathbb W}^i)$. Hence, the $\LL^\varrho$ norm of the right-hand side, after multiplication as before by $\overline w^i(s,t,\omega)^{1/p}$, is less than
$$
c \, \Bigl( \Bigl\langle  {{\mathbf d}_{1}}\Bigl(\overline{\mu}^n_{r}(\cdot),{\mathcal L}(X_{r})\Bigr) \Bigr\rangle_{{\chi} \varrho}
+ n^{-1/2}
\Bigr)
\,
 \bigl\langle \vvvert  X(\cdot) \vvvert_{[0,T],w,p} \bigr\rangle_{{\chi} \varrho} \, \bigl\langle w(r,t,\cdot) 
 \bigr\rangle_{{\chi} \varrho}^{3/p}.
 $$

\smallskip

As for the third term in the right-hand side of \eqref{eq:expansion:RFin}, it fits, up to the additional factor 
$\overline X_{r,s}^i(\omega)$
and 
 {for each value of $\lambda$}, the term studied in the first step. So we get
 as an upper bound for its $\LL^\varrho$ norm, after multiplication by $\overline w^i(s,t,\omega)^{1/p}$, 
 the quantity
\begin{equation*}
\begin{split}
&c \, \biggl( \int_{0}^1 \int_{0}^1 \Bigl\langle 
{{\mathbf d}_{1}} \Bigl( 
\overline \mu_{r;(r,s)}^{n,(\lambda \lambda')}(\cdot), 
{\mathcal L}\bigl(X_{r;(r,s)}^{(\lambda \lambda')}\bigr)
\Bigr)
\Bigr\rangle_{ {\chi} \varrho} \, 
d \lambda d\lambda'
\biggr) 
\\
&\hspace{150pt} \times \bigl\langle 
\vvvert X(\cdot) \vvvert_{[0,T],w,p}
\bigr\rangle_{{\chi} \varrho}^2 \, 
\bigl\langle 
w(r,t,\cdot)
\bigr\rangle_{{\chi} \varrho}^{3/p}
\\
&\hspace{10pt}
+ c \,n^{-1/2} \,\big\langle \vvvert  X(\cdot) \vvvert_{[0,T],w,p} \bigr\rangle_{{\chi}\varrho}^2 \, \bigl\langle w(r,t,\cdot) \bigr\rangle_{{\chi}\varrho}^{3/p}.
\end{split}
\end{equation*} 
Following Step 2, we get exactly a similar bound for the fourth term in the right-hand side of \eqref{eq:expansion:RFin}.
\end{proof}

\subsection{About Law of Large Numbers}
\label{subse:lln}

\begin{lem}
\label{lem:W1:cv:empirical}
There exists a real $\rho_{d} \geq 1$ such that, for any $\rho \geq \rho_{d}$ and any probability measure $\mu$ on $\RR^d$ satisfying $M_{\rho}(\mu) := \big(\int_{\RR^d} \vert x \vert^{\rho} \mu(dx)\big)^{1/\rho} < \infty$, it holds
\begin{equation*}
{\mathbb E} \left[ {\mathbf d}_{1}\bigl(\mu^n(\cdot),\mu\bigr)^{\rho/3} \right]^{3/\rho}
\leq c_{\rho,d} \, M_{\rho}(\mu)\, \varsigma_{n},
\end{equation*}
for a constant $c_{\rho,d}$, only depending on $\rho$ and $d$,  where $\mu^n(\cdot)$ is the empirical distribution of $n$ independent identically distributed random variables
and with $\varsigma_{n}$ as in the introduction of Subsection 
\ref{AppendixAuxiliary}.
\end{lem}

\begin{proof}
Without any loss of generality, we can assume that $M_{\rho}(\mu)=1$, see the argument in \cite[Chapter 5, Theorem 5.8]{CarmonaDelarue_book_I}. Then Theorem 2 in \cite{FournierGuillin} gives us the following results. For $d \geq 3$, we have
(for $\rho \geq 2$)
\begin{equation*}
{\mathbb P} \Bigl( {\mathbf d}_{1}\bigl(\mu^n(\cdot),\mu\bigr) \geq A \varsigma_{n} \Bigr) \leq C \exp \Bigl( -c n  \varsigma_{n}^{d}A^{d} \Bigr) + C n \big(n A \varsigma_{n}\big)^{-\rho/2},
\end{equation*}
in which case the result easily follows. When $d=1$, we have 
\begin{equation*}
{\mathbb P} \Bigl( {\mathbf d}_{1}\bigl(\mu^n(\cdot),\mu\bigr) \geq A \varsigma_{n} \Bigr) \leq C \exp \Big( -c n  \varsigma_{n}^{2} A^{2}\Big) + C n \big(n A \varsigma_{n}\big)^{-\rho/2},
\end{equation*}
and the result follows as well by our choice of $\varsigma_{n}$. Finally, when $d=2$,
\begin{equation*}
{\mathbb P} \Bigl( {\mathbf d}_{1}\bigl(\mu^n(\cdot),\mu\bigr) \geq A \varsigma_{n} \Bigr) \leq C \exp \left( - \frac{c n  \varsigma_{n}^2 A^{2}}{(\ln(2+A^{-1} \varsigma_{n}^{-1}))^2} \right) + C  n (n A \varsigma_{n})^{-\rho/2}.
\end{equation*}
Assuming without any loss of generality that $A \geq 1$, we have 
$$
\ln(2+A^{-1} \varsigma_{n}^{-1}) \leq \ln(2+\varsigma_{n}^{-1}) = \ln (1+ 2 \varsigma_{n}) - \ln(\varsigma_{n}), 
$$
which is less than $-2\ln(\varsigma_{n})$ for $n$ large enough. Given our choice of $\varsigma_{n}$, we have $-\ln(\varsigma_{n}) = \ln(n)/2 - \ln(\ln(1+n))$, which is less than $\ln(n)/2$. Hence, modifying the value of the constant $c$, we get, for $A \geq 1$ and for $n$ large enough, independently of the value of $A$, we get the bound
\begin{equation*}
{\mathbb P} \Bigl( {\mathbf d}_{1}\bigl(\mu^n(\cdot),\mu\bigr) \geq A \varsigma_{n} \Bigr) \leq C \exp \left( - \frac{c  A^{2} \ln(1+n)^2}{\ln(n)^2} \right)+ C n (n A \varsigma_{n})^{-\rho/2},
\end{equation*}
which suffices to complete the proof. 
\end{proof}

{\begin{lem}
\label{lem:LLN:2nd order}
Let $(X_{n})_{n \geq 1}$ be a collection of independent and identically distributed random variables with values in a Polish space
$S$ and 
let $f$ be a real-valued Borel function on $S^2$ such that ${\mathbb E}[ \vert f(X_{1},X_{2}) \vert]$ and 
${\mathbb E}[ \vert f(X_{1},X_{1}) \vert]$ are both finite. Then, with probability 1, 
\begin{equation*}
\lim_{n \rightarrow \infty}
\frac1{n^2} \sum_{i,j=1}^n f(X_{i},X_{j}) = {\mathbb E} \bigl[ f(X_{1},X_{2}) \bigr]. 
\end{equation*}
\end{lem}

\begin{proof}
By the standard version of the law of large numbers, it suffices to prove that, with probability 1,
\begin{equation*}
\lim_{n \rightarrow \infty}
\frac1{n^2} \sum_{i,j =1, i \not = j}^n f(X_{i},X_{j}) = {\mathbb E} \bigl[ f(X_{1},X_{2}) \bigr]. 
\end{equation*} 
Letting $S_{n} = \sum_{i,j=1, i \not =  j}^n f(X_{i},X_{j})$, for $n \geq 1$, we then define the 
$\sigma$-field 
${\mathcal G}_{n} = \sigma(S_{k},k \geq n)$. By independence of the variables $(X_{k})_{k \geq 1}$, 
we have, for any $(i,j) \in \{1,\cdots,n\}^2$ with $i \not =j$, 
${\mathbb E}[ f(X_{i},X_{j}) \vert {\mathcal G}_{n}]
= {\mathbb E}[ f(X_{i},X_{j}) \vert S_{n}]$. By exchangeability, this is also equal to 
${\mathbb E}[f(X_{1},X_{2}) \vert S_{n}]$.  We get
\begin{equation*}
{\mathbb E}[f(X_{1},X_{2}) \vert {\mathcal G}_{n}] = \frac{1}{n^2-n} \sum_{i,j=1,i \not = j}^n 
{\mathbb E} \bigl[
f(X_{i},X_{j}) \vert S_{n}
\bigr]  = \frac{S_{n}}{n^2-n}. 
\end{equation*}
By L\'evy's downward theorem and by Kolmogorov's zero-one law, the left-hand side converges almost surely 
to ${\mathbb E}[f(X_{1},X_{2})]$. 
\end{proof}
}


\begin{thebibliography}{AAA}

\bibitem{AKR1}
Albeverio, S., Kondratiev, Y.G., R\"ockner, M.,
\newblock Differential geometry of Poisson spaces. 
\newblock {\em C. R. Acad. Sci. Paris S\'er. I Math.}, {\bf 323}(10):1129--1134, 1996.  

\bibitem{AKR2}
Albeverio, S., Kondratiev, Y.G., R\"ockner, M.,
\newblock
Analysis and geometry on configuration spaces. 
\newblock 
{\em J. Funct. Anal.}, {\bf 154}(2):444--500, 1998.

\bibitem{AmbrosioGigliSavare}
Ambrosio, L., Gigli, N., Savar\'e, G.,
\newblock {\em Gradient Flows in Metric Spaces and the Wasserstein Spaces of Probability Measures}, 
\newblock Lectures in Mathematics, ETH Z\"urich, Birkh\"auser, 2005.

\bibitem{BailleulRMI}
Bailleul, I.,
\newblock Flows driven by rough paths.
\newblock {\em Revista Mat. Iberoamericana}, {\bf 31}(3):901--934, 2015.

\bibitem{BCD1}
Bailleul, I., Catellier, R., Delarue, F., 
\newblock {Solving mean field rough differential equations}.
\newblock {\em Electron. J. Probab.}, {\bf 25}, pno. 21, 51 pp., 2020.

\bibitem{BCDArxiv}
Bailleul, I., Catellier, R., Delarue, F.,
\newblock Mean field rough differential equations.
\newblock arXiv:1802.05882, 2018.

\bibitem{BailleulRiedel}
Bailleul, I. and Riedel, S.,
\newblock Rough flows.
\newblock {\em J. Math. Soc. Japan}, {\bf 71}(3):915--978, 2019.

\bibitem{BarbuRockner}
Barbu, V. and R\"ockner, M.,
\newblock From nonlinear Fokker-Planck equations to solutions of distribution dependent SDE. 
\newblock {\em arXiv:1808.10706}, 2018.

\bibitem{billing}
Billingsley, P.,
\newblock Convergence of probability measures.
\newblock {\em Second Edition. Wiley Series in Probability and Statistics: Probability and Statistics}, {John Wiley \& Sons Inc.}, 1999.

\bibitem{BudhirajaDupuisFischer}
Budhiraja, A. and Dupuis, P. and Fischer, M.,
\newblock Large deviation properties of weakly interacting processes via weak convergence methods.
\newblock {\em Ann. Probab.}, {\bf 40}(1):74--102, 2012.

\bibitem{BudhirajaWu}
Budhiraja, A. and Wu, R.,
\newblock Moderate Deviation Principles for Weakly Interacting Particle Systems.
\newblock {\em Probab. Th. Rel. Fields}, {\bf 168}(2-4):721--771, 2017.

\bibitem{LionsCardialiaguet}
Cardaliaguet, P.,
\newblock Notes on mean field games.
\newblock https://www.ceremade.dauphine.fr/ cardaliaguet/MFG20130420.pdf, 2013.

\bibitem{CarmonaDelarue}
Carmona, R. and Delarue, F.,
\newblock Forward-backward Stochastic Differential Equations and Controlled McKean Vlasov Dynamics. 
\newblock {\em Ann. Probab.}, {\bf 44}(6):3740--3803, 2016.

\bibitem{CarmonaDelarue_book_I}
Carmona, R. and Delarue, F.
\newblock {\em Probabilistic Theory of Mean Field Games: vol. I, Mean Field FBSDEs, Control, and Games.}
\newblock Probability Theory and Stochastic Modelling, Springer Verlag, 2018.	

\bibitem{CarmonaDelarue_book_II}
Carmona, R. and Delarue, F.
\newblock {\em Probabilistic Theory of Mean Field Games: vol. II, Mean Field Games with Common Noise and Master Equations.}
\newblock Probability Theory and Stochastic Modelling, Springer Verlag, 2018.	

\bibitem{CassdosReisSalkeld}
Cass, T. and dos Reis, G. and Slakeld, W.,
\newblock The Support of McKean Vlasov Equations driven by Brownian Motion.
\newblock {\em arXiv:1911.01992}, 2019.

\bibitem{CassLittererLyons}
Cass, T., Litterer, C. and Lyons, T.,
\newblock Integrability and tail estimates for Gaussian rough differential equations.
\newblock {\em Ann. Probab.}, {\bf 41}(4):3026--3050, 2013.

\bibitem{CassLyons}
Cass, T. and Lyons, T.,
\newblock Evolving communities with individual preferences.
\newblock {\em Proc. Lond. Math. Soc. (3)}, {\bf 110}(1):83--107, 2015.

\bibitem{CoghiDeuschelFrizMaurelli}
Coghi, M., Deuschel, J.-D., Friz, P. and Maurelli, M.,
\newblock Pathwise McKean-Vlasov theory.
\newblock arXiv:1812.11773, 2018.   

\bibitem{CoghiGess}
Coghi, M. and Gess, B.,
\newblock Stochastic nonlinear Fokker-Planck equations.
\newblock {\em Nonlinear Anal.}, {\bf 187}:259--278, 2019.

\bibitem{CoghiFlandoli}
Coghi, M. and Flandoli, F.,
\newblock Propagation of chaos for interacting particles subject to environmental noise.
\newblock {\em Ann. Appl. Probab.}, {\bf 26}(3):1407--1442, 2016.

\bibitem{CoghiNilssen}
Coghi, M. and Nilssen, T.,
\newblock Rough nonlocal diffusions.
\newblock {\em arXiv:1905.07270}, 2019.

\bibitem{DawsonGartner}
Dawson, D. and G\"artner, J.,
\newblock Large deviations from the McKean-Vtasov limit for weakly interacting diffusions.
\newblock {\em Stochastics}, {\bf 20}:247--308, 1987.

\bibitem{DawsonVaillancourt}
Dawson, D. and Vaillancourt, J.,
\newblock Stochastic McKean-Vlasov equations.
\newblock {\em Nonlinear Diff. Eq. Appl.}, {\bf 2},199--229, 1995.

\bibitem{FrizMaurelli}
Deuschel, J.-D., Friz, P., Maurelli, M. and Slowik, M.,
\newblock The enhanced Sanov theorem and propagation of chaos.
\newblock {\em Stoch. Proc. Appl.}, {\bf 128}(7):2228--2269, 2018.

\bibitem{FournierGuillin}
Fournier, N. and Guillin, A.,
\newblock On the rate of convergence in the {W}asserstein distance of the empirical measure.
\newblock
{\em Probab. Theory Related Fields},
{\bf 162}:707-738, 2015.

\bibitem{FrizHairer}
Friz, P. and Hairer, M.,
\newblock A course on rough paths, with an introduction to regularity structures.
\newblock {\em Universitext, Springer}, 2014.

\bibitem{FrizVictoirBook}
Friz, P. and Victoir, N.,
\newblock Multidimensional stochastic processes as rough paths.
\newblock {\em Cambridge studies in advanced Mathematics}, {\bf 120}, 2010.

\bibitem{FrizVictoirGaussian}
Friz, P. and Victoir, N.,
\newblock Differential equations driven by Gaussian signals.
\newblock {\em Ann. Inst. H. Poincar\'e, Probab. Statist}, {\bf 46}(2):369--413, 2010.

\bibitem{GangboTudorascu}
Gangbo, W. and Tudorascu, A.,
\newblock On differentiability in the Wasserstein space and well-posedness for Hamilton--Jacobi equations.
\newblock {\em  J. Math. Pures Appl.}, {\bf 125}(9):119--174, 2019. 

\bibitem{Gartner}
G\"artner, J.,
\newblock On the McKean-Vlasov limit for interacting diffusions.
\newblock {\em Math. Nachr.}, {\bf 137}:197--248, 1988.

\bibitem{Gubinelli}
Gubinelli, M.,
\newblock Controlling rough paths.
\newblock {\em J. Funct. Anal.}, {\bf 216}(1):86--140, 2004.

\bibitem{JabinWang}
Jabin, P.-E. and Wang, Z.,
\newblock Quantitative estimates of propagation of chaos for stochastic systems with $W^{-1,\infty}$ kernels.
\newblock {\em Invent. Math.}, {\bf 214}(1):523--591, 2018.

\bibitem{JKO}
Jordan, R., Kinderlehrer, D., Otto, F.,
\newblock
The variational formulation of the Fokker-Planck equation.
\newblock
{\em SIAM J. Math. Anal.}, {\bf 29}(1), 1--17, 1998.  


\bibitem{jourdain-meleard}
Jourdain, B., and M\'el\'eard, S., 
\newblock Propagation of chaos and fluctuations for a moderate model with smooth initial data,
\newblock{\em Ann. I.H.P. (B)}, {\bf 34}(6):727--766, 1998.

\bibitem{Kac}
Kac, M.,
\newblock Foundations of kinetic theory.
\newblock {\em Third Berkeley Symp. on Math. Stat. and Probab.}, {\bf 3}:171--197, 1956. 

\bibitem{KolokoltsovTroeva}
Kolokoltsov, V.N., and Troeva, M.,
\newblock On the mean field games with common noise and the McKean-Vlasov SPDEs.
\newblock arXiv:1506.04594, 2015.

\bibitem{Lions}
Lions, P.-L.,
\newblock Th\'eorie des jeux \`a champs moyen et applications.
\newblock {\em Lectures at the {C}oll\`ege de {F}rance}. http://www.college-de-france.fr/default/EN/all/equ\_der/cours\_et\_seminaires.htm", 2007-2008

\bibitem{Lott}
Lott, J.,
\newblock Some geometric calculations on Wasserstein space. 
\newblock {\em Comm. Math. Phys.}, {\bf 277}(2):423--437, 2008. 

\bibitem{Lyons98}
Lyons, T.,
\newblock Differential equations driven by rough paths.
\newblock {\em Rev. Mat. Iberoamericana}, {\bf 14}(2):215--310, 1998.

\bibitem{McKean}
McKean, H.P.,
\newblock A class of markov processes associated with nonlinear parabolic equations.
\newblock {\em Prov. Nat. Acad. Sci.}, {\bf 56}:1907--1911, 1966.

\bibitem{meleard}
M{\'e}l{\'e}ard, S., \emph{Asymptotic behaviour of some interacting particle
  systems; {M}c{K}ean-{V}lasov and {B}oltzmann models}, Probabilistic models
  for nonlinear partial differential equations ({M}ontecatini {T}erme, 1995),
  Lecture Notes in Math., vol. 1627, Springer, 1996, pp.~42--95.
  
  
\bibitem{ORS}
Overbeck, L., R\"ockner, M., Schmuland, B.,
\newblock An analytic approach to Fleming-Viot processes with interactive selection. 
\newblock {\em Ann. Probab.}, \textbf{23}(1):1--36, 1995.

  
\bibitem{Otto}
 Otto, F.,
\newblock
The geometry of dissipative evolution equations: the porous medium equation.
\newblock
{\em Comm. Partial Differential Equations}, {\bf 26}:101--174, 2001 

\bibitem{RenWang}
Ren, P., Wang, F.-Y.,
\newblock 
Derivative formulas in measure on
Riemannian manifold.
\newblock arXiv:1908.03711, 2019. 


\bibitem{rosenthal1972}
Rosenthal, H.P.,
\newblock \emph{On the span in $L^p$ of sequences of independent random variables. II.}
\newblock Proceedings of the Sixth Berkeley Symposium on Mathematical Statistics and Probability, Volume 2: Probability Theory,
University of California Press, 1972, pp. 149--167.


\bibitem{Sznitman}
Sznitman, A.-S.,
\newblock Topics in propagation of chaos.
\newblock {\em Lect. Notes Math.}, {\bf 1464}, 1991.

\bibitem{Tanaka}
Tanaka, H.,
\newblock Probabilistic treatment of the Boltzman equation of Maxwellian molecules.
\newblock {\em Probab. Th. Rel. Fields}, {\bf 46}:67--105, 1978.

\bibitem{TanakaTrick}
Tanaka, H.,
\newblock Limit theorems for certain diffusion processes with interaction. 
\newblock {\em Stochastic analysis (Katata/Kyoto, 1982)}:469--488, North-Holland Math. Library, 32, North-Holland, Amsterdam, 1984.   

\bibitem{Villani}
Villani, C.,
\newblock {\em Topics in Optimal Transportation}. 
\newblock Graduate Studies in Mathematics 58, Providence, RI: Amer. Math. Soc., 2003. 

\bibitem{WuZhang}
Wu, C. and Zhang, J.,
\newblock An elementary proof for the structure of derivatives in probability measures.
\newblock arXiv:1705.08046, 2017.

\end{thebibliography}
\end{document}